\RequirePackage{fix-cm}
\documentclass[10pt, two side, letterpaper, final]{amsart}
\usepackage[palatino]{global-custom-200}

\newcommand{\Vdg}{\widetilde{V}^\circ_{Ne+df,g}(a,b)}
\newcommand {\punctured}[1]{#1^\times}
\newcommand{\basepoint}{\star}
\newcommand{\Gbar}{\overline{G}}
\newcommand{\gbar}{\overline{g}}
\newcommand{\Dbar}{\overline{D}}
\newcommand{\fbar}{\overline{f}}

\renewcommand{\tilde}{\widetilde}

\newcommand{\size}[1]{{|{#1}|}}

\theoremstyle{plain}
\newtheorem{finalreduction}{Final Reduction}
\Crefname{finalreduction}{Final Reduction}{Final Reductions}

\theoremstyle{definition}
\newtheorem{inductionhypothesis}{Induction Hypothesis}
\Crefname{inductionhypothesis}{Induction Hypothesis}{Induction Hypotheses}

\title[Covers of elliptic curves and curves in $E\times \Pbb^1$]
{The Hurwitz space of covers of an elliptic curve $E$ \\ and the Severi variety of curves in $E \times \Pbb^1$}
\author[G. Bujokas]{Gabriel Bujokas}
\address{Harvard Mathematics Department \\
1 Oxford Street, Cambridge, MA}
\email{gbujokas@math.harvard.edu}
\date{\today}

\begin{document}
\begin{abstract}
We describe the hyperplane sections of the Severi variety of curves in $E \times \Pbb^1$ in a similar fashion to Caporaso--Harris' seminal work \cite{caporaso_counting_1998}. From this description we \emph{almost} 
get a recursive formula for the Severi degrees---we get the terms, but not the coefficients. 

As an application, we determine the components of the Hurwitz space of simply branched covers of a genus one curve. In return, we use this characterization to describe the components of the Severi variety of curves in 
$E \times \Pbb^1$, in a restricted range of degrees.
\end{abstract}
\maketitle 
\enlargethispage{\baselineskip}

\microtypesetup{protrusion=false}
\tableofcontents
\microtypesetup{protrusion=true}


\section{Introduction} 
\label{sec:introduction}
We will study the geometry of two variants of classical parameter spaces.
The first is the Hurwitz space $\Hc_{d,g}(C)$ of genus $g$, degree $d$
simply branched covers of a fixed curve 
$C$. The classical case is when $C$ is the projective line $\Pbb^1$, and goes as far back
as to Clebsch in 1873. Much less in known in the case $C$ has positive genus. In 1987, 
Gabai and Kazez \cite{gabai_classification_1987} 
described the irreducible components of $\Hc_{d,g}(C)$ for arbitrary smooth curve $C$, 
through topological methods. One of our main results is an algebraic proof of
their description in the case $C=E$ is a genus one curve.

The second space we consider is the Severi variety $V_{\Lc,g}$ of geometric genus $g$ 
integral curves of class $\Lc$ on a smooth surface $S$. The classical case is when
$S=\Pbb^2$, and goes back to Severi himself, who gave an incomplete argument for its irreducibility. 
The gap was filled only decades later by
Harris \cite{harris_severi_1986}, and the techniques introduced there led
to a series of discoveries in both qualitative (as in irreducibility statements)
and quantitative (as in curve counting) aspects of the geometry of
Severi varieties on
various surfaces $S$. The method has been particularly successful in the
case the surface $S$ is rational. Here we will apply the same techniques, 
but in the case where $S=E \times \Pbb^1$, for a connected
smooth genus one curve $E$. Most of the results go through, but with many new interesting phenomena. In particular, we will show that the generic integral curve in $E \times \Pbb^1$ is nodal, 
that the Severi variety has the expected dimension and
we will describe the components of generic hyperplane sections in
terms of \emph{generalized Severi varieties}, much like in 
Caporaso--Harris \cite{caporaso_counting_1998}. We will also
describe the irreducible components of the Severi variety in a restricted range of genus $g$ and class $\Lc$.

Perhaps the main new feature of these two variants, the Hurwitz space 
$\Hc_{d,g}(E)$ and the Severi variety of curves in $E \times  \Pbb^1$, 
is that the study of one aids in the study of the other. The description of the
components of hyperplane section of the Severi variety is the key input
to determine the components of Hurwitz space. Conversely, we can leverage 
the knowledge of the components of Hurwitz space to determine the components of the Severi variety, at least in some range.

We start with a quick but self-contained introduction to the guiding questions
and to a few known results in the study of both Severi varieties and Hurwitz
spaces. This exposition is meant only as a motivation for the results and techniques
that will occupy the main text, and it is by no means a complete survey of either field.

\subsection{Hurwitz spaces} 
\label{sub:hurwitz_spaces}
For a fixed smooth curve $C$, we define the Hurwitz space $\Hc_{d,g}(C)$ as the
scheme parametrizing degree $d$ covers (that is, finite flat maps) 
${f\from X \to C}$ that 
are simply branched, and whose source $X$ is a smooth connected curve of genus $g$.

There is a branch morphism $br \from \Hc_{d,g}(C) \to (C^b-\Delta)/S_b$ sending
a cover $X \to C$ to its set of $b = 2g-2-d(2g(C)-2)$ branch points. By the
Riemann existence theorem, this is an étale map.
In particular $\Hc_{d,g}(C)$ is smooth of dimension $b$.
\footnote{Historically, Hurwitz used the Riemann existence theorem to put 
a complex manifold structure on $\Hc_{d,g}(\Pbb^1)$, which was studied previously
by Clebsch as a parameter space for topological covers of the sphere. For a modern
treatment of the construction of the Hurwitz space, and in
particular how to endow it with a scheme structure, see Fulton's
\cite{fulton_hurwitz_1969}.}

Next, we  ask what are the connected components of $\Hc_{d,g}(C)$---that is, which covers
can be deformed into each other. As the Hurwitz space is smooth, every connected component is 
also irreducible. The famous 
theorem of Clebsch \cite{clebsch_zur_1873} and Lüroth answers this question when
$C$ is rational.
\begin{theorem}[Clebsch, Lüroth]
The Hurwitz space $\Hc_{d,g}(\Pbb^1)$ is connected.
\end{theorem}

As the forgetful map $\Hc_{d,g}(\Pbb^1) \to \Mc_g$ is dominant for large $d$, 
this implies that $\Mc_g$
is irreducible! This argument, first indicated by Klein, and then refined by 
Hurwitz and Severi, was the first known proof
of the irreducibility of $\Mc_g$.

For higher genus targets, the Hurwitz space $\Hc_{d,g}(C)$
is usually not irreducible. For example, 
for each cover $X \to C$, consider the 
maximal unramified subcover $\tilde{C} \to C$ it factors through as $X \to \tilde{C} \to C$. 
The space of unramified covers is discrete (since the dimension 
of the Hurwitz space is the number of branch points). Therefore, to each cover
point in $\Hc_{d,g}(C)$, we associated a discrete invariant, which will 
separate connected components.

\begin{definition}
We call a cover $f \from X \to C$ \emph{primitive} if it does not factor
as $X \to \widetilde{C} \to C$, with $\widetilde{C} \to C$
a non-trivial unramified map.
Let $\Hc^0_{d,g}(C)$ be the open and closed subscheme of
$\Hc_{d,g}(C)$ parametrizing primitive covers.
\end{definition}

We can reexpress this discrete invariant in a more topological vein. 
Over the complex numbers, for each cover $f \from X \to C$, pick base points and
consider the image of the fundamental group 
$f_* \pi_1(X,\basepoint) \subset \pi_1(C,\basepoint)$.
Its conjugacy class is a discrete invariant of the cover $X \to C$.
By the theory of covering spaces, the conjugacy class carries the same information
as the maximal unramified subcover $\tilde{C} \to C$. 
In particular, being primitive is
equivalent to having a surjective pushforward map on $\pi_1$.
For a simply branched cover, being primitive is also equivalent to having 
full monodromy---see \Cref{proposition:full_monodromy}.

Note that any cover $f\from X \to C$ can be factored as
$X\to \widetilde{C} \to C$ where $X \to \widetilde{C}$ is primitive, and
$\widetilde{C} \to C$ is unramified. Hence, if
we understand the geometry of each $\Hc^{0}_{d,g}(C)$ for every $d$
and $C$, we can recompose the whole Hurwitz space $\Hc_{d,g}(C)$---it
is just a disjoint union of $\Hc_{\tilde{d},g}(\widetilde{C})$
over unramified maps $\widetilde{C} \to C$ of degree $d/\tilde{d}$.

The main theorem of \cite{gabai_classification_1987} is that this is the only discrete invariant.
\begin{theorem}[Gabai--Kazez]
\label{conjecture:hurwitz-connected}
The Hurwitz space $\Hc^{0}_{d,g}(C)$ of primitive covers
is connected (and hence irreducible) for any 
smooth curve $C$ and $g>g(C)$.
\end{theorem}

The traditional approach to irreducibility of Hurwitz space is group theoretic. 
Riemann existence theorem
describes the fibers of the branch morphism in terms of the monodromy permutations associated to the
branch points and the generators of $\pi_1(C)$. To establish connectedness of
Hurwitz space, we need to understand how these permutations change as we move the branch points around, 
and to show that under a suitable sequence of moves, we can get from any permutation data to another.
This is the approach in Clebsch's original proof of irreducibility of $\Hc_{d,g}(\Pbb^1)$. This method
can be extended to covers of higher genus curves, as long as there are at least $2d-2$ branch points. This was first discovered in Hamilton's unpublished thesis, and rediscovered by Berstein--Edmonds and Graber--Harris--Starr \cite{graber_note_2002}.

A weaker version of \Cref{conjecture:hurwitz-connected} was conjectured by Berstein and Edmonds in 
\cite{berstein_classification_1984} as the \emph{uniqueness conjecture}. There, besides moving the branch points around, they also allowed performing Dehn twists on
the target curve. They got more mileage out of that, but
to establish the full range of \Cref{conjecture:hurwitz-connected} some new key constructions had to be introduced by 
Gabai and Kazez in \cite{gabai_classification_1987}. For a flavor of their techniques, see 
\Cref{lemma:gabai-kazez-style}.

The language and methods in \cite{gabai_classification_1987} are very much topological,
which may have created somewhat of a gap between the algebraic geometry and topology literature. As a matter of fact, I have not seen \Cref{conjecture:hurwitz-connected} stated in the form above in the algebraic geometry literature,
 even though there has been a recent spur of interest in these Hurwitz spaces of covers of higher genus curves, as in 
 for example \cite{graber_note_2002,kanev_hurwitz_2004,kanev_hurwitz_2005,kanev_irreducibility_2005,chen_covers_2010,kanev_unirationality_2013}.


Now remains the question of whether algebraic methods can be used to establish \Cref{conjecture:hurwitz-connected}.
We highlight two results in this direction. In \cite{kanev_unirationality_2013}, Kanev deals with the case where the degree of the cover is at most five (with some constraints on the number of branch points), using Casnati--Ekedahl--Miranda type of constructions. In \cite{kani_hurwitz_2003}, Kani gives a beautiful proof that \Cref{conjecture:hurwitz-connected}
holds when the genus of the target is one, and the genus of the source is two. As a matter of fact, he shows that 
$\Hc^0_{d,2}(E)$ is isomorphic to $E \times X(d)$, where $X(d)$ is the modular curve.

Here we will give an algebraic proof of \Cref{conjecture:hurwitz-connected} when the genus of the target $C$ is one.
\begin{theorem}
\label{theorem:hurwitz_irreducible}
For a smooth genus one curve $E$ and $g>1$, the Hurwitz space $\Hc^0_{d,g}(E)$ of
primitive covers is irreducible.
\end{theorem}

Going back to the entire Hurwitz space $\Hc_{d,g}(E)$, we get the following corollary.
\begin{corollary}
The components of $\Hc_{d,g}(E)$ are in bijection with the set of isogenies
$\tilde{E} \to E$ of degree diving $d$, but not equal to $d$. In particular,
there are 
	$\sum _{d\neq \tilde{d}\divides d} \sigma(\tilde{d})$
components, where $\sigma(n)$ is the sum of the positive divisors of $n$.
\end{corollary}

An interesting---and perhaps more natural---variant is to let the target curve $E$ vary as well.
\begin{definition}
\label{definition:global-hurwitz}
Let $\Hc_{d,g,h}$ be the Hurwitz space parametrizing degree $d$ simply branched covers from a pointed genus $g$ to a genus $h$ curve.
More formally, a map $T \to \Hc_{d,g,h}$ corresponds to a smooth family of 1-pointed genus $h$ curves
 $(\pi:\Dc \to T, \,  \sigma:T \to \Dc)$, a finite flat map
$f:\Cc \to \Dc$ such that $f_t:C_t \to D_t$ is simply branched for every $t \in T$, and a section $\tau:T \to \Cc$
such that $\tau \circ f = \sigma$.
Let $\Hc^0_{d,g,h} \subset \Hc_{d,g,h}$ be the subscheme parametrizing primitive covers.
\end{definition}
\begin{remark}
The extra data of the marked point is irrelevant for irreducibility questions, since the forgetful map
to the unpointed variant will have irreducible one-dimensional fibers.
\end{remark}

There is a map from the Hurwitz space $\Hc^0_{d,g,1}$ to $\Mc_{1,1}$. The fiber over $(E,p)$ is the universal curve 
over $\Hc^0_{d,g}(E)$.
\Cref{theorem:hurwitz_irreducible} says the fibers are all irreducible of the same dimension, which implies
the following corollary.
\begin{corollary}
The Hurwitz space $\Hc^0_{d,g,1}$ is irreducible. 
\end{corollary}
\begin{remark}
In topological language, the irreducibility of $\Hc^0_{d,g,h}$ is equivalent to the uniqueness conjecture of
 \cite{berstein_classification_1984}. Berstein and Edmonds do establish it in the case $h=1$, corresponding to the corollary above, by a careful combinatorial study of the monodromy data.
\end{remark}

To determine the components of $\Hc_{d,g,1}$, the key
discrete invariant of a cover $f:C \to E$. One of them is the isomorphism type of the cokernel $C(f)$ of the pushforward
map $f_*$ on $\pi_1$ (or $H_1$). It is an abelian group generated by at most two elements. The only possibilities are 
$C_{d_1} \oplus C_{d_2}$, with $d_1 \divides d_2$ and $C_n$ is the cyclic group of order $n$.

We can re-encode this invariant in a more algebraic way. Let $\tilde{d}$ be the degree 
of the maximal isogeny $\tilde{E} \to E$ that $f$ factors through. These isogenies in general will get swapped by the monodromy of varying the target $E$, but there is a special family of isogenies that are fixed by the monodromy: the multiplication by $m$ maps. We can remember as well the maximal $m$ for which $C \to E$ factors through the multiplication by $m$ map. We constructed a map that sends a cover $f:C \to E$ to a pair $(\tilde{d},m)$. This carries the
same data as the isomorphism type of the cokernel group $C(f)$ discussed above. Indeed, 
\[
C(f)=C_m \oplus C_{\tfrac{\tilde{d}}{m}}	
\]

In \Cref{sub:varying_the_moduli_of_E}, we will prove the following.
\begin{corollary}
\label{corollary:components-of-Hdg1}
The components of $\Hc_{d,g,1}$ are separated by the isomorphism type of $C(f)$. In particular, there is
a bijection between the components of $\Hc_{d,g,1}$ and pairs $(\tilde{d},m)$ where $m^2 \divides \tilde{d} \divides d$.
\end{corollary}
This answers one of the questions raised in \cite{chen_covers_2010}.

Our argument for \Cref{theorem:hurwitz_irreducible}
 will be analogous to Fulton's proof of the irreducibility of $\Mc_g$
given in the appendix of \cite{harris_kodaira_1982}. For more details and motivation
for both our argument and Fulton's proof, please turn to
 \Cref{sub:our-strategy}. But, in a nutshell,
the proof of \Cref{theorem:hurwitz_irreducible} goes like this. First consider the 
compactification of $\Hc_{d,g}(E)$ in the Kontsevich space $\overline{\Mc}_{g}(E,d)$
of stable maps $X \to E$ of degree $d$. Let us restrict ourselves to the locus 
where the map $X \to E$ has no contracted components---call this
$\widetilde{\Hc}_{d,g}(E)$.

Using induction on $d$ and $g$, we will show that 
only one irreducible component of $\widetilde{\Hc}^0_{d,g}(E)$ may contain covers with singular source.
This amounts to identifying the boundary divisors, and use induction to show that each of them is irreducible,
and that their union is connected. Since $\widetilde{\Hc}^0_{d,g}(E)$ is smooth, only one component may contain
a point in these boundary divisors.

To finish it off, we want to show that every component of $\widetilde{\Hc}_{d,g}(E)$ does contain
 a cover with singular source. This is surprisingly hard to prove directly, 
 and here is where our study of Severi varieties on
$E \times \Pbb^1$ enters.  

\subsection{Severi varieties} 
\label{sub:severi_varieties}
Given a line bundle $\Lc$ on a surface $S$, we define the \emph{Severi variety}
$V_{\Lc,g} \subset |\Lc|$
as the locus of integral curves of geometric genus $g$. The guiding
questions to answer are the following.
\begin{itemize}
	\item What is the dimension of the Severi variety?
	\item Do the generic points correspond to nodal curves?
	\item What are its irreducible components?
	\item What is its degree?
\end{itemize}

There has been a lot of work in these problems, and variants of them.
Historically the problem originated by Severi's study of nodal planar curves,
and nowadays we can answer these questions for most rational surfaces $S$
and line bundles $\Lc$. There have been impressive results in the case $S$
is a K3 surface, where the questions are much more delicate, and even some 
results in the much harder case where the surface $S$ is of general type.
For a general introduction, we refer the reader to Joachim Koch notes'
of Ciro Ciliberto's lectures \cite{ciliberto_nodal_1999}.

We will focus on the technique introduced by Harris in \cite{harris_severi_1986},
and perfected by  Caporaso--Harris, Vakil, Tyomkin, and recently Shoval--Shustin 
\cite{caporaso_counting_1998,vakil_counting_2000,tyomkin_severi_2007,shoval_gromov-witten_2013}.
The originating question of this project was how 
much of these techniques extend over to non-rational surfaces, such as 
$E \times \Pbb^1$.

Before we explain the technique in the context of $E \times \Pbb^1$,
let us deal with the fact its Picard group is not discrete,
as opposed to the rational surface case. The Picard group of $E \times \Pbb^1$ is generated by 
the pullback of $\Oc_{\Pbb^1}(1)$ (which we will denote by $e$)
 and pullbacks from $E$. Therefore, it is one dimensional, just as
$\Pic(E)$.

This potentially could cause minor issues, 
since fixing the numerical class of a divisor
does not pin it down.
We avoid this issue entirely by noting that the automorphism group of $E$ acts
transitively on the set of divisors of same (non-zero) degree. Hence, the choice of line bundle
is irrelevant to the questions above.



Based on this observation, we will denote $V_{df+Ne,g} \subset |\Lc|$
as the Severi variety of integral genus $g$ curves of class $\Lc$, for some 
divisor $\Lc$ with homology class $df+Ne$. When needed, we will be more 
explicit about $\Lc$ and write $V_{\Lc,g}$.

To answer the first two questions (the dimension and the singularities of the 
generic point), we will use the now standard deformation theory approach introduced
in \cite{caporaso_counting_1998}. We conclude that the Severi variety $V_{df+Ne,g}$
has the expected dimension $2d+g-1$ (with a single exception), and that the generic point does 
correspond to a nodal curve. We will do this in \Cref{sec:deformation_theory}.
The exception corresponds to the case $d=0, N=g=1$: when the Severi variety parametrizes 
the fibers of the projection  $E\times \Pbb^1 \to \Pbb^1$. Here the dimension is one, while the
expected dimension works out to zero.

Note that, surprisingly, the expected dimension does not depend on $N$. This will have 
many interesting effects.

Our approach to the degree and irreducibility questions 
is directly analogous to Caporaso--Harris as well. We will fix an elliptic fiber 
$E_0 \subset E \times \Pbb^1$ of the projection to $\Pbb^1$. For a generic point $p \in E_0$, consider the hyperplane 
$H_p \subset |\Lc|$ of curves passing through $p$. We want to describe the 
intersection 
\begin{equation}
\label{equation:intersection}
V_{df+Ne,g} \cap H_{p_1} \cap H_{p_2} \cap \ldots \cap H_{p_a}	
\end{equation}
for generic points $p_1,\ldots, p_a \in E_0$. We do this one hyperplane at a time.
Eventually, the elliptic fiber $E_0$ is forced to split off, 
and most of our work goes into describing the residual 
curve.

To understand this intersection, we need a better compactification of the Severi variety $V_{df+Ne,g}$.
 We could do this by simply taking the closure $\closure{V_{df+Ne,g}} \subset |\Lc|$.
Unfortunately this closure is not normal, so it will be more convenient for us to work with its normalization
$\widetilde{V}_{df+Ne,g}$.

Here is a component we expect to see in the intersection \eqref{equation:intersection}. 
\begin{definition}
\label{definition:gen-severi-variety-1}
For $p_1,\ldots, p_a \in E_0$, and a line bundle $L \in \Pic^b E_0$ such that $a+b=d$,
we define a \emph{generalized Severi variety} 
$\widetilde{V}_{\Lc,g}(a,b)[p_1,\ldots,p_a, L]$ as the normalization of the closure of the 
locus of integral curves $C \subset E \times \Pbb^1$ of geometric genus $g$ and class $\Lc$, such that
the intersection of $C$ with $E_0$ is transverse, composed of $a+b=d$ points, $a$ of which are the
points $p_1,\ldots, p_a$, and the sum of the remaining $b$ points has class $L$. 
When no confusion arises,
we will write only $\widetilde{V}_{df+Ne,g}(a,b)$, where the corresponding 
homology class of $\Lc$ is $df+Ne$, and leave the data of $\Lc$, $L$ and the 
points $p_i$ implicit.
\end{definition}


We will show that when $b >0$ the Severi variety $\widetilde{V}_{df+Ne,g}(a,b)$ has expected dimension and that the
generic point corresponds to a nodal curve.
However, these are not the only components we will see in the intersections with $H_{p_i}$. 
Whenever the fiber $E_0$ is split off generically, the residual curve might not be transverse to $E_0$. We will
need some notation to keep track of the tangency profile of the residual curve with 
$E_0$. Here our notation will differ slightly from the standard notation in Caporaso--Harris.

\begin{definition}
\label{definition:tangency-profile}
A \emph{(tangency) profile} $\alpha$ is an ordered sequence of positive integers $(\alpha_1,\alpha_2,\ldots, \alpha_k)$.
A sum of profiles $\alpha^1 + \alpha^2$ is the profile obtained by concatenating the sequences. The
\emph{size} of the profile $\size{\alpha}$ is $k$, the number of entries in the sequence. The 
\emph{multiplicity} of a profile $m(\alpha)$ is the sum $\sum_{i=1}^k \alpha_i$.

A subprofile $\alpha' \subset \alpha$ is a subsequence of $(\alpha_1,\ldots, \alpha_k)$. The
complement $\alpha - \alpha'$ is the profile given by the complementary subsequence.

For consistency, we will say that an integer $k$ corresponds to the profile 
$1^k=(1,1,\ldots, 1)$ with 
$k$ entries.
\end{definition}

We will use tangency profiles to book-keep tangencies in the most naive 
way---having a tangency
profile $\alpha$ means there is a point with tangency of order $\alpha_1$, another with order $\alpha_2$, and so on
\footnote{Caporaso--Harris in \cite{caporaso_counting_1998} re-encoded these tangency conditions in a
different sequence $\overline{\alpha}$, such that $\overline{\alpha}_k$ is the number of $\alpha_i$ that are
equal to $k$. Their convention definitely has computational advantages, but it would also make the
notation more loaded later on. As we will not derive any recursive degree formulas in this
article, we keep the more naive notation for the sake of notational simplicity.}.

\begin{definition}
\label{definition:gen-severi-variety-2}
For points $p_1,\ldots, p_a$, and a tangency profile $\beta$ such that $a+m(\beta)=d$, 
and $L \in \Pic^{m(\beta)} E_0$, define $\widetilde{V}_{df+Ne,g}(a,\beta)[p_1,\ldots,p_a, L]$ 
as the normalization of the closure of the locus
of integral curves $C \subset E \times \Pbb^1$ of geometric genus $g$, class $\Lc$, 
whose intersection with $E_0$ is 
\[
 	E_0 \cap C = \sum_{i=1}^a p_i + \sum_{j=1}^{\size{\beta}} \beta_j q_j 
 \] 
and the class of $\sum_{j=1}^{\size{\beta}} \beta_j q_j$ is $L$.
\end{definition}
\begin{remark}
We have kept this line bundle $L$ in our
notation for the Severi variety $V_{df+Ne,g}(a,\beta)[p_1,\ldots,p_a,L]$,
 but of course we can
recover it from the class $\Lc$ of the curve in $E\times \Pbb^1$ and the points
$p_i$, as we have 
\[
	\Lc|_{E_0} = \sum_{i=1}^a p_i +L	
\]
This choice will be justified in \Cref{definition:gen-severi-variety-3}.
\end{remark}

Note that under the convention that an integer $b$ corresponds to the profile $1^b$,
\Cref{definition:gen-severi-variety-1,definition:gen-severi-variety-2} agree.

Still, these are not all the possible varieties parametrizing
the residual curves. There are two more phenomena that can happen.
First, even though $V_{df+Ne,g}$ parametrizes integral curves,
there is no reason to expect that once a fiber $E_0$ splits off, the
residual curve is still irreducible. 

We could define $\widetilde{V}^\text{reduced}_{df+Ne,g}(a,\beta)$ as in 
\cref{definition:gen-severi-variety-2}, but imposing that the curves
are only \emph{reduced}, instead of integral. Now $g$ corresponds to
the arithmetic genus of the normalization. This should keep track of the
reducible residual curves as well. 

One would hope that the dimension count will still work for  $\widetilde{V}^\text{reduced}_{df+Ne,g}(a,\beta)$.
However, it does not, because of the exception mentioned earlier: the fibers of 
$E \times \Pbb^1 \to  \Pbb^1$ move in a one-dimensional family, while the dimension count 
predicts a zero dimensional space. Let us introduce a
notation for the components of $\widetilde{V}^\text{reduced}_{df+Ne,g}(a,\beta)$ which do
not contain generically any of these fibers.

\begin{definition}
\label{definition:severi-circ}
Let $\widetilde{V}^\circ_{df+Ne,g}(a,\beta) \subset \widetilde{V}^\text{reduced}_{df+Ne,g}(a,\beta)$ 
be the closure of the components whose generic point does not contain a fiber of the projection 
$E \times \Pbb^1 \to \Pbb^1$.
\end{definition}

\begin{remark}
Note that $\widetilde{V}^\circ_{df+Ne,g}(a,\beta)$ parametrizes curves whose (possibly disconnected) normalizations have arithmetic genus $g$. If the normalization is disconnected, there may be even components with genus bigger than 
$g$, as long as there are rational components to balance off. Of course, there could be at most $d$ of those, so no components of arbitrarily high genus ever arise in 
$\widetilde{V}^\circ_{df+Ne,g}(a,\beta)$.

Still, it makes sense, and will be needed, to consider 
$\widetilde{V}^\circ_{df+Ne,g}(a,\beta)$
for \emph{negative} values of
$g$. The curves thus parametrized always contain some number of fibers from $E\times \Pbb^1 \to E$.
\end{remark}

There is a second possible phenomenon, as seen in the 
following example. Say we are looking at the intersection of
$\widetilde{V}_{Ne+df,g}(a,b)[p_1,\ldots,p_a,L]$ with a hyperplane $H_p$. Let $W$
be a component of the intersection. It could be equal to a component of 
${\widetilde{V}_{Ne+df,g}(a+1,b-1)}$, meaning that one point
which was ``free to move'' became fixed at $p$. Or it could
be that the fiber $E_0$ was split off generically with multiplicity
$m$. Let $X$ be the residual curve. We will completely characterize it soon,
but for now let us consider just the case where the limit of all the $b$ points
land in $X$, meaning, the intersection of $X$ with $E_0$ contains the limit
of the $b$ points. Now, the linear class of the sum of these $b$ points
was fixed and equal to $L$---it still has to be so in the limit.
Moreover, the $d-b$ remaining points of intersection of $X$ with $E_0$
might not be all fixed. In this case, we will need to keep \emph{two}
tangency profiles around to keep track of the intersection of $X$
with $E_0$---one for the limit of the $b$ points in the generic fiber, whose class
sums up to $L$, and another for any new tangencies that might occur.
We accommodate for this with the following definition.
\begin{definition}
\label{definition:gen-severi-variety-3}
For points $p_1,\ldots, p_a$, and tangency profiles $\beta^1,\beta^2$ such that 
\[
 a+m(\beta^1)+m(\beta^2)=d	
\]
and line bundles $L_k \in \Pic^{m(\beta^k)} E_0$, define 
$\widetilde{V}^{\circ}_{df+Ne,g}(a,\beta^1,\beta^2)[p_1,\ldots,p_a, L_1,L_2]$ 
as the normalization of the closure of the locus
of reduced curves $C \subset E \times \Pbb^1$ (not containing any fiber of $E \times \Pbb^1 \to \Pbb^1$)
of geometric genus $g$, class $\Lc$, 
whose intersection with $E_0$ is 
\[
 	E_0 \cap C = \sum_{i=1}^a p_i + \sum_{j=1}^{\size{\beta^1}} \beta^1_j q^1_j +
 	\sum_{j=1}^{\size{\beta^2}} \beta^2_j q^2_j
 \] 
and the class of $\sum_{j=1}^{\size{\beta^k}} \beta^k_j q^k_j$ is $L_k$,
for $k=1,2$.
\end{definition}

We are now ready to state a complete characterization of the possible
components of $H_p \cap \widetilde{V}_{df+Ne,g}(a,b)$. The description is easier to
state if we look at the intersection with $\widetilde{V}^\circ_{df+Ne,g}(a,b)$, so
let us do so.
\begin{theorem}
\label{theorem:simple_hyperplane_section_of_severi}
For $g \geq 2$, and $p$ be a generic point in $E_0$, let 
\[
W \subset H_p\cap \widetilde{V}^\circ_{df+Ne,g}(a,b)[p_1,\ldots,p_a, L]
\]
be an irreducible component of the intersection. Then $W$ falls into one of the following three cases.

If the curve $E_0$ does not split off generically in $W$, then $W$ is a component of
\begin{equation}
\label{eq:typeI}
\widetilde{V}^\circ_{df+Ne,g}(a+1,b-1)[p_1,\ldots, p_a,p, L-p]
\end{equation}
This case can only happen if $b \geq 2$.

Otherwise, the curve $E_0$ splits off generically, say with multiplicity $m$.
We can view $W$ as parameter space for the residual curve, whose class is $\Lc-me$. 
Then there is a subsequence $(i_1,\ldots, i_{\bar{a}})$ of
$(1,2,\ldots, a)$ and a tangency profile $\tau \neq (1)$
such that if we let $\bar{g}=g - |\tau|$, then either:
\begin{itemize}
	\item  $W$ is a component of 
\begin{equation}
 \label{eq:typeII}
	\widetilde{V}^\circ_{df+(N-m)e,\bar{g}}(\bar{a},b-1+\tau)[p_{i_1},\ldots, p_{i_{\bar{a}}}, \bar{L}]
\end{equation}
where $\bar{a}=d-(b-1+m(\tau))\geq 0$ and $\bar{L} = \Lc|_{E_0} - \sum_{k=1}^{\bar{a}} p_{i_k}$,
\item or W is a component of
\begin{equation}
 \label{eq:typeIII}
	\widetilde{V}^\circ_{df+(N-m)e,\bar{g}}(\bar{a},b,\tau)[p_{i_1},\ldots, p_{i_{\bar{a}}}, L,\bar{L}]
\end{equation}
where $\bar{a}=d-(b+m(\tau))\geq 0$ and $\bar{L} = \Lc|_{E_0} - L -\sum_{k=1}^{\bar{a}} p_{i_k}$.
\end{itemize}
\end{theorem}
\begin{remark}
\Cref{theorem:simple_hyperplane_section_of_severi} gives
only a set-theoretic description of the boundary.
To have a complete scheme theoretic characterization, we would need to
compute the multiplicity in which each component appears. This computation
is planned to appear in an upcoming follow-up article.
\end{remark}

The main new feature here, compared with results in 
\cite{caporaso_counting_1998,vakil_counting_2000,shoval_gromov-witten_2013},
is that $E_0$ may be generically split off with 
multiplicity $m>1$---this will correspond in the stable reduction to getting an $m$-fold
unramified cover dominating $E_0$.

The next question is: what happens when we intersect $\widetilde{V}^\circ_{df+(N-m)e,g'}(a',b,\tau)$
with a hyperplane section $H_p$? We will answer this question in \Cref{theorem:hyperplane_section_of_severi},
and furthermore we can describe all the components that ever arise in the process of intersecting
with hyperplane sections $H_p$. To do so, however, we need to extend even further our
\Cref{definition:gen-severi-variety-3} of Severi varieties $\widetilde{V}^\circ_{df+Ne,g}(a,\beta^1,\beta^2)$
to allow for $a=1^a$ to be an arbitrary tangency profile, and to allow arbitrarily many $\beta^k$. 
The statement of the result becomes notationally more loaded, and we leave it to \Cref{ssub:the_general_case}.

The good news is that the proof of \Cref{theorem:simple_hyperplane_section_of_severi} already contains 
all the interesting challenges, and adapting it to the general case of \Cref{theorem:hyperplane_section_of_severi}
is just a matter of additional bookkeeping. We will then first prove the simpler 
\Cref{theorem:simple_hyperplane_section_of_severi} in \Cref{sec:hyperplane_sections}, and then 
translate its proof to the stronger \Cref{theorem:hyperplane_section_of_severi} in 
\Cref{ssub:the_general_case}.

The missing ingredient to get a recursive formula for the degree of $\widetilde{V}_{df+Ne,g}$ 
is to find out if all these components appear, and with what multiplicity. Thinking in terms of a recursive
formula for the Severi degrees, \Cref{theorem:hyperplane_section_of_severi} tells us what terms would appear in the formula, but does not tell us with what coefficients. The multiplicity computation
is planned to appear in an upcoming follow-up article.
As a starting
point for the recursion, we would need to know the degrees when $g=1$, which have
been already calculated by Cooper--Pandharipande in \cite{cooper_fock_2012}.

The technique used to prove \Cref{theorem:simple_hyperplane_section_of_severi} was 
 introduced in \cite{caporaso_counting_1998}. 
We will exploit the tension between two different points of views on Severi varieties: 
embedded curves versus parametrized curves. From the embedded curves perspective, we know
the dimension of the Severi variety $V^\circ_{df+Ne,g}(a,b)$, which 
gives us control on how many conditions we
may impose on the curve corresponding to the general point of the divisor $W$. From the parametric 
point of view, we can study the general point of $W$ by taking a general arc in $V^\circ_{df+Ne,g}(a,b)$
approaching it, and applying stable reduction to the family of normalization maps. The central
fiber is replaced by a nodal
curve, which allows us to concretely understand the combinatorial restrictions of having arithmetic
genus $g$. It turns out that the number of constraints we discover from the parametric
point of view exactly matches the number of conditions the embedded point of view allows.
This precise balance pins down exactly what are the possible limit components listed in
\Cref{theorem:hyperplane_section_of_severi}.

It is worth mentioning that although our exposition uses both points of view,
it is definitely possible to write it all in terms of the parametric one.
This is done, for example, in \cite{vakil_counting_2000}. The proofs 
usually get shorter, but at a slight cost of some additional technical details.
Our compromise of using both points of view is just a matter of taste---we
 want to make the implicit tension between the two points of view visible already at the level
of language.

There is a rational map  $\widetilde{V}_{df+Ne,g}(a,b) \rationalmap \overline{\Mc}_{g}(E,d)$,
sending integral curves $C \subset E \times \Pbb^1$ to the composition of 
the normalization $C^\nu \to C$ with the projection $C \subset E \times \Pbb^1 \to E$.
The rational map can be completed in codimension 2, and hence it is well defined on the generic points
of the components of the hyperplane section $H_p\cap \widetilde{V}_{df+Ne,g}(a,b)$.
Out of the proof of \Cref{theorem:simple_hyperplane_section_of_severi}, we also get a description of
this map, as follows.

\begin{theorem}
\label{theorem:simple_hyperplane_to_hurwitz}
Using the notation of \Cref{theorem:simple_hyperplane_section_of_severi},
 the image of the general point $[C] \in W$ under the rational map to $\overline{\Mc}_g(E,d)$
 is as follows.
\begin{itemize}
\item If $W$ is a component of $V_{df+Ne,g}(a+1,b-1)$, then
the general point corresponds to an integral curve $C \subset E \times \Pbb^1$, and 
its image in $\overline{\Mc}_{g}(E,d)$ is the composition of the normalization map 
$C^\nu \to C$ with the projection $C \subset E \times \Pbb^1 \to E$.

\item Otherwise the curve $E_0$ splits off generically with multiplicity $m$, 
and the residual curve is a generic point $[R]$ of a component of
$\widetilde{V}^\circ_{df+(N-m)e,\bar{g}}(\bar{a},b-1+\tau)$ or
of ${\widetilde{V}^\circ_{df+(N-m)e,\bar{g}}(\bar{a},b,\tau)}$.
Then its image in $\overline{\Mc}_{g}(E,d)$ is a stable map
$X \to E$, where $X$ is a connected nodal curve such that: 
\begin{itemize}
	\item the normalization of $X$ is the disjoint union of $\tilde{E}$ and
the normalization of $R$, where $\tilde{E} \to E_0$ is a degree $m$ unramified cover 
(which may be disconnected),
	\item and $X$ has $\card{\tau}$ nodes, all of them connecting components of $\tilde{E}$ 
to components of the normalization of $R$.
\end{itemize}
\end{itemize}
\end{theorem}

\Cref{theorem:simple_hyperplane_to_hurwitz} 
is the key ingredient in our proof of \Cref{theorem:hurwitz_irreducible},
the irreducibility of the Hurwitz space of primitive covers of elliptic curves. 
As briefly discussed in \Cref{sub:hurwitz_spaces}, we can reduce the irreducibility
statement to showing that every connected component of 
$\widetilde{\Hc}_{d,g}(E)$ contains a cover with singular source. 
And \Cref{theorem:simple_hyperplane_to_hurwitz} produces this cover for us!

After using the degeneration coming from the Severi variety world to prove 
\Cref{theorem:hurwitz_irreducible}, we come back and use the irreducibility of 
$\Hc^0_{d,g}(E)$ to prove a result about Severi varieties:
we determine the irreducible components of $V_{Ne+df,g}$ when $d\geq 2g-1$. Again,
their description is very much like in \Cref{sub:hurwitz_spaces}---the unique
discrete invariant will be the maximal unramified map through which the map from the
normalization of the embedded curve to $E$ factors.
 For more details, see \Cref{sec:components_of_the_severi_variety}.


\subsection{Guide to readers} 
\label{sub:guide_to_readers}
The sections are organized to be as independent of each other as possible.
The reader interested in a particular section should be able to read it directly, 
as long as one is willing to assume a couple of statements proven in other sections.
Here is a short summary of the sections' content and dependencies.
\begin{itemize}
	\item In \Cref{sec:deformation_theory}, we start our study of Severi varieties
	by determining its dimension, and showing that its generic point corresponds
	to a nodal curve. All results are based on a now standard theorem in deformation
	theory, of which many versions already exist in the literature. For completeness,
	we prove the precise formulation we use (\Cref{theorem:deftheory}) in 
	\Cref{sec:appendix-def-theory}.
	\item In \Cref{sec:hyperplane_sections}, we analyze the components of
	a hyperplane section of the Severi variety, and prove
	\Cref{theorem:simple_hyperplane_section_of_severi,theorem:simple_hyperplane_to_hurwitz}
	and some generalizations.
	The argument depends on the dimension bounds \Cref{proposition:deformation_bound} and \Cref{proposition:tail},
	both proven in \Cref{sec:deformation_theory}. 
	\item In \Cref{sec:monodromy_groups_of_covers}, we study the monodromy group
	of both primitive and non-primitive simply branched covers. This discussion does not depend
	on the previous sections.
	\item In \Cref{sec:irreduciblity_of_hurwitz_spaces}, we prove \Cref{theorem:hurwitz_irreducible},
	that is, the 
	Hurwitz space of primitive covers of an elliptic curve is irreducible. Most
	of the argument is independent of the previous sections, but at two key
	moments we invoke \Cref{lemma:fiber product} (proven in \Cref{sec:monodromy_groups_of_covers})
	and \Cref{theorem:simple_hyperplane_to_hurwitz} (proven in \Cref{sec:hyperplane_sections}).
	\item Finally, in \Cref{sec:components_of_the_severi_variety}, we use \Cref{theorem:hurwitz_irreducible}
	to describe the irreducible components of the Severi variety $V_{df+Ne,g}$ whenever $d \geq 2g-1$.
\end{itemize}

The results are valid over algebraically closed fields of characteristic zero. For simplicity though,
 we will work over the complex numbers throughout this article. The only places this choice shows up are in 
\Cref{sec:monodromy_groups_of_covers} 
and \Cref{ssub:proof_of_lemma_technical_isogeny}. In the former, the use of complex numbers may be bypassed by
 the machinery of étale fundamental groups, and in the latter by identifying isogenies with subgroups of the $E[N]$ instead of sublattices in the complex plane.

With the alterations above, the arguments of \cref{sec:monodromy_groups_of_covers,sec:irreduciblity_of_hurwitz_spaces} should go through in characteristic $p$ large enough compared to the degree of the cover. However, the key \cref{theorem:deftheory}---on which the whole Severi variety discussion is based on---is only known to be valid in characteristic zero.
This is the same challenge that shows up, for example, in 
\cite{harris_severi_1986,vakil_counting_2000,tyomkin_severi_2007}.

For the reader interested in Hurwitz spaces in characteristic $p>0$, 
probably adapting the results of Fulton's \cite{fulton_hurwitz_1969}
is a more likely route. There he defines the Hurwitz space of covers of $\Pbb^1$ as a scheme over $\Zbb$, and deduces 
irreducibility in characteristic $p$ from the result in characteristic zero.

\subsection{Acknowledgments} 
\label{sub:acknowledgements}

I would like to thank Joe Harris, who welcomed me into the breathtaking world of algebraic geometry, and guided me throughout this entire journey. The original questions that initiated this project arose in conversations with Anand Patel, who has been a great collaborator and friend in the past few years. I thank him again for the thoughtful comments on an earlier draft, along side with Dawei Chen, Anand Deopurkar and Ravi Vakil. I am very grateful to Tom Graber, Jason Starr and Vassil Kanev for introducing me to the relevant topology literature. Finally, a special thanks to Carolina Yamate, who kindly drew the 
(several) beautiful pictures that give life to this manuscript.



\section{Deformation Theory} 
\label{sec:deformation_theory}





We start our study of Severi varieties of curves in $E \times \Pbb^1$ by determining
its dimension, and describing the geometry of the generic curve.
The techniques used are the same as in 
\cite{caporaso_counting_1998,vakil_counting_2000,tyomkin_severi_2007}.
The key result is the following.

\begin{theorem}
\label{theorem:deftheory}
Let $\Sigma$ be a smooth projective surface, and fix a homology class 
$\tau \in H^2(\Sigma, \Cbb)$ and an integer $g$.
Let $W \subset {\Mc}_g(\Sigma,\tau)$ be a non-empty subvariety of Kontsevich space
of stable maps.
 Let $f \from C \to \Sigma$ be a general point of $W$. Assume that $f$ is birational
 onto its image.

Let $D \subset \Sigma$ be a curve, and $\Omega \subset D$ a finite subset.
Let $Q=f^{-1}(D - \Omega) \subset C$. Let $\card Q$ denote the cardinality of
the set $Q$ (that is, we are not counting multiplicity). Fix some integer $b$,
such that $b \geq \card{Q}$, and let $\gamma=-(K_\Sigma+D) \tau + b$.

If $\gamma \geq 1$, then
\[
\dim W \leq g-1+ \gamma
\]
Conversely, suppose that equality occurs.  Then, the more positive $\gamma$, the 
better behaved the map $f$ is:
\begin{itemize}
\item If $\gamma \geq 2$, then $\card{Q} =b$ and
$f$ is unramified away from $Q$.
\item If $\gamma \geq 3$, then $f$ is unramified everywhere, and the image $f(C)$ 
has at most ordinary multiple points, and is smooth along $f(Q)$.
\item If $\gamma \geq 4$, then the image of $f$ is nodal, and smooth along $f(Q)$.
\end{itemize}
\end{theorem}

\begin{remark}
Note that as $W \subset \Mc_g(\Sigma,\tau)$ and not its closure, the source
 $C$ of $f\from C \to \Sigma$ is smooth and irreducible. Later we will see
 what one can say about more general sources, in 
 \Cref{proposition:preliminary_deformation_bound}.
\end{remark}




\begin{remark}
We will apply \Cref{theorem:deftheory} only in the following setup. 
Fix a line bundle $\Lc$ on $\Sigma$, and let $W \subset |\Lc|$ be a non-empty
subvariety, and $f\from C \to \Sigma$ the normalization of the general point of $W$. Hence, for our applications, 
we could have bypassed the language of Kontsevich spaces if we wanted to. 
However, the result as stated above is sharper, because it allows for the line bundle $\Lc$ to vary continuously, which is important when
the Picard group of $\Sigma$ is not discrete.
\end{remark}

Setting $\Omega= \emptyset$, and $b=D \cdot \tau$, the 
first part is explicitly stated as Corollary 2.4 in \cite{caporaso_counting_1998}.
The other parts are also in \cite{caporaso_counting_1998}, but implicit in the proof of
Proposition 2.2. There are many similar results in the literature, such as
Theorem 3.1 in \cite{vakil_counting_2000}, and Theorem 1 and Lemma 2 in 
\cite{tyomkin_severi_2007}. Unfortunately, I could not locate a version of 
this result for the set up we are working on. However,
it is easy to adapt the proofs in any of these sources to the
exact statement above. For completeness, we do include a proof of \Cref{theorem:deftheory}, 
but leave it to \Cref{sec:appendix-def-theory}.

Let us apply the general \Cref{theorem:deftheory} to our study of curves in 
$E\times \Pbb^1$, with tangency
conditions along a fiber $E_0 \subset E \times \Pbb^1$. Here is a preliminary result. 

\begin{lemma}
\label{lemma:dimension_lower_bound}
The Severi variety $\widetilde{V}_{df+Ne,g}(\alpha,0)$ has dimension at least $d+g-1$.
\end{lemma}
\begin{proof}
We want to parametrize integral curves $C \subset E \times \Pbb^1$ of divisor class $\Lc$, geometric genus $g$, and whose
intersection with $E_0$ is fixed. We get our lower bound by a naive dimension count. Fixing the intersection with
$E_0$ takes $h^0(E_0,\Lc|_{E_0})=d$ conditions. And having geometric genus $g$ imposes $\delta=p_a(\Lc)-g$ conditions.
Indeed, if we knew our curve was nodal, then this count asks one condition for each node. We do not know the curve is
nodal yet, but in any case we know that an isolated planar singularity of 
delta-invariant $\delta$ 
imposes $\delta$ conditions (in the language of 
\cite{diaz_ideals_1988}, the \emph{equigeneric locus} $EG$ has codimension $\delta$ 
in the versal deformation space of the singularity).

In the worst case, all these conditions would be independent, and we get
\begin{align*}
	\dim \widetilde{V}_{df+Ne,g}(\alpha,0) &\geq h^0(E\times\Pbb^1, \Lc) -d -\delta \\
	&=d(N+1) -d - (Nd-d+1 -g)=d+g-1
\end{align*}
as we wanted to show.
\end{proof}

\begin{lemma}
\label{lemma:dimension-alpha}
For $d \geq 1$, the dimension of $\widetilde{V}_{df+Ne,g}(\alpha,0)[p_1,\ldots,p_k]$ is $d+g-1$, and the generic point corresponds to a nodal curve.
\end{lemma}
\begin{proof}
\Cref{lemma:dimension_lower_bound} tell us that
the dimension of $\widetilde{V}_{df+Ne,g}(\alpha,0)$ is at least $d+g-1$.
We only need to show that it is also at most that. We will do that by applying
\Cref{theorem:deftheory}. Set $\Sigma=E \times \Pbb^1$, $D=E_0$, 
and $\Omega=\set{p_1,\ldots, p_k}$.

Then $Q= f^{-1}(D-\Omega)=\emptyset$, and we may set $b=0$. We get
\[
	\gamma= -(K_{E \times \Pbb^1} +D) \cdot (df+Ne) +b = -(-2e +e)(df+Ne)+0 = d \geq 1
\]

Hence, by \Cref{theorem:deftheory}, we have
\[
	\dim \widetilde{V}_{df+Ne,g}(\alpha,0) \leq \gamma +g -1 =d+g-1
\]
By \Cref{lemma:dimension_lower_bound}, the opposite inequality holds as well, so
 we actually have equality.

Let us prove that the general point of $\widetilde{V}_{df+Ne,g}(\alpha,0)$ is nodal. By \Cref{theorem:deftheory},
as soon as $\gamma=d \geq 4$, we are good. Let us consider the remaining cases:

\begin{itemize}
	\item For $d=1$, the only class with integral curves is  $f$, which only has
	smooth elements.
	\item For $d=2$, \Cref{theorem:deftheory} tell us it is enough
	to rule out tacnodes and triple points. Triple points can't occur because the
	map $C \to \Pbb^1$ is two to one. We will rule out tacnodes by an ad hoc argument, 
	in \Cref{ssub:the_hyperelliptic_locus}.
	\item For $d=3$, we only have to rule out triple points, by \Cref{theorem:deftheory}.
	We will do this in \Cref{ssub:the_trigonal_locus}.
\end{itemize}
\end{proof}

Now we are ready to analyze the Severi variety $V_{Ne+df,g}(a,b)$.
\begin{theorem}
\label{theorem:dimension-of-severi}
The Severi variety $V_{Ne+df,g}(a,b)$ has dimension $d+g-2+b$, 
and the generic point corresponds to a nodal curve.
\end{theorem}
\begin{proof}
Consider the map $V_{Ne+df,g}(a,b) \to \Sym^{a+b}E$ sending a curve $C \subset E \times \Pbb^1$
to its $a+b$ intersection points with $E_0$. The fiber over $(p_1,\ldots,p_{a+b})$ is 
$V_{df+Ne,g}(a+b,0)$, which by \Cref{lemma:dimension-alpha} has dimension $d+g-1$ and 
the generic point is nodal. We only have to show that the
image of $V_{Ne+df,g}(a,b) \to \Sym^{a+b}E$ has dimension $b-1$. This is because $a$ of
the points are fixed, while the other $b$ are constrained to be in a fixed linear series
of degree $b$. Riemann--Roch tells us they vary in a $(b-1)$-dimensional series.
\end{proof}
\begin{remark}
One of the key difficulties in studying curves in $C \times \Pbb^1$, for curves
$C$ of genus at least two, is the lack of an analogue of \Cref{theorem:dimension-of-severi}.
Indeed, the techniques here stem out of \Cref{theorem:deftheory}, which works 
better the more negative the canonical bundle of the surface is.
\end{remark}

In the discussion above, we were studying integral curves in $E \times \Pbb^1$.
In our application, however, we will not know a priori that
the residual curves are integral, but we still want to have a bound on their dimension.
In the parametric approach, integral curves correspond to maps $f \from C \to E\times \Pbb^1$ 
with smooth connected source and which are birational onto their image.
Let us relax these conditions to just having a nodal source $C$, possibly disconnected,
and with an arbitrary map $f\from  C \to E\times \Pbb^1$. We ask for an upper bound
on the deformations of the map $f$. More precisely, assume the following set up.

\begin{hypothesis}
\label{hypothesis:def-theory-set-up}
Fix a line bundle $\Lc$ on $E \times  \Pbb^1$, and a sequence of positive
integers $\alpha_1, \ldots, \alpha_n$.
Let $\Cc \to V$ be a flat family of (possibly disconnected) semistable genus $g$ curves, equipped with a map 
$\phi\from \Cc \to E \times \Pbb^1$, and sections $\sigma_i\from V \to \Cc$, such that
\begin{itemize}
	\item For each $b \in V$, the pullback divisor $\phi^{-1}(E_0)$ in $C_b$ is equal to
	\[
		\phi_b^{-1}(E_0) = \alpha_1 \sigma_1(b)+ \alpha_2 \sigma_2(b)+ \ldots + \alpha_n \sigma_n(b)
	\]
	\item Class of $\phi(C_b) \subset E \times \Pbb^1$ is $\Lc$.
	\item The induced map $V \to |\Lc|$ has finite fibers 
	(equivalently, the image family $\phi(C_b) \subset E \times  \Pbb^1$ is nowhere 
	isotrivial).
\end{itemize}
\end{hypothesis}

We want to bound the dimension of $V$, and to give a description of the equality case. For sake of 
preliminary discussion, let us fix completely the intersection with $E_0$. That is, let us
say that the sections $\sigma_i(V)$ are contracted under $\phi$ down to points $p_i$,
for each $i=1, \ldots, n$.

We would hope that the bound of \Cref{lemma:dimension-alpha} would still hold. This is not the case.
For example, consider letting $C =C' \sqcup \widetilde{E}$ be a disjoint union, where $C' \to E \times \Pbb^1$ is the normalization of 
a point in $V_{(N-n)e+df,g}(\alpha,0)$, and 
$\widetilde{E} \to E\times \set{t}  \subset E \times \Pbb^1$ is a degree $n$ isogeny
landing in any fiber $E \times \set{t}$.
Then the map $C' \to E \times \Pbb^1$ varies in a $d+g-1$ dimensional family,
while $\widetilde{E} \to E \times \Pbb^1$ varies in one dimension (corresponding
to choosing the target fiber $E \times \set{t}$), for a total of a $d+g$
dimensional family. 

Fortunately, this is the only way to go beyond the $d+g-1$ bound.
Let a \emph{floating component} be a connected component of 
$C_b$ whose image under $\phi(C_b)$ has class multiple of $e$. We will show that if the family $\Cc \to V$ has no floating
components, then the dimension of $V$ is at most $d+g-1$.
Let us first describe the equality cases.

As we are allowing disconnected sources $C$, we could take an arbitrary disjoint union
$C =C_1 \sqcup C_2 \sqcup \ldots \sqcup C_k$, and let $C_i \to E \times \Pbb^1$
be a generic point of $V_{d_if+N_ie, g_i}(\alpha^i,0)$. We require that no $C_i$ is a floating component (i.e. we exclude the $d_i=0$ case), and moreover that
\[
	\sum d_i =d, \;\; \sum N_i =N, \;\; \sum \left(g_i-1\right) =g-1, \;\; 
	\sum \alpha^i = \alpha
\]
Then the map $C \to E \times \Pbb^1$ will vary in a 
\[
	\sum d_i +g_i -1 = d+g-1
\]
dimensional family as well.

We already have a notation to parametrize such curves. Recall \cref{definition:severi-circ}: 
the Severi variety $\widetilde{V}^\circ_{df+Ne,g}(\alpha,0)$ is
the normalization of the closure of the locus of reduced curves in $E\times \Pbb^1$, of some specified 
class $\Lc$, some fixed intersection with $E_0$,
and whose normalization has arithmetic genus $g$, which contains no fibers of the projection
$E \times \Pbb^1 \to \Pbb^1$ (that is, no floating components). Compare to
$\widetilde{V}_{df+Ne,g}(\alpha,0)$, which further imposes that the curves are irreducible.
In general, when we want to allow  disconnected sources, but avoid floating components, we will
decorate our Severi varieties with a $\circ$ superscript.

The variety 
$\widetilde{V}^\circ_{df+Ne,g}(\alpha,0)$ has dimension $d+g-1$, which we will soon prove to be maximal.
One could hope that these are all equality cases. Unfortunately,
there is one more circumstance under which equality is achieved.
 Take a curve $C \to E \times \Pbb^1$ in 
$\widetilde{V}^\circ_{df+(N-N')e,g-k}(\alpha,0)$,
 choose $k$ points $q_1,\ldots, q_k \in C$, and attach isogenies
$E^i \to E$ at each of these points, such that the sum
of the degrees of the $k$ isogenies is $N'$. Let 
\[
\overline{C}= C \cup_{q_1}E^1 \cup \ldots \cup_{q_k}E^k \to E\times \Pbb^1	
\] 
be this new curve, with its map to $E \times  \Pbb^1$ (see \cref{figure:elliptic tails}).
Then it also varies in 
\[
	k+ \dim \widetilde{V}^\circ_{df+(N-N')e,g-k}(\alpha,0) = k+d+g-k -1 = d+g-1
\]
dimensions. Call the curves $E^i$ \emph{elliptic tails}, and set the space of all such maps
 $\overline{C}\to E\times \Pbb^1$ as
$\widetilde{V}^{\circ,\text{tails}}_{df+Ne,g}(\alpha,0)$. Generally, we will use the superscript
\emph{tails} to denote that we allow elliptic tails in the source.

\begin{figure}
\centering
\includegraphics{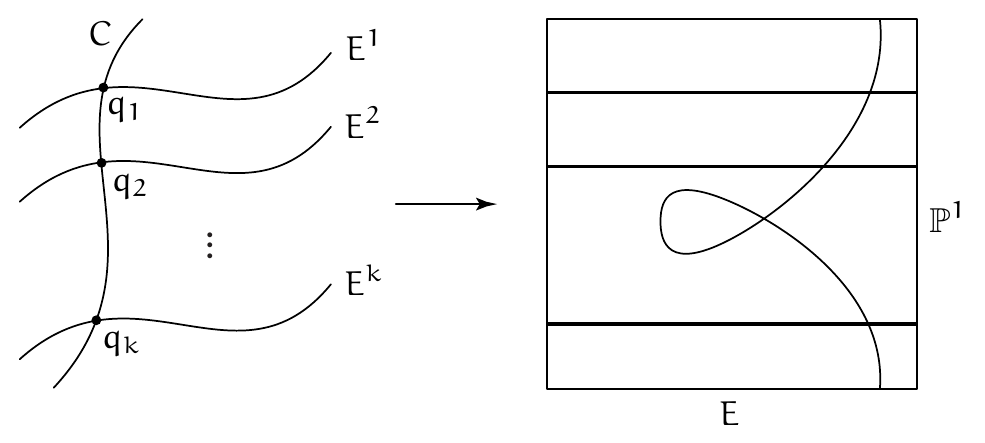}
\caption{A curve with \emph{elliptic tails}. We impose that each $E^i \to E \times \Pbb^1$
is unramified.}
\label{figure:elliptic tails}
\end{figure}

The components of $\widetilde{V}^{\circ,\text{tails}}_{df+Ne,g}(\alpha,0)$ \emph{will}
 be all equality cases. The result is the following.

\begin{proposition}
\label{proposition:preliminary_deformation_bound}
Assume \Cref{hypothesis:def-theory-set-up} and moreover that for
generic $b \in V$, the map $C_b \to E \times \Pbb^1$ has no floating components,
and that the sections $\sigma_i(V)$ get contracted by the map $\phi\from \Cc \to E \times \Pbb^1$
to points $p_i$, for all $i$. Then, $\dim V$ is at most $d+g-1$.

If equality occurs, then the image of $V$ in $|\Lc|$ is dense in a component of
the Severi variety
$\widetilde{V}^{\circ,\text{tails}}_{df+Ne,g}(\alpha,0)[p_1,\ldots,p_n]$.
\end{proposition}

The following proof follows Proposition 3.4 in Vakil's \cite{vakil_counting_2000} closely.
\begin{proof}
After passing to a dominant generically finite cover of $V$,
 we may distinguish the components of $C_b$. If generically there are
contracted components, discard them. This only makes the bound tighter.

Let us replace $C_b$ with its normalization $\widetilde{C_b}$. Again, this will only make
the bound tighter. Let $k$ be the number of floating
components in $\widetilde{C_b}$, and let
$Y$ be the union of the other components. 
As the original curve $C_b$ did not have any floating component, there were
at least $k$ nodes connecting them to other components of $\widetilde{C_b}$. Hence, the arithmetic
genus of $p_a(Y)$ is at most $g-k$, with equality only if $C_b$ is equal to the
curve $Y$ with $k$ elliptic tails attached.

Each floating component of $\widetilde{C_b}$ contributes to at most one dimension to $V$,
because the map $V \to |\Lc|$ is generically finite, so the only moduli comes from
picking the target fiber $E \times \set{t}$.

Let us focus on each component $X$ of $Y$ then. Let $\overline{X}$ be the image of 
$\phi\from X \to E \times \Pbb^1$, and $\widetilde{X} \to \overline{X}$ its normalization.
As $X$ is smooth, the map $X \to \overline{X}$ factors through 
$\widetilde{X} \to \overline{X}$. Moreover, since
the fibers of $V \to |\Lc|$ are finite, a bound in the dimension in which 
$\widetilde{X} \to \overline{X} \subset E \times \Pbb^1$ varies
also bounds the dimension of $V$. Let $m$ be the degree of
$X \to \widetilde{X}$. Summing over the components of $\widetilde{C_b}$, and applying \Cref{lemma:dimension-alpha}, we get
\begin{align*}
	\dim V &\leq \left( \sum_{X \text{ floating}} \dim \left\{ X \to E\times \Pbb^1 \right\} \right)+ 
	\left(\sum_{X \text{ not floating}} \dim \left\{\widetilde{X} \to E \times \Pbb^1 \right\} \right)\\
	&\leq \left( \sum_{X \text{ floating}} 1 \right)+ 
	\left( \sum_{X \text{ not floating}} \frac{d_i}{m} + p_a(\widetilde{X})-1 \right)\\
	&\leq k + \left(\sum_{X \text{ not floating}} \frac{d_i}{m} + \frac{p_a(X)-1}{m} \right)\\
	&\leq k + \left(\sum_{X \text{ not floating}} d_i + p_a(X) -1 \right) \\
	&= k + d+ p_a(\widetilde{Y}) -1 \\
	&\leq k+ d+(g-k)-1 = d+g-1
\end{align*}
as we wanted to show.

If equality holds, then $m=1$, $Y$ is a generic point of $V^\circ_{df+N'e,g-k}(\alpha,0)$
and $C_b$ is a generic point in $\widetilde{V}^{\circ,\text{tails}}_{df+Ne,g}(\alpha,0)$.
\end{proof}

In our application, the intersection with $E_0$ will not be fixed. 
However, we can immediately adapt \Cref{proposition:preliminary_deformation_bound} to the
following.
\begin{corollary}
\label{proposition:deformation_bound}
Assume \Cref{hypothesis:def-theory-set-up} and that for
generic $b \in V$, the map $C_b \to E \times \Pbb^1$ has no floating components.
There is a map $V \to E_0 \times \ldots \times E_0$ that sends
\[
	b \mapsto \Bigl(\phi(\sigma_1(b)), \ldots, \phi(\sigma_n(b)) \Bigr)
\]
Let $P$ be the image of this map. Then 
\[
\dim V \leq d+g-1 + \dim P
\]
If equality holds, then each component of the 
fiber of $V \to P$ over $(p_1,\ldots,p_n) \in P$ is dense in a
component of
$\widetilde{V}^{\circ, \text{tails}}_{Ne+df,g}(\alpha, 0) [p_1,\ldots, p_k]$.
\end{corollary}

The elliptic tails arose naturally as an equality case in our dimension bound,
but they do not actually occur as residual curves in the degeneration argument.
To see that, we ask when elliptic tails can be smoothed away,
that is, when is a map $f\from C \to E \times \Pbb^1$ with elliptic tails is 
a limit of a map with smooth source. By dimensional reasons, the general point of
$V^{\circ,\text{tails}}_{df+Ne,g} \setminus V^\circ_{df+Ne,g}$ cannot be smoothed into $V^\circ_{df+Ne,g}$. But more is
true---as long as $f$ is unramified, it cannot be smoothed!

\begin{proposition}
\label{proposition:tail}
Let $f\from X=C \cup_q \tilde{E} \to E \times \Pbb^1$ be a finite flat unramified map 
from a nodal curve such that $\tilde{E}$ lands in a $E$ fiber, and 
$\tilde{E} \to E$ is an isogeny. Then the map $f$ cannot be deformed in a way that the
node $q$ is smoothed.
\end{proposition}

In particular, no component of $V^{\circ,\text{tails}}_{Ne+df,g}(\alpha,0)$ is contained in the 
closure of $V^{\circ}_{df+Ne,g}$.

\begin{proof}
By Theorem~5.1 of \cite{vistoli_deformation_1997}, the deformations of the unramified map
$f$ correspond to $\Hom_{\Oc_X}(\Nc^\vee_f, \Oc_X)$, where $\Nc^\vee_f$ is
the dual normal sheaf defined by: 
\begin{equation}
\label{equation:normal sheaf}
	0 \to \Nc^\vee_f \to f^* \Omega_{E \times \Pbb^1} \to \Omega_X \to 0
\end{equation}

The key computation is the following.
\begin{claim}
\label{claim:computing normal sheaf}
The restriction $\Nc^\vee_f|_{\tilde{E}}$ is equal to $\Oc_{\tilde{E}}(-q)$.
\end{claim}

This implies that the sheaf $\Nc^\vee_f$ is locally free at $q$. The map
to the component of the $T^1_X$ sheaf supported on $q$ is just the restriction map
\begin{equation}
\label{equation: map to t1}
\Def(f) = H^0(X, \Nc_f) \to H^0(X, \Nc_f|_q) = H^0(X,T^1_X|_q)
\end{equation}

To establish that the node $q$ cannot be smoothed away, we just need to show that 
the map in \cref{equation: map to t1} is zero. But the restriction map factors as
\[
	H^0(X, \Nc_f) \to H^0(\tilde{E},\Nc_f|_{\tilde{E}})\to H^0(\tilde{E}, k_q)
\]
and the latter map 
	$H^0(\tilde{E},\Nc_f|_{\tilde{E}})
	=H^0(\tilde{E},\Oc_{\tilde{E}}(q))\to H^0(\tilde{E}, k_q)$
is zero, since
\[
	H^0(\tilde{E},\Oc_{\tilde{E}}) \iso
	H^0(\tilde{E}, \Oc_{\tilde{E}}(q))\overset 0 \to H^0(\tilde{E}, k_q) 
	\iso H^1(\tilde{E},\Oc_{\tilde{E}})
\]
This finishes the proof of \Cref{proposition:tail}, given 
\Cref{claim:computing normal sheaf}. \end{proof}

\begin{proof}[Proof of \Cref{claim:computing normal sheaf}]
Restricting the sequence 
(\ref{equation:normal sheaf}), we get
\begin{equation}
\label{equation:restriction}
	\Tor^1_{\Oc_X}(\Omega_X,\Oc_{\tilde{E}}) \to 
	\Nc^\vee_f|_{\tilde{E}} \to \pi^* \Omega_{E \times \Pbb^1} \to 
	\Omega_X|_{\tilde{E}} \to 0
\end{equation}
where $\pi\from \tilde{E} \to E \times \Pbb^1$ is the restriction of $f$ to $\tilde{E}$.
We can compute the extremal terms of the sequence as follows.
\begin{claim}
\label{claim:tor computation}
The sheaf
$\Tor^1_{\Oc_X}(\Omega_X,\Oc_{\tilde{E}})$ vanishes, and
$\Omega_X|_{\tilde{E}} = \Omega_{\tilde{E}} \oplus k_q$.
\end{claim}
\begin{proof}
We only have to check this around the node $q$. We replace $X$ by $\Spec R$, for 
$R=k[x,y]/(xy)$, and $\tilde{E}$ by the $x=0$ branch. Then $\Omega_R$ has the free
resolution
\[
	0 \to R \langle ydx +x dy \rangle \to R \langle dx, dy \rangle \to \Omega_R \to 0
\]
Tensoring with $R/(x)$, we get the complex
\[
	0 \to R/(x) \langle ydx \rangle \to R/(x) \langle dx,dy \rangle
\]
This is exact at $R/(x) \langle ydx \rangle$, and hence $\Tor^1$ vanishes. The cokernel
is $R/(x) \langle dy \rangle \oplus R/(x,y) \langle dx \rangle$, which implies that $\Omega_X|_{\tilde{E}} = \Omega_{\tilde{E}} \oplus k_q$.
\end{proof}

We use \Cref{claim:tor computation} to simplify the sequence (\ref{equation:restriction}),
and compare it to the normal sheaf sequence of $\pi\from \tilde{E} \to E \times  \Pbb^1$
as follows.
\[
	\xymatrix{
	0 \ar[r] & \Nc^\vee_f|_{\tilde{E}} \ar[r] \ar[d]& 
	\pi^* \Omega_{E\times \Pbb^1} \ar[r] \ar[d] & 
	\Omega_{\tilde{E}} \oplus k_q \ar[d] \ar[r] & 0 \\
	0 \ar[r] & \Nc^\vee_{\pi}  \ar[r] 
	& \pi^* \Omega_{E\times \Pbb^1} \ar[r] &
	\Omega_{\tilde{E}} \ar[r] & 0 
	}
\]
The snake lemma gives us the exact sequence
\[
	0 \to \Nc^\vee_f|_{\tilde{E}} \to \Nc^\vee_{\pi} \to k_q \to 0
\]
Hence, it is enough to show that $\Nc^\vee_{\pi}= \Oc_{\tilde{E}}$. This
in turn follows from the exact sequence
\[
	0 \to \Nc_{\pi}^\vee \to \pi^* \Omega_{E \times \Pbb^1} \to \Omega_{\tilde{E}} \to 0
\]
since
$\pi^* \Omega_{E\times \Pbb^1}= \Omega_{\tilde{E}} \oplus \Oc_{\tilde{E}}$,
and the map  to $\Omega_{\tilde{E}}$ is the projection onto the first factor.
\end{proof}

\subsection{Covers of \texorpdfstring{$E$}{E} of low gonality} 
\label{sub:cover_of_E_of_low_gonality}
In the proof of \Cref{lemma:dimension-alpha}, we left two details to be checked later:
that the general curve parametrized by $V_{df+Ne,g}(\alpha,0)$ does not have tacnodes, for $d=2$,
and that it does not have triple points, for $d=3$. Fortunately, 
we can deal with these cases directly through ad hoc methods,
as follows.

\subsubsection{The hyperelliptic locus \texorpdfstring{($d=2$)}{}} 
\label{ssub:the_hyperelliptic_locus}
The intersection data with $E_0$ is an effective divisor $D=p_1+p_2$
on $E_0$, which is fixed. This defines a degree $2$ map
$\phi\from E \to \Pbb^1$. We will translate our problem of studying
 curves in $E \times  \Pbb^1$ to studying curves in $\Pbb^1\times \Pbb^1$,
 using the map
 \[
 	\Phi = (\phi, id)\from  E\times \Pbb^1 \to \Pbb^1 \times \Pbb^1
 \]

 The divisor $D=p_1+p_2$ is a fiber of $\phi$. Let $p_0=\Phi(p_1)=\Phi(p_2)$.

Let $C \subset E \times  \Pbb^1$ be a curve parametrized by 
$V_{df+Ne,g}(\alpha,0)[p_1,p_2]$.
Then the map $\Phi$ restricted to $C$ is two to one. Indeed,
every intersection with $E \times  \set{t}$ is linearly equivalent to
$D$, and hence a fiber of $\phi$.

Let $\overline{C}= \Phi(C)$ be its image. It has class $(N,1)$ in $\Pbb^1 \times \Pbb^1$,
and is a smooth rational curve. Note that $p_0 \in \overline{C}$. 

The map $\phi\from E \to \Pbb^1$ has four branch points. Set 
\[
	D= \set{\text{branch point of }f} \times \Pbb^1 \subset \Pbb^1 \times \Pbb^1
\]

Away from $D$, the map $\Phi$ is étale. Hence, the curve $C$ can be singular only along
the preimage of $D$. We can relate the singularity of $C$ with the tangency index of 
$\overline{C}$ with $D$. We exhibit the first cases in \Cref{table:hyperelliptic}.
\begin{table}[t]
\begin{tabular}{lccc}
\toprule
\multicolumn{2}{c}{Singularity of $C$} & 
\multicolumn{2}{c}{Tangency of $\overline{C}$ with $D$}\\
\cmidrule(lr){1-2} \cmidrule(lr){3-4}
Type & Local Equation & Multiplicity & Local Equation \\ 
\midrule
Smooth & $x=y^2$ & $1$ & $x=y$ \\
Node & $x^2=y^2$ & $2$ & $x^2=y$ \\
Cusp & $x^3=y^2$ & $3$ & $x^3=y$ \\
Tacnode & $x^4=y^2$ & $4$ & $x^4=y$\\
\bottomrule
\end{tabular}
\caption{A dictionary between tangency conditions of $\overline{C}$ with $D$,
and singularities of $C \subset E \times \Pbb^1$.}
\label{table:hyperelliptic}
\end{table}


We will constrain the tangency profile of $\overline{C}$ with $D$. Let 
\begin{align*}
	R &= \sum_{p\in D} \left( \mult_p(\overline{C} \cdot D) - 1 \right) \\
	&= \overline{C} \cdot D  - \card{\overline{C} \cap D} 
	= 4N -\card{\overline{C} \cap D}
\end{align*}
This is the total ramification of $\overline{C} \to \Pbb^1$ over the four 
branch points of $\phi\from E \to \Pbb^1$.
 Let $\punctured{E}$ (resp. $\punctured{\left(\Pbb^1 \right)}$) be the complement of
 the ramification (resp. branch) points of $\phi$. Let $\punctured{C}$
 (resp. $\punctured{\overline{C}}$) be the corresponding preimage.
 Using Riemann--Hurwitz, we can count the ramification as follows.
 \begin{align*}
 2g-2 &= \Ram(\phi\from C \to E) \\
 &\geq \Ram( \phi\from \punctured{C} \to \punctured{E}) \\
&=2 \Ram( \punctured{\overline{C}} \to \punctured{\left(\Pbb^1 \right)}) \\
&= 2 \left( \Ram(\overline{C} \to \Pbb^1) - R \right) \\
& = 2 ( 2N-2 -R) =2( \card{\overline{C} \cap D} -2N -2)
 \end{align*}

 Equality holds
 only if $C \to E$ is not branched over the four ramification points of $\phi$.
 However, if $\overline{C}$ had more than simple
  tangencies along $D$, then $C \to E$ would be ramified there, as we
  can see from the local equations in \Cref{table:hyperelliptic}.
Hence, if equalities hold, 
all tangencies of $\overline{C}$ with $D$ are simple, 
and the curve $C \subset E \times  \Pbb^1$ is nodal. 

We can rewrite the inequality 
 as
 \begin{equation}
 \label{eq:Q_bound}
 	\card{\overline{C} \cap D} \leq 2N+g+1
 \end{equation}

We now invoke \Cref{theorem:deftheory} to bound the dimension in
which $\overline{C} \subset \Pbb^1 \times \Pbb^1$ varies. Let $\Sigma$ be $\Bl_{p_0} \Pbb^1 \times \Pbb^1$, 
the blow up of $\Pbb^1 \times \Pbb^1$ at $p_0$, and set $\tau = Nf_1 + f_2 -e$,
where $e$ is the exceptional divisor.
We use the same $D$ as above, and set $\Omega= \emptyset$. By \cref{eq:Q_bound},
we may set $b=2N+g+1$. We have
\begin{align*}
	\gamma &= - (K_{\Bl_{p_0} \Pbb^1\times \Pbb^1} +D)\cdot \overline{C} + 2N+g+1\\
	&= -(-2f_1 + 2f_2 +e)(Nf_1 +f_2 -e) + 2N+g+1 \\
	&= -(2N-2+1)+2N+g+1 = g+2
\end{align*}
And hence the dimension bound
\[
	\dim V_{Ne+2f,g}(\alpha, 0) \leq \gamma + 0 -1 = g+1
\]
On the other hand, by \Cref{lemma:dimension_lower_bound}, the opposite inequality
holds as well. Hence, all equalities hold, including in \cref{eq:Q_bound}. As 
$\gamma \geq 2$, \Cref{theorem:deftheory} tell us that equality must hold
in (\ref{eq:Q_bound}), and hence,
$C$ is nodal, as we wanted to show.


\subsubsection{The trigonal locus \texorpdfstring{($d=3$)}{}} 
\label{ssub:the_trigonal_locus}
We now deal with the case of $V_{3f+Ne,g}(\alpha,0)$. We want to show that the general point 
corresponds to a nodal curve, and we just have to eliminate the case of triple points.

The tangency conditions with $E_0$ define a fixed divisor $D=p_1+p_2+p_3$ on $E_0$.
 This divisor defines an embedding of $E \into \Pbb^2$ as a smooth cubic.
  Each intersection $C \cap \left(E \times \set{t}\right)$ maps to three collinear
points in $E \subset \Pbb^2$. We think of this as a map 
$\phi\from \Pbb^1 \to \dual{\Pbb^2}$
sending $t$ to the line joining the three collinear images of 
$C \cap \left(E \times \set{t}\right)$.
Note that $\phi(0)$ goes to the fixed line $\ell_0$ connecting the images of 
$p_1,p_2$ and $p_3$.

Conversely, given the map $\phi$, we can reconstruct $C$ as the intersection of
$E \times \Pbb^1 \subset \Pbb^2 \times \Pbb^1$ with 
the ruled surface in $\Pbb^2 \times \Pbb^1$ traced out by the lines $\phi(t)$
as we vary $t \in \Pbb^1$. See \cref{figure:trigonal curve construction- fig13}.

\begin{figure}
\centering
\includegraphics{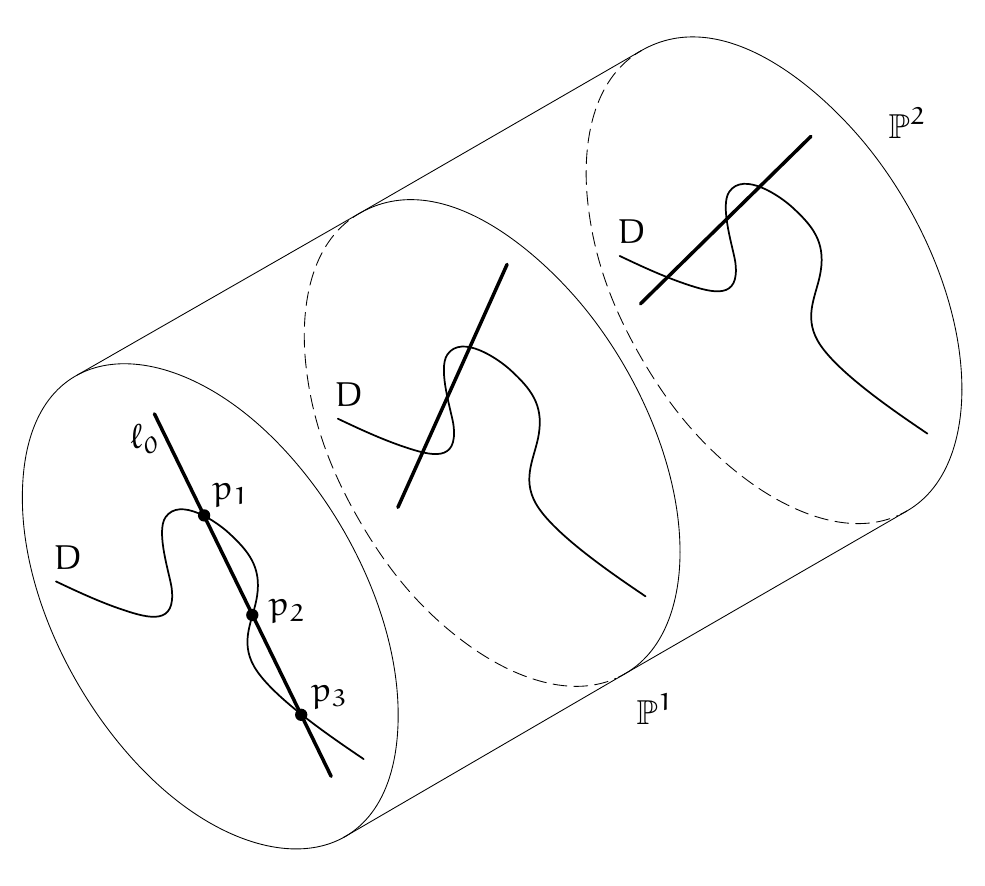}
\caption{Recovering the curve $C$ as the intersection of $D \times \Pbb^1$ 
and the ruled surface traced out by the lines $\phi(t)$.}
\label{figure:trigonal curve construction- fig13}
\end{figure}

We can relate the singularities of $C$ with the tangency of the image
of $\phi$ along $\dual{E} \subset \dual{\Pbb^2}$, the dual to $E \subset \Pbb^2$.
The dual $\dual{E}$ is a degree $6$ curve with $9$ cusps (one for each flex of $E$). 

If $C$ is singular at $(p,t) \in E \times \Pbb^1$, then $\phi(t)$ lies on 
$\dual{E}$---specifically,
at the point corresponding to the tangent to $E$ at $p$. Conversely,
if we denote by $\widetilde{C} \to C$ the normalization map, we have:
\begin{itemize}
	\item If $\phi(\Pbb^1)$ is transverse to $\dual{E}$ at a point $\phi(t)$,
	then $\widetilde{C} \to \Pbb^1$ is simply branched over $t$, and $C$ is smooth at
	the corresponding point of $E \times \Pbb^1$.
	\item If $\phi(\Pbb^1)$ is simply tangent to $\dual{E}$ at a point $\phi(t)$,
	then $C \subset E \times \Pbb^1$ is nodal at the corresponding point, 
	and $\widetilde{C} \to \Pbb^1$ is not branched over $t$.
	\item If $\phi(\Pbb^1)$ does not contain any of the $9$ cusps of $\dual{E}$,
	then $C$ has no triple points.
\end{itemize}

Hence, it is enough to show that the image of $\phi(\Pbb^1)$ has at most
simple tangencies with $\dual{E}$, and does not pass through any of the cusps.
We will establish that by invoking once more  \Cref{theorem:deftheory},
now to bound the dimension in which the image of $\phi$ can
vary.

By Riemann--Hurwitz, $\widetilde{C} \to \Pbb^1$ has $2g+4$ branch points. Let $T$
be the number of transverse intersections of $\phi(\Pbb^1)$ and $\dual{E}$.
Hence $T \leq 2g+4$. Moreover,
\begin{align*}
	6N &= \phi(\Pbb^1) \cdot \dual{E}  \\
	&= \sum_p \mult_p(\phi(\Pbb^1) \cdot \dual{E}) \\
	&\geq T + 2(\card{\phi(\Pbb^1)\cap \dual{E}} - T)  \\
	&= 2 \card{\phi(\Pbb^1)\cap \dual{E}} - T \\
	&\geq 2 \card{\phi(\Pbb^1)\cap \dual{E}} - (2g+4)
\end{align*}
With equality holding only if at most simple tangencies occur.
Rearranging, we get
\begin{equation}
\label{eq:trigonal-bound-Q}
	3N +g +2 \geq \card{\phi(\Pbb^1)\cap \dual{E}} 
\end{equation}

We set $\Sigma= \Bl_{\ell_0} \dual{\Pbb^2}$. The class $\tau$ of the proper
transform of $\phi(\Pbb^1)$ is $Nh-e$, where $e$ is the class of the
exceptional divisor.
 Using the notation of
\Cref{theorem:deftheory}, and setting $b=3N+g+2$, we have
\begin{align*}
	\gamma&= -(K_{\Sigma} + \dual{E})\cdot \tau + 3N+g+2 \\
	&=-(3h+e)(Nh-e)+3N+g+2 \\
	&= -(3N+1)+3N+g+2 = g+1
\end{align*}

So the image of $\phi(\Pbb^1)$ can vary in at most $0-1+g+1=g$ 
dimensions. Adding $\dim (\Aut(\Pbb^1,0))=2$ to remember the map 
$\phi\from \Pbb^1 \to \dual{\Pbb^1}$ (we know that $0$ maps to $\ell_0$),
we get that
the dimension of $V_{3f+Ne, g}(\alpha,0)$ is at most $g+2$.
However, by \Cref{lemma:dimension_lower_bound}, the opposite inequality
holds as well. Hence, equality in (\ref{eq:trigonal-bound-Q}) holds, 
and the original curve $C$ is nodal.

\section{Hyperplane Sections of \texorpdfstring{$\Vdg$}{the Severi Variety}} 
\label{sec:hyperplane_sections}

In this section, we will prove \Cref{theorem:simple_hyperplane_section_of_severi,theorem:simple_hyperplane_to_hurwitz} stated in the introduction,  
and \Cref{theorem:hyperplane_section_of_severi} which is a generalization of them.
 Let us briefly remind
 ourselves of the setup. We fix a fiber $E_0 \subset E \times \Pbb^1$, 
 and let $p \in E_0$ be a general point. We want to characterize the irreducible 
 components of
  of $\Vdg \cap H_p$. Let $W$ be one of these components, and $Y_0 \subset E \times  \Pbb^1$ 
  be the curve corresponding to a general point in $W$.

   If the curve $Y_0$ does not contain $E_0$, it is easy to describe what $W$ must be---we will
do that in \Cref{ssub:proof_of_theorem_ref_theorem_hyperplane_section_of_severi}.
Let us assume, then, that the curve $E_0$ is contained in $Y_0$.
Now it is much harder to analyze the residual curve. 
 Our main tool to get our hands on its geometry is the 
following construction.

\subsection{Nodal reduction} 
\label{sub:nodal_reduction}

Pick a general point $Y_0 \subset E \times  \Pbb^1$ in $W$. 
Choose a general arc $\Delta$ in $\Vdg$ approaching the point $(Y_0) \in W$.
We may assume that $\Delta$ is smooth by taking its normalization.

We get a flat family of curves mapping to $E\times \Pbb^1$,
\[
	\xymatrix{\Yc \ar[r]^-{\phi} \ar[d]^\pi & E \times \Pbb^1 \\
	\Delta}
\]
where the generic fiber is a general point in $\Vdg$, and the central fiber 
is a generic point in $W$. Let $\Yc^\times$ be the restriction of $\Yc$ over
the punctured base $\Delta^\times$.

The preimage of $E_0 \subset E \times \Pbb^1$ under $\phi$ on the punctured family $\Yc^\times$
 will consist of some multisections.
 After a a base change, we may
assume that the multisections are actually sections. Being in $\Vdg$, there will be $a$ of those which
get contracted under the map $\phi$, and $b$ that are allowed to vary---that is, are not necessarily contracted
by $\phi$. To avoid overburden our notation even more, we will refer to these as
$A$-sections and $B$-sections.

At this point there is not much we can say about the central fiber. However,
if we perform nodal reduction on $\Yc \to \Delta$, we can replace it by a more
tractable object. Let us denote by $\Cc \to B$  the family obtained by nodal reduction,
and $f\from B \to \Delta$ the corresponding base change map. Here is what we can assume about
it.

\begin{hypothesis}
The family $\Cc \to B$ satisfies the following properties.
\label{hypothesis:nodal-reduction}
\begin{itemize} 
	\item The total space $\Cc$ is smooth.
	\item The fibers $C_b$ are equal to the corresponding fiber $Y_{f(b)}$
	of the original family $\Yc \to \Delta$, as long as $b\neq 0$.
	\item The (scheme-theoretic) central fiber $C_0$ is nodal (and therefore
	reduced).
	\item The maps $\phi\from C_b \to E\times \Pbb^1$ are semistable.
	\item The preimage of $E_0$ is a union of sections of $\Cc \to B$, plus
	components in the central fiber. These sections are disjoint on the 
	general fiber, but are allowed to meet along the central fiber.
	We will still call them A-sections and B-sections, as above.
	\item The family $\Cc \to B$ is the minimal one respecting the conditions above.
\end{itemize}
\end{hypothesis}

See \cref{figure:nodal reduction family - fig10} for a diagram of the family $\Cc \to B$
away from the central fiber. Our goal is to determine how the central fiber looks like, and
where the A and B-sections meet it.

\begin{figure}
\centering
\includegraphics{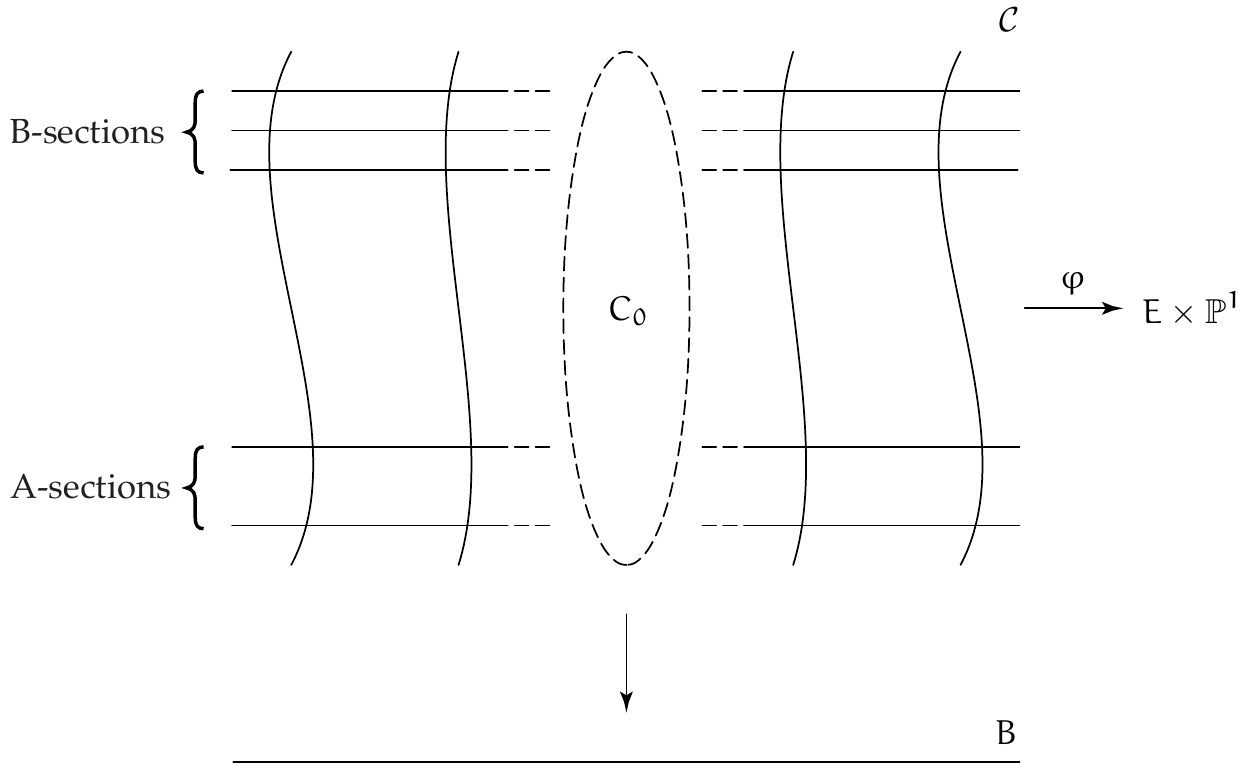}
\caption{A diagram of the family $\Cc \to B$ away from the central fiber.}
\label{figure:nodal reduction family - fig10}
\end{figure}

The idea of the proof of \Cref{theorem:simple_hyperplane_section_of_severi}
is simple, but combinatorially challenging. There are
two opposing forces controlling the central fiber. On one hand, it varies in
a large dimension, and hence its singularities should be mild, and few extra 
conditions should be imposed. That is, the dimension count forces
 the central fiber to be ``simple''. 
 On the other hand, from the parametric point of view, the flatness of
the family $\Cc$ says the central fiber has to have arithmetic genus $g$, 
and hence it cannot be \emph{too} simple. 
We will determine the equilibrium positions,
the conditions under which the two forces precisely cancel out. 
And it will turn out that these conditions 
completely describe the components of the hyperplane section.

To make the structure of the proof transparent, we will break up the argument
 into a chain of small claims. Each one will be an inequality expressing 
the tension between the two points of view. As we compose this long chain of
inequalities, we will conclude that all equalities must hold! We then come back
and use the equality conditions to pin down the possible components of the hyperplane 
section.

Again, we do not prove that all of these components do appear in a 
hyperplane section, or even if they do appear, we do not say with what multiplicity.
Here we only compile a list of \emph{suspects}: any component $W$ must be in this
list.

The difficulty is that our starting point, the central fiber $C_0$,
is an arbitrary nodal curve. Just notionally this is already very cumbersome. We
have to design our notation to express all the arguments in this generality.

We divide the components of $C_0$ in three groups, forming nodal curves each:
\begin{itemize}
 	\item We call $\widetilde{E}$ the curve which is the union of irreducible components 
 	of $C_0$ that dominate $E_0$ under the map $\phi\from  C_0 \to E \times \Pbb^1$.
 	\item Let $Z$ be the union of the components that get contracted to a point in
 	$E_0$ by the map $\phi_0$.
 	\item Let $X$ be the union of the remaining components.
 \end{itemize} 

Let $T$ be the number of connected 
 components of $Z$ meeting both $X$ and $\widetilde{E}$, 
plus the number of nodes between $X$ and $\widetilde{E}$. For an example, see \cref{figure:central fiber - fig11}.

\begin{figure}
\centering
\includegraphics{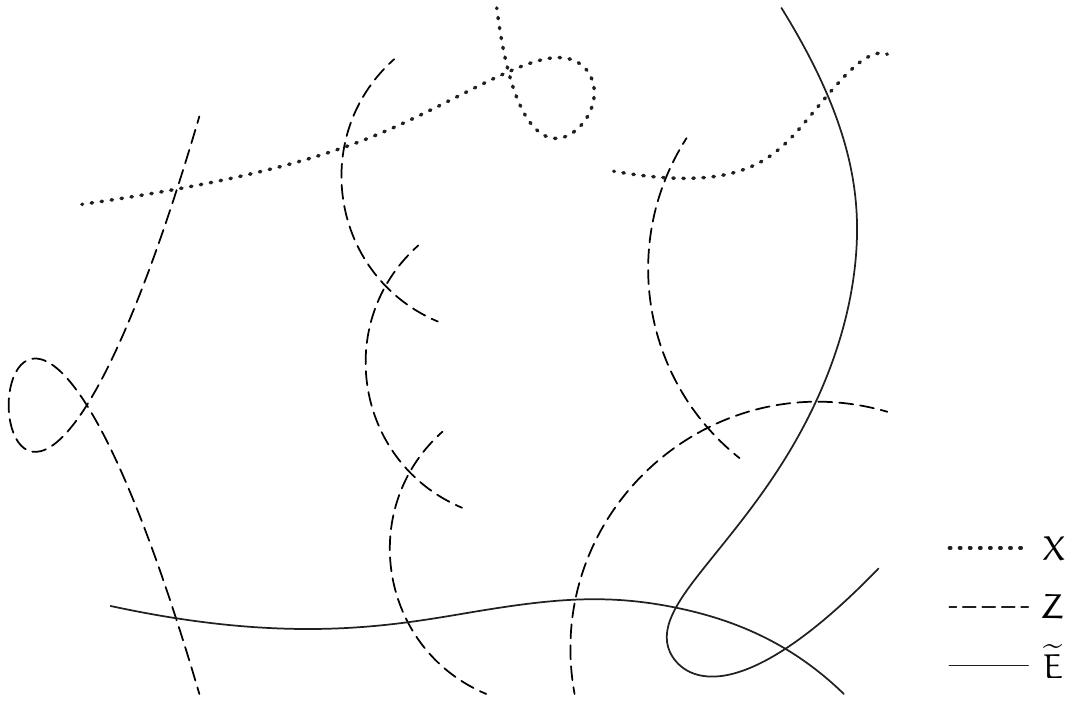}
\caption{An example of an a priori possible central fiber $C_0$, with its components grouped as $\tilde{E}$,
$Z$ and $X$. In this example, $T=4$. The central fibers that do arise will turn out to be much simpler than
this, as we will see.}
\label{figure:central fiber - fig11}
\end{figure}


\subsection{The genus bound} 
\label{sub:the_genus_bound}
In this section we will leverage the flatness of the family $\Cc \to B$ to
bound the arithmetic genus of $X$ in terms of $g$ and $T$.

Start by enumerating the irreducible components of $Z$ by $Z_1,\ldots, Z_m$.
To help us keep  track of the combinatorial data of $C_0$, we introduce an auxiliary 
graph $G$. We obtain $G$ by starting with the \emph{dual graph} of 
the nodal curve $C_0$, and then collapsing the subgraph corresponding to components of
$X$ (resp. $\widetilde{E}$) into a single vertex. Hence, it has $m+2$ 
vertices---one for each $Z_i$, and one for $\widetilde{E}$ and one for
$X$, even though $\widetilde{E}$ and $X$ may well be reducible. The edges of $G$ correspond
to the nodes of $C_0$ which are not \emph{internal} to $X$ or 
$\widetilde{E}$---that is, we 
don't consider nodes which connect two components of $X$, or two components of $\widetilde{E}$.
In particular, there are no loops in $G$ at the vertices corresponding to 
$\widetilde{E}$ and $X$.
We will refer to the nodes corresponding to edges in $G$ as \emph{external} nodes. Note that 
the number of edges in $G$ is equal to the number of external nodes. For an example, see
\cref{figure:nodal reduction family - fig12}.

\begin{figure}
\centering
\includegraphics{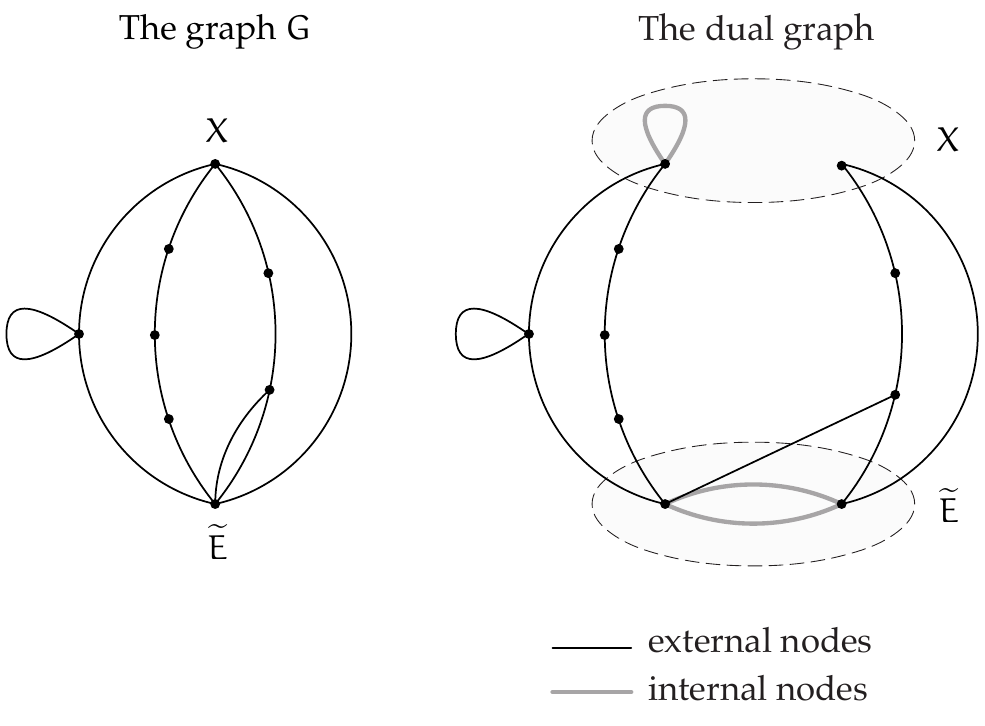}
\caption{The dual graph of the example $C_0$ in \cref{figure:central fiber - fig11}, 
and the respective graph $G$ obtained by contracting the subgraphs corresponding to
$X$ and $\tilde{E}$.}
\label{figure:nodal reduction family - fig12}
\end{figure}

\begin{claim}
\label{claim:T}
Let $d_X$ (resp. $d_{\widetilde{E}}$) be the degree of $X$ (resp. $\widetilde{E}$)
in the graph $G$. 
Then,
$T\leq d_X$  and $T \leq d_{\widetilde{E}}$.
If both equalities hold, then every connected component of $Z$
meets $X$ and $\widetilde{E}$ exactly once each.
\end{claim}
\begin{proof}
This follows directly from the definition of $T$. The inequality is strict in
two situations: if there are connected components of $Z$ meeting only one of
$X$ and $\widetilde{E}$, or if there are connected components of $Z$ meeting 
$X$ or $\widetilde{E}$ in more than one node.
\end{proof}

\begin{claim}
\label{claim:g_Z}
Let $d_{Z_i}$ be the number of nodes on $Z_i$ in $C_0$, that is, the 
degree of $Z_i$ in the graph $G$. Then
\[
	0 \leq g(Z_i)-1+\frac{d_{Z_i}}{2}
\]
with equality only if $Z_i$ is rational with exactly two nodes on it.
\end{claim}
\begin{proof}
This comes from the minimality of the family. If this quantity was negative,
then $Z_i$ would be a rational curve meeting the rest of the 
central fiber at only one point. Then, we compute the intersection pairing on
the surface $\Cc$:
\[
	0=Z_i C_0 =Z_i^2 +1
\]
from which we conclude that $Z_i$ is a ($-1$)-curve. Moreover, being a component
in $Z$, the map $\phi$ contracts $Z_i$. Hence, we could
contract it in $\Cc$ and maintain the conditions of \Cref{hypothesis:nodal-reduction}, contradicting
the minimality of $\Cc$.

If equality holds, then $g(Z_i)=0$ and $d_{Z_i}=2$, as we want, or
$g(Z_i)=1$ and $d_{Z_i}=0$. The latter case does not occur, because $Z_i$ would be a disconnected component of $C_0$
which gets contracted under the map to $E \times  \Pbb^1$, and hence cannot be in the limit of maps 
$C_t \to E\times \Pbb^1$ which are birational onto their image.
\end{proof}

Next, we compute the arithmetic genus of the central fiber. By summing the degrees of all vertices
of $G$, we get twice the number of edges. That is, 
\[
	d_X+d_{\widetilde{E}}+\sum d_{Z_i} = 2 \times \card{\text{edges in }G}
	= 2 \times \card{\text{ external nodes in }C_0}
\]
Hence, we get
\begin{align*} 
 p_a(C_0)-1 &= p_a(X)-1 + p_a(\widetilde{E})-1 + \sum_{i=1}^m \left( g(Z_i)-1 \right) + 
 \card{\text{external nodes in }C_0} \\
	&= p_a(X)-1 + p_a(\widetilde{E})-1 + \frac{d_X+ d_{\widetilde{E}}}{2}
	+\sum_{i=1}^m \left(g(Z_i)-1 + \frac{d_{Z_i}}{2} \right)
\end{align*}

Moreover, by flatness,
\[
 	g-1 = p_a(C_t)-1 =p_a(C_0)-1
 \] 

Summing up, we get the following.
\begin{claim}
\label{claim:g_C_0}
We have,
\[
g-1=p_a(X)-1 + p_a(\widetilde{E})-1 +\frac{d_{\widetilde{E}}+ d_X}{2}
	+\sum_{i=1}^m \left(g(Z_i)-1 + \frac{d_{Z_i}}{2} \right)
\]
\end{claim}

 The last ingredient is the following.
\begin{claim}
\label{claim:g_tilde_E}
We have
\[
	0 \leq p_a(\widetilde{E})- 1
\]
with equality only if $\widetilde{E}$ is a disjoint union of smooth genus 1 curves, and
$\widetilde{E}\to E_0$ is unramified.
\end{claim}
\begin{proof}
This follows from the Riemann--Hurwitz formula applied to $\tilde{E} \to E_0$. Note that $\widetilde{E}$ may
be disconnected.
\end{proof}

We are ready for the following key bound.
\begin{proposition}
\label{proposition:genus_bound}
[The Genus Bound]
We have
\[
	p_a(X)+T \leq g
\]
Equality holds only if $\widetilde{E}$ is smooth, $\widetilde{E}\to E$
is unramified, and $Z$ is a union of chains of rational curves connecting $X$ and
$\widetilde{E}$.
\end{proposition}

The equality does not hold in our previous example \cref{figure:central fiber - fig11},
but in \cref{figure:genus bound equality - fig14} we can see an example where it does hold.

\begin{figure}
\centering
\includegraphics{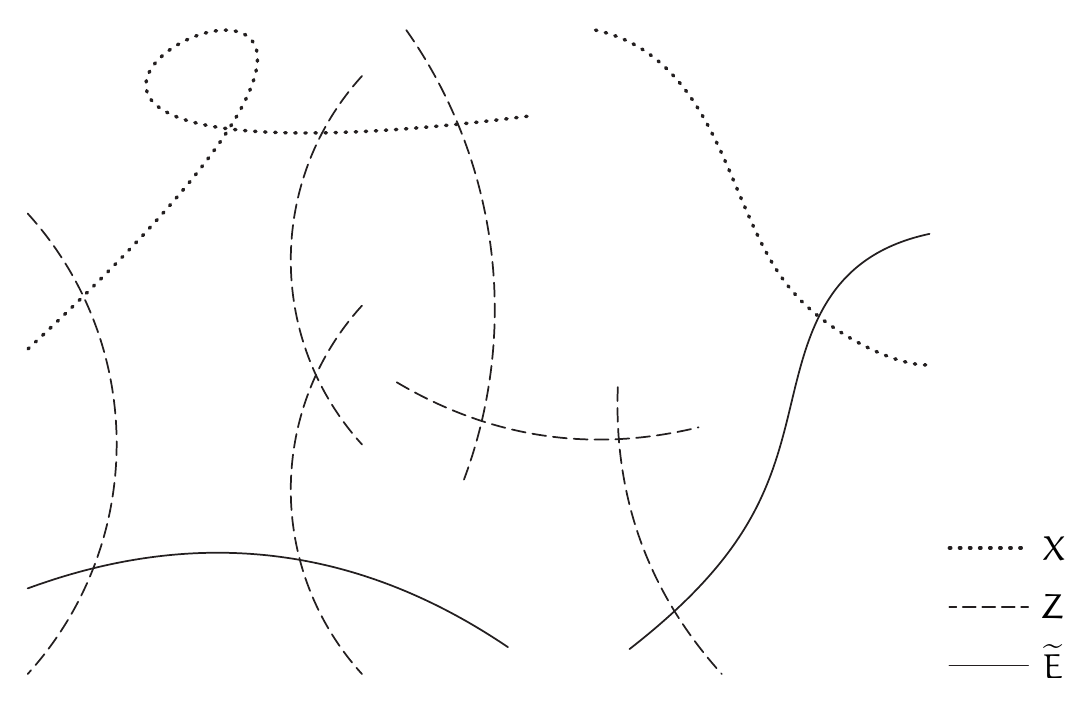}
\caption{A picture of $C_0$ in a case where equality in the genus bound (\Cref{proposition:genus_bound}) holds.}
\label{figure:genus bound equality - fig14}
\end{figure}

\begin{proof}
Just add up the inequalities in 
\Cref{claim:T,claim:g_C_0,claim:g_Z,claim:g_tilde_E}. If equality holds,
then it does as well in every claim, and hence all equalities conditions are true.

From \Cref{claim:g_tilde_E} we get that $\widetilde{E}$ is smooth and 
$\widetilde{E} \to E$ is unramified. 
The equality conditions of \Cref{claim:g_Z} imply that $Z$ is a disjoint union of
chains of rational curves. \Cref{claim:T} implies that each of these chains connects
a component of $\widetilde{E}$ to a component of $X$.
\end{proof}

Note that, so far, even if equality occurs in \Cref{proposition:genus_bound},
we can't say much about the residual curve $X$. Also note that we haven't used the sections
our family came equipped with either. These will play a role in 
the following.

\subsection{The dimension bound} 
\label{sub:the_dimension_bound}
We now want to constraint the geometry of $X$ based on the
fact it varies in a family of dimension $\dim W= \dim H_p \cap \Vdg= d+g+b-3$.

The construction in \Cref{sub:nodal_reduction} replaces a general point 
$\{Y_0 \subset E \times \Pbb^1\}$ of the component $W$ with a map $C_0 \to E \times \Pbb^1$. 
To leverage the dimension of $W$, we need to extend this \emph{pointwise} 
construction to a 
\emph{local} one. The following technical lemma does so.

\begin{lemma}
For each general point $p \in W$, we can find a map $\eta\from U \to U_0 \subset W$ which is
finite onto a open subscheme which
 contains $p$, a flat family $\Cc^U \to U$ and a map $\phi\from \Cc^U \to E \times  \Pbb^1$
 such that for each $u \in U$, the map $C^U_u \to E \times \Pbb^1$ is equal to the result of
the pointwise construction of \Cref{sub:nodal_reduction} on $\eta(u) \in W$.
\end{lemma}
\begin{proof}
The idea is the following: the
only choice we had in the construction in \Cref{sub:nodal_reduction} was of the
arc $\Delta \to V$. However, as we are performing a codimension one degeneration, the limit
will not depend on the direction of approach. Here is a standard way
to formalize this.

The variety $\Vdg$ admits
a rational map to the Kontsevich space of stable maps 
$\overline{M}_{g,a+b}(E \times \Pbb^1, df+Ne)$. Since the target is
projective and the source is normal,
 we can extend the rational map in codimension one. Hence, for a general point 
in the divisor $H_p \cap \Vdg$, we can find a neighborhood of $p$ mapping to the coarse
space $\overline{M}_{g,a+b}(E \times \Pbb^1, df+Ne)$. Possibly after a base change, we find some
neighborhood $\widehat{U}$ which maps to the stack $\overline{\Mc}_{g,a+b}(E \times \Pbb^1, df+Ne)$. Now pulling back
the universal family, we get a family $\Cc^{\widehat{U}} \to \widehat{U}$. 
Let $U\subset \widehat{U}$ be the preimage of $W$. Our universal family $\Cc^U \to U$ is the restriction
of $\Cc^{\widehat{U}} \to \widehat{U}$.
\end{proof}

The total space $\Cc^U$ is reducible: there is at least a component containing the generic
fiber's $\widetilde{E}$, and a component containing the generic fiber's $X$.
We want to pick out the component $\Xc \subset \Cc^U$ corresponding to the $X$'s.
A way to directly construct it is to set
\[
	\Xc = \closure{\Cc - \phi^{-1}(E_0)}, \text{ where } \phi\from \Cc^U \to E \times \Pbb^1
\]

Let us denote the restriction of $\phi\from \Cc^U \to E \times \Pbb^1$ to $\Xc$ by $\phi$ again.
We will bound arithmetic genus $p_a(X_u)$ in terms of the dimension of $U$ and the tangency
conditions the image of the $X_u$ satisfies along $E_0$, using \Cref{proposition:deformation_bound}.

To get our hands on the tangency conditions, we need to introduce some notation.
Let $k$ be the cardinality of $\phi^{-1}(E_0) \cap X_u$ for $u$ generic. After
a base change, we may assume that $\phi^{-1}(E_0)$ consists of $k$ distinct sections.
 Let $U \to E_0^k$ the map sending $u$ to the images under
$\phi$ of the $k$ sections. 
Let $P \subset E_0^k$ be the image of this map.

Each one of the $k$ sections appears with some multiplicity, say $\alpha_i$. We get
a corresponding map
\begin{align*}
	E_0^k &\to \Div^d E_0 \\
	(p_1,\ldots, p_k) & \mapsto \sum \alpha_i p_i \in \Div^d E_0 
\end{align*}

For example, the composite map 
$U \to E_0^k \to \Div^d(E_0)$ sends $u$ to the divisor $\phi(X_u) \cdot E_0$ on $E_0$.

Note that the linear class of $\phi(X_u) \subset E \times \Pbb^1$ is equal to the class of $C_0$ 
minus some number of $E_0$'s that got
split off. That is, $[\phi(X_u)] = \Lc - m e$, and the intersection $\phi(X_u) \cdot E_0 =\Lc|_{E_0}$ is a fixed linear class.
Hence, the following diagram commutes.

\[
	\xymatrix{
	U \ar[r] \ar@{-->}[drr] & P \subset E_0^k \ar[r] & \Div^d E_0 \ar[r] & \Pic^d E_0  \\
	& &  | \Lc|_{E_0}| \ar[u] \ar[r] & \Spec \Cbb = \set{\Lc|_{E_0}} \ar[u]
	}
\]

Denote by $X$ the generic fiber $X_u$.
\begin{claim}
\label{claim:deformation-bound}
We have
\[
	\dim W \leq d+ p_a(X) -1 + \dim P
\]
If equality holds, then, after a base change, $W$ is fibered over $P$, and the
the fiber over $(p_1,\ldots, p_k) \in P \subset E_0^k$ is dense in a component of
\[
	\widetilde{V}^{\circ, \text{tails}}_{df+N'e,p_a(X)}(\alpha, 0) [p_1,\ldots, p_k]
\]
where $\alpha_i$ is the multiplicity in which the section over the point $p_i$ appears, and
$df+N'e$ is the homology class of $\phi(X)$.
\end{claim}
\begin{proof}
This follows from \Cref{proposition:deformation_bound} applied to $\Xc \to U$,
and the fact that $U$ and $W$ have the same dimension.
\end{proof}

To bound the dimension of $P$, let us go back to our
one-dimensional degeneration $\Cc \to B$ as in \Cref{sub:nodal_reduction}.
Let $\epsilon_b$ be the number of B-sections which meet the central
fiber on $\widetilde{E}$ or on a connected component of $Z$ that does not meet $X$.
That is, out of the $b$ B-sections that the family $\Cc \to B$ has, $b- \epsilon_b$
of them land in $X$, or in a component of $Z$ that contracts to a point in $X$. Either way, these 
$b -\epsilon_b$ points are in the intersection of $\phi(X)$ with $E_0$. 

Define 
\[
\epsilon = \begin{cases}
1, \text{ if }\epsilon_b>0 \\
0, \text{ else.}
\end{cases}
\]

For example, we have
\begin{equation}
\label{equation:epsilon}
\epsilon \leq \epsilon_b
\end{equation}
with equality only if $\epsilon_b$ is equal to zero or one.

Now we are ready to state the  bound on $\dim P$.
\begin{claim}
\label{claim:dim-P}
We have
\[
	\dim P \leq T + b -\epsilon_b - 1- (1-\epsilon)
\]
If equality holds, none of the $A$ or $B$-sections land in components of $Z$ meeting $X$ and $\widetilde{E}$,
 and one of the two following scenarios
holds. 
\begin{enumerate}
 	\item\label{scenario:1} 
 	If $\epsilon=1$, then $P$ is dense in the preimage under 
 	$E_0^k \to \Div^d E_0$ of the sum of:
 	\begin{itemize}
 		\item a full linear series of degree $T+b-\epsilon_b$,
 		\item and $a'=d-(T+b-\epsilon_b)$ base points.
 	\end{itemize}
 	\item \label{scenario:2} Else, both $\epsilon_b$ and $\epsilon$ are zero, and
 	 $P$ is  dense in the preimage under 
 	$E_0^k \to \Div^d E_0$ of the sum of:
 	\begin{itemize}
 		\item a full linear series of degree $b$ on $E$ (corresponding to the limit of the B-sections),
 		\item a full linear series of degree $T$ on $E$,
 		\item and $a'=d-b-T$ fixed base points.
 	\end{itemize}
 \end{enumerate} 
\end{claim}
\begin{proof}
The curve $\phi(X)$ meets $E_0$ in three types of points:
\begin{itemize}
 	\item The image of A-sections. These add no moduli to $P$.
 	\item The images of B-sections. There are $b-\epsilon_b$ of those.
 	\item The images of connected components of $Z$ meeting both $X$ and 
 	$\widetilde{E}$, or nodes connecting $\widetilde{E}$ and $X$ in $C_0$. There
 	are exactly $T$ of those.
 \end{itemize} 
 Hence, the dimension of $P$ is bounded by
\[
\dim P \leq T+ b-\epsilon_b
\]
However, if equality were to hold, then $P$ would be dense in $E^{T+ b-\epsilon_b} \subset E^k$ for some subset
of the indices. But $P \to E^k \to \Div^d E$ has to land inside the full linear series $\left| \Lc|_{E_0} \right| \subset \Div^d E$,
while the image of $E^{T+ b-\epsilon_b} \subset E^k \to \Div^d E$ does not have a fixed linear class. This is a contradiction.

Hence, we get the slightly stronger inequality
\[
\dim P \leq T+ b-\epsilon_b -1
\]
for which equality happens exactly if $P$ is the preimage of some full linear series of degree $T+b-\epsilon_b$, plus some fixed 
base points (which correspond to A-sections landing in $X$).
That is, only in scenario (\ref{scenario:1}).

Something interesting happens when $\epsilon_b=0$. In this case, all the
B-sections land in $X$, and their sum is a fixed divisor up to linear equivalence. Indeed,
it is exactly
the class of $\phi(C_t)=\Lc$, restricted to $E_0$, minus the fixed images of the A-sections
$p_1,\ldots, p_a$. Hence, we get one moduli less than expected. That is, 
we may subtract $(1-\epsilon)$ from the count above, to get to
\[
\dim P \leq T+ b-\epsilon_b -1-(1-\epsilon)
\]
as we wanted to prove. Now equality happens in both scenarios (\ref{scenario:1}) and
(\ref{scenario:2}).
\end{proof}

Putting all of this together, we get the following.
\begin{proposition}[The dimension bound]
\label{proposition:dimension_bound} We have
\[
	g \leq p_a(X) +T
\]
If equality holds, then there is a partition $\tau$ of length $T$, and an integer 
$a' \leq a$, such that 
$X$ is a general point of \nobreakpar
 \begin{itemize}
 	\item $\widetilde{V}^{\circ, \text{tails}}_{df+N'e,p_a(X)}(a',b,\tau)$, if $\epsilon=0$, or
 	\item $\widetilde{V}^{\circ, \text{tails}}_{df+N'e,p_a(X)}\left(a',\tau+(b-1)\right)$, 
 	if $\epsilon=1$.
 \end{itemize}
\end{proposition}
\begin{proof}
By \Cref{theorem:dimension-of-severi}, we have
\[
 \dim W = \dim \Vdg -1
	=d+g+b-3
\]
Now combine \Cref{equation:epsilon} 
and Claims \ref{claim:deformation-bound} and \ref{claim:dim-P}, and we get the desired inequality.

If equality holds, we let $\tau=(\tau_1,\tau_2,\ldots)$, where $\tau_i$ is
 the multiplicity in which each of the corresponding $T$ points appear in the
 divisor $\phi^*(E_0)|_X$,
 and $a'$ be the number of A-sections landing in $X$.
 Considering all the equality conditions together, we get the description above.
\end{proof}

\subsection{Proof of 
Theorems 
\ref{theorem:simple_hyperplane_section_of_severi} and \ref{theorem:simple_hyperplane_to_hurwitz}} 
\label{ssub:proof_of_theorem_ref_theorem_hyperplane_section_of_severi}

Let us start with \Cref{theorem:simple_hyperplane_section_of_severi}.
Let $W$ be an irreducible component of $\Vdg \cap H_p$. If the 
generic point of $W$ is a curve that does not contain $E_0$, then
$W$ is contained in 
\[
	\widetilde{V}^\circ_{Ne+df,g}(a+1,b-1) [p_1,\ldots, p_a, p]
\]
However, by \Cref{theorem:dimension-of-severi}, this has the same dimension as $W$. Hence,
$W$ must be one of its irreducible components, as we wanted to show.

Let us assume now that the general point of $W$ is a curve which contains $E_0$ with
multiplicity $m>0$. Set $N'=N-m$, so that the homology class of the residual curve, $\phi(X)$, is $N'e+df$.

Then we  apply the results of 
\Cref{sub:the_dimension_bound,sub:the_genus_bound,sub:nodal_reduction}. 
On one hand, \Cref{proposition:genus_bound} (the genus bound) says
\[
	p_a(X)+T\leq g
\]
while \Cref{proposition:dimension_bound} (the dimension bound) says
\[
	g\leq p_a(X)+T
\]
Hence, both equalities have to hold! We can now use the equality conditions to describe
the central fiber $C_0$ and the way the A and B-sections meet it.

\begin{figure}
\centering
\includegraphics{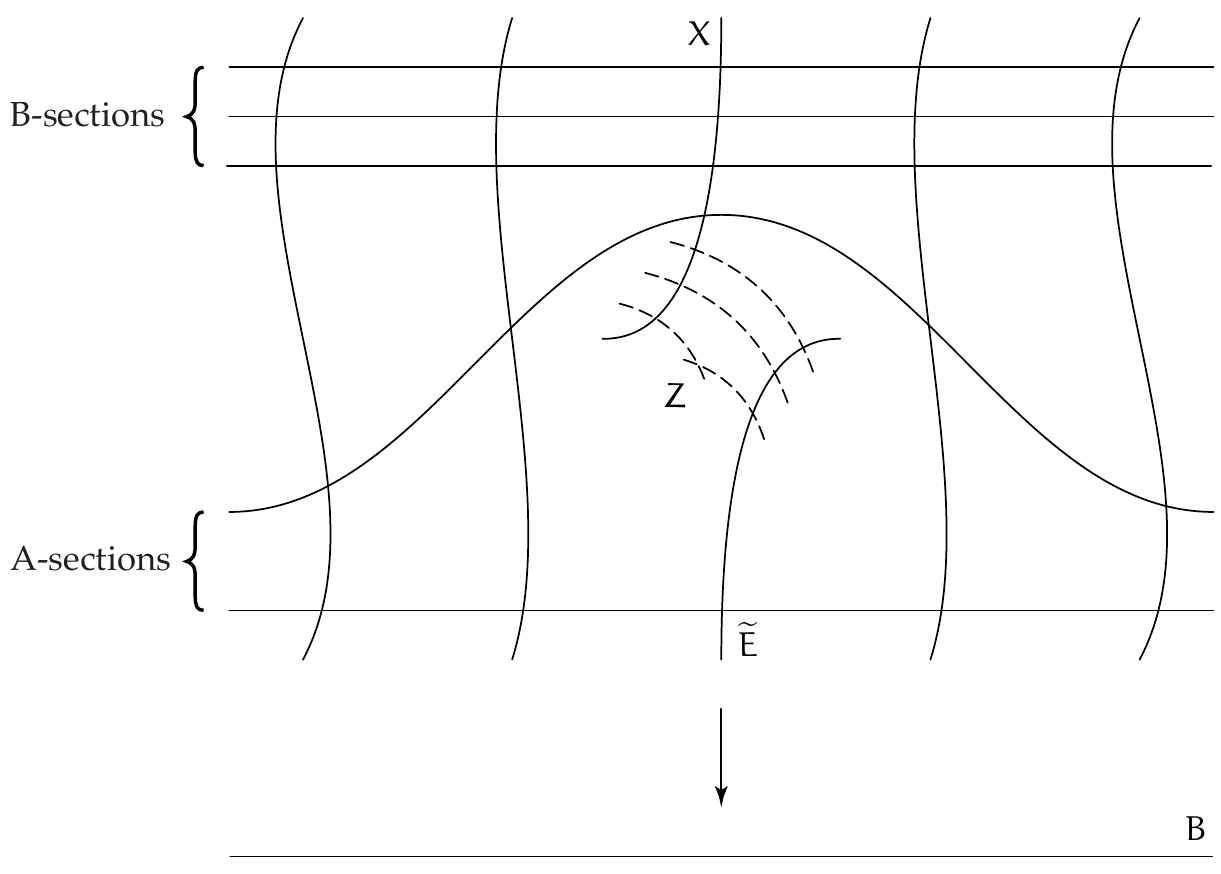}
\caption{An example of a central fiber, and of the first of the two possible scenarios of how the sections meet the central
fiber---never on the $Z$ components, and none of the B-sections meets $\widetilde{E}$.}
\label{figure:complete degeneration picture - fig15a}
\end{figure}

\begin{figure}
\centering
\includegraphics{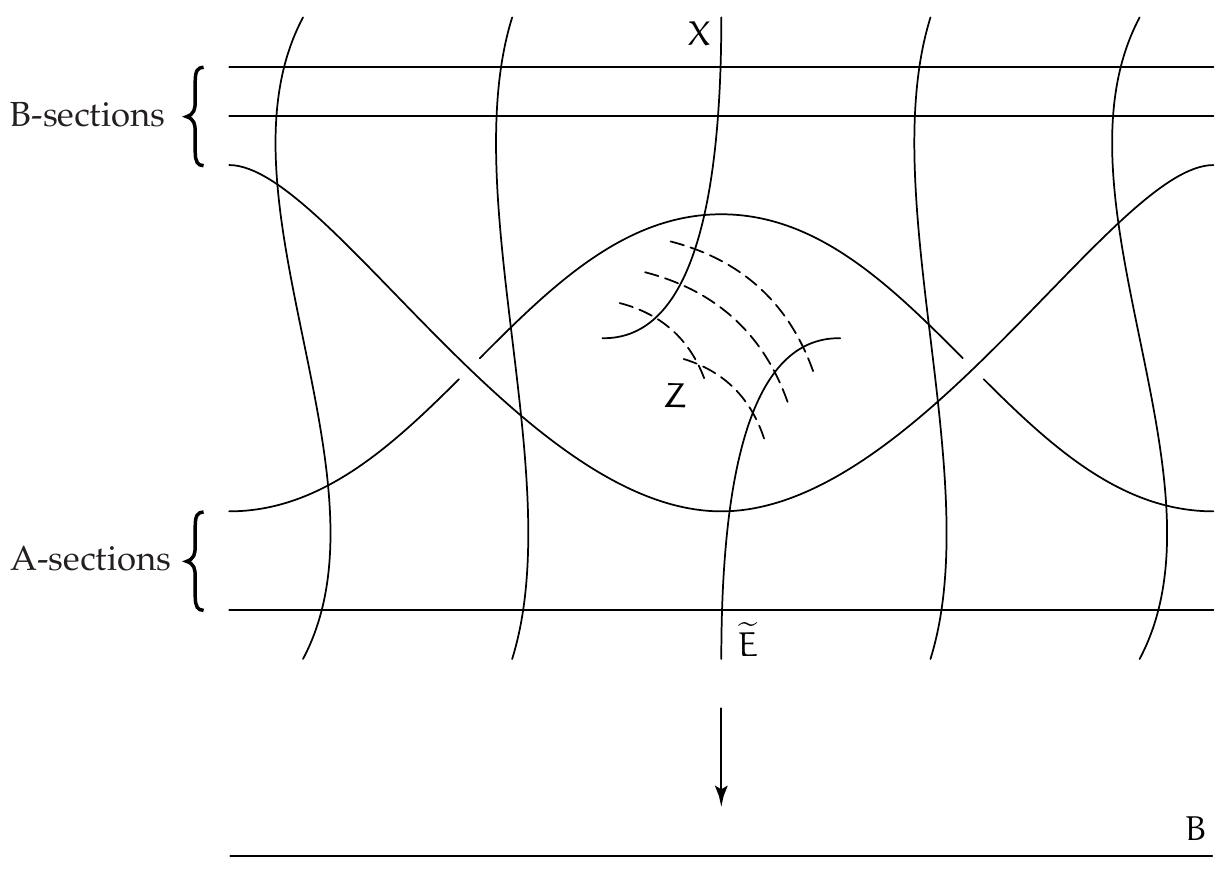}
\caption{The same example of a central fiber as in 
\cref{figure:complete degeneration picture - fig15a}, but with the 
second possibility for how the sections meet the central
fiber---again never on the $Z$ components, and with exactly one of the B-sections meeting $\widetilde{E}$.}
\label{figure:complete degeneration picture - fig15b}
\end{figure}

By \Cref{proposition:genus_bound}, the map $\widetilde{E} \to E$ is unramified, and
$Z$ is just a disjoint union of $T'$ rational chains connecting $X$ and $\widetilde{E}$, 
plus $(T-T')$ nodes between $X$ and $\widetilde{E}$. Let $\tau$ be a partition keeping track 
of the multiplicities with which each of the $T$ rational chains/nodes connecting $X$ and $\widetilde{E}$ 
appears in the pullback $\phi^*(E_0)|_X$.

By \Cref{proposition:dimension_bound}, none of the A and B-sections hit the central fiber
at $Z$. Let $a'$ be the number of A-sections that land in $X$. In particular, $a'\leq a$.

There are two possible scenarios for $X$, according to if there are any B-sections approaching 
$\widetilde{E}$ in the limit.

If there are B-sections approaching $\widetilde{E}$, then there is exactly
one such. In this case, the residual curve $X$ is a general point of a component of
the Severi variety
$\widetilde{V}^{\circ, \text{tails}}_{N'e+df,p_a(X)}\left(a',\tau+(b-1)\right)$.

On the other hand, if there are no B-sections approaching $X$, then $X$ is a general point of a component of
$\widetilde{V}^{\circ, \text{tails}}_{N'e+df,p_a(X)}(a',b,\tau)$.
See \cref{figure:complete degeneration picture - fig15a,figure:complete degeneration picture - fig15b} for 
diagrams for the family $\Cc \to B$ and its sections.

There is one last issue to be taken care of: we have to eliminate the
elliptic tails. This follows directly from Proposition \ref{proposition:tail}.
Note that this also rules out $\tau=(1)$---in this case, $\tilde{E}$ itself
would be an elliptic tail, which can't be smoothed.
This concludes the proof of \Cref{theorem:simple_hyperplane_section_of_severi}.

\begin{remark}
As opposed to the analogue arguments applied on rational surfaces \cite{caporaso_counting_1998,vakil_counting_2000}, 
our directrix $E_0$ can be split off generically with multiplicity greater than one. Compare
with \cite{shoval_gromov-witten_2013} where the ``directrix'' is still split off with multiplicity one,
but the residual curve can be non-reduced.
\end{remark}

To prove \Cref{theorem:simple_hyperplane_to_hurwitz}, we note that the
nodal reduction construction in \Cref{sub:nodal_reduction} already resolved the map
$\Vdg \rationalmap \overline{\Mc}_g(E,N)$ at the general point of $W$. 
By composing the map $\phi\from  C_0 \to E \times \Pbb^1$ with the projection 
$E \times \Pbb^1 \to E$, we get a map that is almost an element
of $\overline{\Mc}_g(E,N)$---we just have to contract the rational chains $Z$.
The description in \Cref{theorem:simple_hyperplane_to_hurwitz} follows.


\subsection{The general case} 
\label{ssub:the_general_case}
Let us tackle the problem of computing the degree of the Severi variety 
$V^\circ_{Ne+df,g}$. The strategy is to
intersect it with multiple hyperplanes of the form $H_p$, for general $p \in E_0$, until we get down
to some number of points, which we could count. Let $V$ be a component that arises in the process of
intersecting $\Vdg$ with hyperplanes. To succeed, we have to
answer the following two questions: \nobreakpar
\begin{itemize}
	\item what are the components of the intersection $V \cap H_p$? \nobreakpar
	\item with what multiplicity does each component appear?
\end{itemize}

We will deal with only the first question. \Cref{theorem:simple_hyperplane_section_of_severi}
describes the components of the intersections until the first time
the directrix $E_0$ is split off generically. To continue, we would need an analogue
of \Cref{theorem:simple_hyperplane_section_of_severi} for varieties like 
$\widetilde{V}^\circ_{N'e+df,p_a(X)}(a',b,\tau)$. 

Fortunately, the argument above readily adapts to this more general situation, and all 
the components that ever arise in this process can be described in similar tangency terms.
We will follow an analogue of Caporaso--Harris notation for these new components.
The key difference in our case is that, since $E$ has non trivial Picard group,
we are also forced to remember the classes of some of the moving points, as
we already had in the case of $\widetilde{V}^\circ_{N'e+df,p_a(X)}(a',b,\tau)$.

\begin{definition}
Fix line bundles $\Lc, L_j \in \Pic E_0$, points $p_1,\ldots, p_k \in E_0$,
and tangency profiles
$\alpha$ and $\beta^1,\ldots, \beta^\ell$, such that $m(\beta^j)  = \deg L_j$ and
\[
	\Lc = \sum_{i=1}^k \alpha_i p_i + \sum_{j=1}^\ell L_j
\]
where $\alpha = (\alpha_1\ldots \alpha_k)$.
We define the generalized Severi variety 
\[ 
\widetilde{V}^\circ_{Ne+df,g}(\alpha,\beta^1,\beta^2,\ldots, \beta^\ell)
[p_1,\ldots ,p_k, L_1,\ldots, L_\ell]	
\]
as the normalization of the closure of the locus of reduced (possibly reducible) curves in $E\times \Pbb^1$, of
class $\pi_1^*\Lc\otimes \pi_2^*(\Oc_{\Pbb^1}(d))$, geometric genus $g$, containing no fibers of $E \times \Pbb^1 \to \Pbb^1$, and whose intersection with $E_0$ 
consists of $\card{\alpha}+\sum_j \card{\beta^j}$ points as follows:
\begin{itemize}
	\item The fixed points $p_1, \ldots, p_k$, each appearing with the corresponding multiplicity $\alpha_1, \ldots, \alpha_k$
	 (note that $\card{\alpha}=k$),
	\item and for each $j$, $\card{\beta^j}$ points with multiplicities described
	by $\beta^j$, and whose sum is linearly equivalent to $L_j$.
\end{itemize}

We will omit the brackets when possible, and just write
$\widetilde{V}^\circ_{Ne+df,g}(\alpha, \beta^1,\ldots, \beta^\ell)$.
\end{definition}

Note that by an argument similar to \Cref{theorem:dimension-of-severi}, we have
\begin{equation}
\label{equation:dimension-bound-general}
	\dim \widetilde{V}^\circ_{Ne+df,g}(\alpha, \beta^1,\ldots, \beta^\ell)= d+g+\sum_j \card{\beta^j}-1-\ell
\end{equation}
and its generic point corresponds to a nodal curve.

\begin{remark}
If for some $j$ we have $\beta^j=(b)$, then the generalized Severi variety is a disjoint union
of $b^2$ different generalized Severi varieties, where we remove $\beta^j$ and add an extra $\alpha_{k+1}=b$, 
with the corresponding point $p_{k+1}$ being a $b$-th root of $L_j$. Hence, we may assume that $\card{\beta^j} \geq 2$
for all $j$.
\end{remark}

We can now describe the components that will ever arise in the process of intersecting with the hyperplanes $H_p$'s.
\begin{theorem}
\label{theorem:hyperplane_section_of_severi}
Let $p$ be a generic point on $E_0$, and $W$ a component of the intersection 
$\widetilde{V}^\circ_{Ne+df,g}(\alpha,\beta^1\ldots, \beta^\ell) \cap H_p$. Assume that $|\beta_j| \geq 2$
for all $j$. Then one of the following holds.
\begin{itemize}
	\item Either $W$ is a component of
	\[
		\widetilde{V}^\circ_{Ne+df,g}(\alpha+(n),\beta^1,\ldots,\beta^j-(n),\ldots, \beta^\ell)
		[p_1,\ldots,p_k,p,L_1,\ldots,L_j-p,\ldots L_{\ell}]
	\]
	for some $j$. That is, one of the moving points became a fixed at $p$. 
	\item Or the elliptic fiber $E_0$ is generically split off in the limit with
	multiplicity $m$. As a parameter space for the residual curve, $W$ is a component of 
	\[
		\widetilde{V}^\circ_{(N-m)e+df,g'}(\alpha',\beta^{j_1},\ldots,\beta^{j_t}, \beta')
		[p_{i_1},\ldots,p_{i_{\card{\alpha'}}}, L_{j_1},\ldots, L_{j_t},L']
	\]
	where 
	\begin{enumerate}
		\item $J=\set{j_1,\ldots, j_t}\subset[\ell]$ is a subset
		\item for $j\not \in J$, let $\widehat{\beta^j} \subset \beta^j$ be a 
		tangency profile with one less term.
		\item $\tau$ is a tangency profile such that $\card{\tau} \geq 2$ and
		\[
			g=g'+\card{\tau}
		\]
		\item $\beta$ is a tangency profile such that
		\[
			\beta'=\tau+ \sum_{j \not \in J} \widehat{\beta^j}
		\]
		\item and $\alpha' \subset \alpha$ is a subprofile such that
		\[
		 	d=m(\alpha') +m(\beta')+\sum_{j \in J}m(\beta^j)
		 \] 
	\end{enumerate}
	
	That is, for each group $\beta^j$ of B-sections, at most one of them may specialize to
	the curve dominating $E_0$. All the others stay in the residual curve. 
	If the group of $\beta^j$ remains intact, then its class is preserved.
	Otherwise, all the $\beta^j$ that lost an element bundle up together, plus
	a partition $\tau$ that remembers how the residual curve meets the cover of
	$E_0$ in the stable limit.
\end{itemize}
\end{theorem}
\begin{proof}
The argument will be very similar to the one for the simpler case of 
\Cref{theorem:simple_hyperplane_section_of_severi}. 
As a matter of fact, the proof  goes through mostly word by word---the only difference is that we have to do more bookkeeping 
to keep track of all the tangency profiles $\alpha, \beta^j$. Let us just
highlight the differences, and leave the details to the reader.

\Cref{sub:nodal_reduction,sub:the_genus_bound} go through as they are. In
particular, we will use the same definition of $\widetilde{E}, Z, X$ and $T$. 
We obtain \Cref{proposition:genus_bound} (the genus bound):
\[
	p_a(X)+T \leq g	
\]

In \Cref{sub:the_dimension_bound}, the key adjustment is our definition of
$\epsilon_b$ and $\epsilon$. 
We let $\epsilon_j$ be the number of sections in the $\beta^j$ group which meet
the central fiber on $\widetilde{E}$ or on a connected component of $Z$ that does
not meet $X$. We define 
\[
\epsilon = \card{ \set{j \text{ such that } \epsilon_j \neq 0} }\\
\]

From the definition of $\epsilon$, we get the following.
\begin{claim}
\label{claim:epsilon}
We have
\[
\epsilon \leq \sum_{j=0}^\ell \epsilon_j
\]
with equality only if, for each $j$, $\epsilon_j$ is equal to zero or one.
\end{claim}

And the corresponding \Cref{claim:dim-P} is the following.
\begin{claim}
\label{claim:general-P}
We have
\[
	\dim P \leq T+\sum_j(\card{\beta^j} 
	-\epsilon_j) - (\ell -\epsilon +1) 
\]
\end{claim}

We get the dimension bound
\[
	g \leq p_a(X) +T
\]
by coupling \Cref{claim:general-P,claim:epsilon,claim:deformation-bound},
and the following consequence of \cref{equation:dimension-bound-general}:
\[
	\dim W = \dim V_{df+Ne,g}(\alpha,\beta)-1=d+g+(\sum|\beta^j|)-1-\ell-1 
\]

Again, the dimension bound balances exactly the genus bound, and equalities
hold throughout. Analyzing the equality conditions,
and eliminating elliptic tails by \Cref{proposition:tail},
we arrive at the list of \Cref{theorem:hyperplane_section_of_severi}.
\end{proof}


\section{Monodromy Groups of Covers} 
\label{sec:monodromy_groups_of_covers}
In this section we will study the monodromy groups of simply branched covers.

Let $f\from C \to D$ be a degree $d$ cover of  smooth connected curves. Let 
$\punctured{D} \subset D$ be the
complement of the set of branch points of $f$, and $\punctured{C}$ its preimage in $C$.
We will deal with fundamental groups, so let us fix base points 
$\basepoint \in \punctured{D}$ and $\basepoint \in \punctured{C}$ that
map to each other under $f$.

The map $f\from \punctured{C} \to \punctured{D}$ is a covering of topological spaces. The 
fundamental group $\pi_1(\punctured{D},\basepoint)$ acts on the fiber 
$f^{-1}(\basepoint)$. This action defines a \emph{monodromy map}
$\phi\from \pi_1(\punctured{D},\basepoint) \to S_d$.
We call the image of $\phi$ the monodromy group of the cover $f$.
Note that it is only defined up to conjugation in $S_d$. If $\phi$ is surjective,
we say that $f$ has \emph{full} monodromy.

A driving question in the field is 
\emph{what are the possible monodromy groups $G$}. 
There are many interesting (and some unsolved) problems along this line, see for example, \cite{guralnick_monodromy_2003,artebani_algebraic_2005,frohardt_composition_2001,guralnick_rational_2003}.
However, when the covering map $f$ is simply-branched, one can say much more.

\begin{proposition}[Berstein--Edmonds]
\label{proposition:full_monodromy}
Let $f\from C \to D$ be a simply branched cover. Then $f$ has full monodromy if,
and only if, $f$ is primitive (that is,
the pushforward map in fundamental groups $\pi_1(C,\star) \to \pi_1(D,\star)$
is surjective).
\end{proposition}
This is Proposition 2.5 of \cite{berstein_classification_1984}. One can find a proof of this in 
\cite{kanev_irreducibility_2005} as well. For completeness, here is a modified proof.
\begin{proof}
From a factorization of $f\from C \to D$ through a non-trivial unramified map $\tilde{D} \to D$, we
get a partition of $f^{-1}(\star)$ according to the different images in $\tilde{D}$. The monodromy of
$f$ has to preserve this partition, and hence cannot be the full symmetric group.

Suppose now that $f$ is primitive. We want to show it has full monodromy.
Let $K$ be the kernel of the map 
$i_*\from \pi_1(\punctured{D},\basepoint) \to \pi_1(D,\basepoint)$. We
start by showing the following.
\begin{claim}
The action of $K$ on $f^{-1}(\basepoint)$ is transitive.
\end{claim}
\begin{proof}
Consider the following commutative diagram
of groups:
\[
\xymatrix{
	& \pi_1(\punctured{C},\basepoint)\ar@{->>}^{i_*}[r] \ar^{f_*}[d]& 
	\pi_1(C, \basepoint)\ar@{->>}^{f_*}[d] \\
	K\ar[r] & \pi_1(\punctured{D},\basepoint)\ar@{->>}^{i_*}[r] & \pi_1(D, \basepoint)
}
\]

Given any element $g \in \pi_1(\punctured{D},\basepoint)$, we can use the surjectivity of the top and right arrows
 to find
an element $\lift{g} \in \pi_1(\punctured{C}, \basepoint)$ such that
\[
	i_*(g)=f_*i_*(\lift{g}) = i_*(f_*(\lift{g}))
\]
Hence,
\[
	g \cdot \left(f_*(\lift{g})\right)^{-1} = \kappa \in K
\]

Now let us consider
the monodromy action of $g$ on the fiber $f^{-1}(\basepoint)$. More specifically, let us
see how $g$ acts on the base point $\basepoint \in \punctured{C}$ (which in turn lies over the
base point $\basepoint \in \punctured{D}$). 

If we lift a representative loop of $f_*(\lift{g})$ with start point $\basepoint$, we
will get a representative loop of $\lift{g}$. Hence, the monodromy action of $f_*(\lift{g})$ 
fixes $\basepoint$. As $g= \kappa \cdot f_*(\lift{g})$, we get
$g(\basepoint) = \kappa(\basepoint)$.

We showed that for any $g \in \pi_1(\punctured{D},\basepoint)$, there is a $\kappa \in K$
that acts on $\basepoint$ in the same way. Hence, the orbit of $\basepoint$ under the $K$-action is equal to the orbit under the $\pi_1(\punctured{D},\basepoint)$-action. But
the latter action is transitive, since $\punctured{C}$ is connected. Hence, the
action of $K$ is transitive as well.
\end{proof}


\begin{claim}
The image of $K \subset \pi_1(\punctured{D},\basepoint) \to S_d$ is generated
by transpositions.
\end{claim}
\begin{proof}
We use the standard presentations of $\pi_1(D,\basepoint)$ and
$\pi_1(\punctured{D},\basepoint)$.
Let $g(D)$ be the genus
of $D$. Then $\pi_1(D,\basepoint)$ is the free group generated by 
$\alpha_i, \beta_i$  for $i$ ranging from $1$ to $g(D)$, subject to the relation
\[
	[\alpha_1,\beta_1]\ldots[\alpha_{g(D)}, \beta_{g(D)}] = \Id
\]

Let $b$ be the number of branch points of $f\from C\to D$, and 
choose simple contractible loops $\tau_j$ around each branch point. Then
$\pi_1(\punctured{D},\basepoint)$ is the free group generated by $\alpha_i,\beta_i$ for $i=1,2,\ldots, g(D)$
and $\tau_j$ for $j=1,2,\ldots,b$, subject to the relation
\[
[\alpha_1,\beta_1]\ldots[\alpha_{g(D)}, \beta_{g(D)}] = \tau_1 \tau_2 \ldots \tau_b 	
\]

The map $\pi_1(\punctured{D},\basepoint) \to \pi_1(D , \basepoint)$ sends the
$\alpha_i,\beta_i$ to the corresponding $\alpha_i,\beta_i$, and
each $\tau_j$ to $\Id$. Hence, the kernel $K$ is the normal closure of
the subgroup generated by the $\tau_j$.

Since $f\from C \to D$ is simply branched, each $\tau_j$ acts as a transposition. Hence
so will any conjugate of $\tau_j$. Therefore, the image of $K$ in $S_d$ is generated
by transpositions, as we wanted to show.
\end{proof}

Now we can finish the proof of \Cref{proposition:full_monodromy}. 
The image of $K$ in $S_d$ is transitive and generated by transpositions. But the only
subgroup of $S_d$ satisfying these two properties is $S_d$ itself. Hence, $K \to S_d$ is surjective, and so is the monodromy map $\phi\from \pi_1(\punctured{D},\basepoint) \to S_d$.
\end{proof}

\begin{corollary}
\label{corollary:primitive-automorphism}
A primitive degree $d$ simply branched cover $f:C \to D$ does not admit any isomorphisms over $D$,
unless $d=2$.
\end{corollary}
\begin{proof}
This is similar to Proposition 1.9 in \cite{fulton_hurwitz_1969}.
An isomorphism of $C$ over $D$ corresponds to an element $\sigma$ of $S_d$ that commutes with every 
element of the monodromy group $\im \phi \subset S_d$.

By \Cref{proposition:monodromy-group}, 
if a cover is primitive, then the monodromy map is surjective, and therefore $\sigma$ would be in the
center of $S_d$, which is the trivial group as soon as $d \geq 3$.
\end{proof}

An  arbitrary cover may be related to primitive ones by the following construction.
\begin{corollary}
\label{corollary:factorization}
Let $f\from C \to D$ be a simply branched covering map. Then
it can be factored in a unique way as 
$C \to \Dbar \to D$
such that $C \to \Dbar$ has full monodromy, and
$\Dbar \to D$ is unramified.
\end{corollary}
\begin{proof}
The unramified cover $\Dbar \to D$ is the covering space corresponding
to the subgroup $\im \pi_1(C,\basepoint) \subset \pi_1(D,\basepoint)$.
By the lifting property, the map $C \to D$ factors through $\Dbar$. The
map $C \to \Dbar$ is surjective on $\pi_1$. Hence, by \Cref{proposition:full_monodromy},
it has has full monodromy.
\end{proof}

\Cref{corollary:factorization} gives us a tool to study the monodromy group
of an arbitrary simply branched cover $f\from C \to D$. 
We would hope to relate the monodromy
group of $f$ with the monodromy groups of each factor, which are easier to understand.
The following proposition is a first step in this direction.

\begin{proposition}
\label{proposition:monodromy-group}
Let $C \to \Dbar \to D$ be as in \Cref{corollary:factorization}.
Let $G$ (resp. $\Gbar$)  be the monodromy group of $C \to D$ 
(resp. $\Dbar \to D$). The following holds.
\begin{enumerate}
	\item \label{item:part1} There is a surjective map $G \onto \Gbar$,
	\item There kernel of $G \onto \Gbar$ is as large as it can be. That is, the 
	monodromy group $G$ fits in a exact sequence
\[
	1 \to S_{\tilde{d}} \times \ldots \times S_{\tilde{d}} \to G \to \Gbar \to 1
\]
where the product in the left term has $\bar{d}=\deg (\Dbar \to D)$ factors, each
of them isomorphic to the symmetric group on
$\tilde{d}=\deg (C \to \Dbar)$ elements.
\end{enumerate}
\end{proposition}
\begin{proof}
\begin{figure}
\centering
\includegraphics{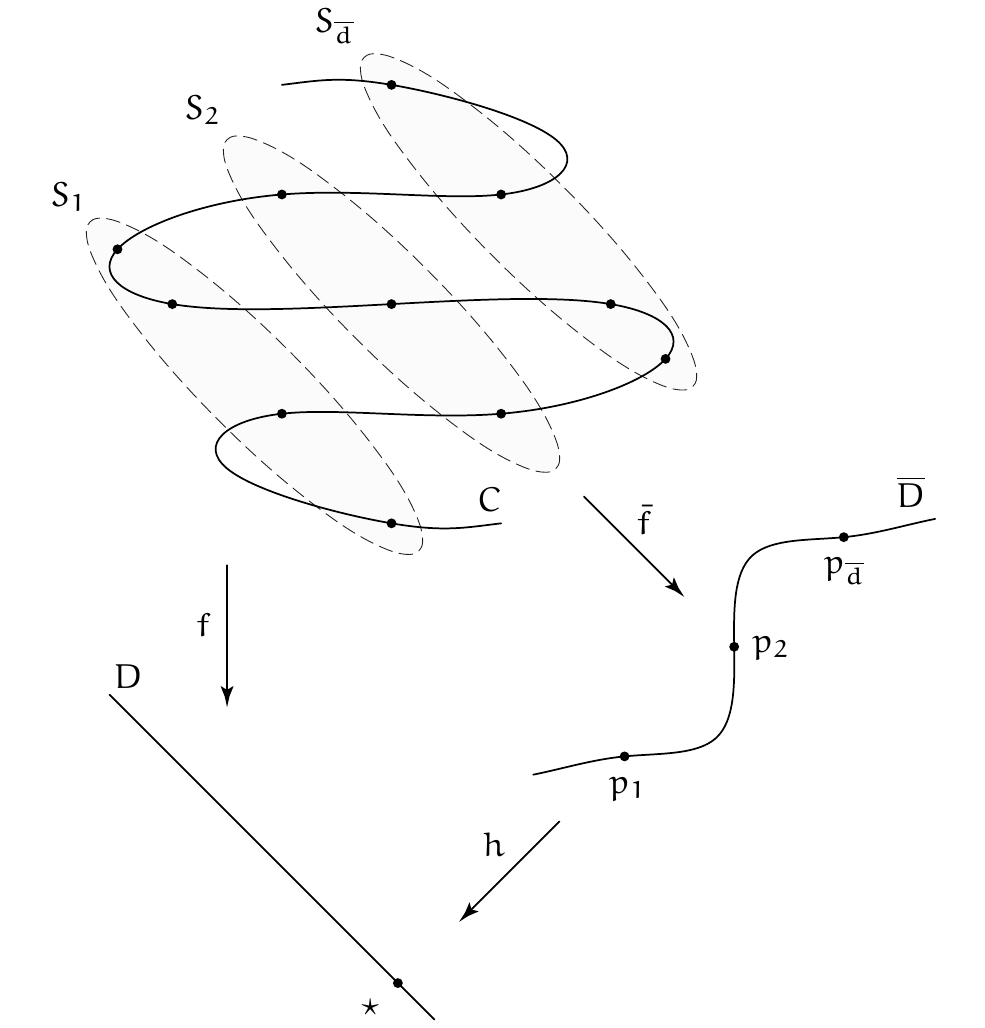}
\caption{The preimages of $\basepoint$ in $\Dbar$ and $C$.}
\label{figure:notation for monodromy argument - fig09}
\end{figure}

We will define a map $G \to \Gbar$ such that the monodromy maps are compatible. That is,
we want the following diagram to commute
\[
	\xymatrix{\pi_1(\punctured{D}, \basepoint) \ar[r] \ar[d]^{i_*} & G\ar@{-->}[d] \\
			  \pi_1(D, \basepoint) \ar^{\overline{\phi}}[r] & \Gbar }
\]

The map is defined as follows. Given a monodromy element $g \in G$, 
choose a representative loop $\gamma\from I \to \punctured{D}$ inducing it.
We send $g$ to $\gbar=\overline{\phi}(i_*(\gamma))$. Geometrically, $\gbar$ is the
 monodromy element induced by the intermediate lifts of $\gamma$ to $\Dbar$:
\[
\xymatrix{
	& C \ar[d] \\
	& \Dbar \ar[d] \\
	I \ar[r]^\gamma
	\ar@{-->}[ru]
	\ar@{-->}[ruu] & D
}
\]

The map $G \to \Gbar$ is well defined because if a loop $[\gamma] \in \pi_1(\punctured{D}, \basepoint)$
induces the trivial monodromy on the fibers of $C \to D$, it also induces the trivial monodromy on the fibers of the
intermediate map $\Dbar \to D$.






Since $i_*\from  \pi_1(\punctured{D}, \basepoint) \to \pi_1(D,\basepoint)$ and
$\overline{\phi}\from \pi_1(D,\basepoint) \to \Gbar$
 are surjective, so is the map $G \to \Gbar$. This establishes \cref{item:part1}.

Let us introduce some notation. 
Let $\fbar\from C \to \Dbar$ and $h\from \Dbar \to D$ be the
factors of the factorization of $f\from C \to D$. 
Call $p_1,\ldots,p_{\bar{d}} \in \Dbar$ 
the points in the preimage of $\basepoint \in D$, and let $S_i = \fbar^{-1}(p_i) \subset C$.
See \cref{figure:notation for monodromy argument - fig09}.

Let $K$ be the kernel of $G \onto \Gbar$. The elements of $K$ preserve the sets 
$S_i$ for all $i$. That is, 
	$K \subset \Sym(S_1) \times \Sym(S_2) \times \ldots \times \Sym(S_{\bar{d}})
	\subset \Sym(\cup S_i)$.
We want to show that $K = \Sym(S_1) \times \Sym(S_2) \times
 \ldots \times \Sym(S_{\bar{d}})$.

Pick a simple loop $\gamma \in \pi_1(\punctured{D},\basepoint)$, contractible in $D$,
around one of the branch points. Let $\tau= \phi(\gamma)\in G$ be the corresponding monodromy element. 
It is a transposition, since
the branching is simple. Moreover, since $\gamma$ is contractible in $D$, the transposition 
$\tau$ maps to the identity in $\Gbar$. Hence $\tau \in K$. The transposition $\tau$ has 
to act trivially in all but one of the $S_i$. Say it acts non-trivially on $S_1$.
By a slight abuse of notation, we will say that $\tau \in \Sym(S_1)$.

We will exploit the action of $G$ on $K$ by conjugation. First, consider 
the composition of the pushforward
$h_*:\pi_1(\punctured{\Dbar}, p_1) \to \pi_1(\punctured{D},\basepoint)$
with the monodromy map $\phi_f:\pi_1(\punctured{D},\basepoint) \to G$. 
The composite map $\phi_f \circ h_*$ fixes $S_1$, and hence it factors through
$\Sym(S_1) \times \Sym(S_2 \cup \ldots \cup S_{\bar{d}})$. Projecting on the
$\Sym(S_1)$ factor, we recover the monodromy map of $C \to \Dbar$,
as in the following diagram.
\[
	\xymatrix{
	& \Sym(S_1) \\
	\pi_1(\punctured{\Dbar}, p_1) \ar@{->>}^{\phi_{\fbar}}[ru] \ar@{-->}[r] \ar[d]^{h_*} & 
	\Sym(S_1)\times \Sym(S_2 \cup \ldots \cup S_{\bar{d}}) \ar[u] \ar[d] \\
	\pi_1(\punctured{D}, \basepoint) \ar[r]^(.35){\phi_f} & G \subset\Sym(S_1 \cup S_2 \cup \ldots \cup S_{\bar{d}})
	}
\]

Pick $\gamma \in \pi_1(\punctured{\Dbar}, p_1)$.
Conjugating an element of $\Sym(S_1) \cap K$ by $\phi_f(h_*(\gamma))$ is the same as conjugating by $\phi_{\fbar}(\gamma) \in \Sym(S_1)$. As $C \to \Dbar$ has full monodromy,
we can get an arbitrary conjugate in $\Sym(S_1)$.

We already have a transposition $\tau \in \Sym(S_1) \cap K$. Hence, we can get \emph{all}
transpositions in $\Sym(S_1)$. But these generate the full symmetric group. Hence,
$\Sym(S_1)\times \Id \times \Id \times \ldots \times \Id  \subset K$.

For each $i$, pick $\gbar \in \Gbar$ that sends $p_1$ to $p_i$. Lift $\gbar$ to $g \in G$.
Then
\[
	g \Sym(S_1) g^{-1} = \Sym(S_i) \subset K
\]
As this holds for any $i$, we get
\[
	\Sym(S_1) \times \Sym(S_2) \times \ldots \times \Sym(S_{\bar{d}}) \subset K
\]
as we wanted to show.
\end{proof}

Note that the exact sequence in \Cref{proposition:monodromy-group} does not completely characterize 
the monodromy group $G$---we do not know what extension it is. Determining what
extensions do arise as monodromy groups of simply branched maps is an interesting question, 
which, as far as I know, is unsolved.

We now apply \Cref{proposition:monodromy-group} to establish the result we will
need in \Cref{sec:irreduciblity_of_hurwitz_spaces}.
\begin{lemma}
\label{lemma:fiber product}
Let $C$ be a smooth curve, $C \to \Dbar$ be a simply branched map 
with full monodromy, and $\Dbar \to D$ an unramified map.
Consider the map
\[
C\times_D C -\Delta \to \Dbar \times_D \Dbar
\]
The target $\Dbar \times_D \Dbar$ is smooth, but possibly disconnected.
 Nevertheless, the preimage of each component is irreducible.
\end{lemma}
\begin{proof}
Let $\Gamma \subset \Dbar \times_D \Dbar$ be a component. We want to show 
that its preimage is connected.

If $\Gamma=\Delta$ is the diagonal, having connected preimage is equivalent to the
monodromy group of $C \to \Dbar$ being 2-transitive. This is immediate, since
the monodromy group is actually the full symmetric group.

Suppose now $\Gamma$ is not the diagonal. For convenience, let us use the notation
as in the proof of \Cref{proposition:full_monodromy}. Pick a point in the fiber over 
$\basepoint$ of the map $\Gamma \to D$, say $(p_i,p_j)$. For any $q_i, q'_i \in S_i$
and $q_j, q_j' \in S_j$, we want to exhibit a path in $C \times_D C$ connecting 
$(q_i,q_j)$ to $(q'_i, q'_j)$. It is enough to find an element in the monodromy group
of $C \to D$ that sends $q_i$ to $q'_i$, and $q_j$ to $q'_j$. This follows
from \Cref{proposition:full_monodromy}, since the monodromy group
contains $\Sym(S_i) \times \Sym(S_j)$.
\end{proof}



\section{Irreducibility of the Hurwitz Space 
\texorpdfstring{$\Hc^0_{N,g}(E)$}{of Primitive Covers}} 
\label{sec:irreduciblity_of_hurwitz_spaces}

Let us resume our study of Hurwitz spaces. Recall their definition.
\begin{definition}
The Hurwitz space $\Hc_{N,g}(E)$ parametrizes simply branched covers (meaning, dominant finite flat morphisms)
 $C \to E$ from a smooth genus $g$ curve $C$  to the fixed curve $E$.
 The space $\Hc^0_{N,g}(E)$ is the open and closed subscheme of $\Hc_{N,g}(E)$
 of primitive covers $C \to E$ (that is, covers which do not factor through any non-trivial unramified maps 
 $\tilde{E} \to E$).
 \end{definition}

 There are several alternate characterizations of $\Hc^0_{N,g}(E)$. 
 For example, as a consequence of 
 \Cref{proposition:full_monodromy} and \Cref{corollary:factorization}, we have the following useful one.
 \begin{proposition}
 \label{proposition:hurwitz-full-monodromy}
 The space $\Hc^0_{N,g}(E) \subset \Hc_{N,g}(E)$ consists of the covers 
 $C \to E$ with full monodromy.
 \end{proposition}

Our goal is to prove \Cref{theorem:hurwitz_irreducible}, which says that 
for a smooth genus one curve $E$ and $g>1$, the Hurwitz space $\Hc^0_{N,g}(E)$
is irreducible. Note that from the irreducibility of $\Hc^0_{N,g}(E)$, we can characterize all the 
irreducible components of $\Hc_{N,g}(E)$.

\begin{corollary}
\label{corollary:components-of-HcNgE}
	The components of $\Hc_{N,g}(E)$ are in bijection with isogenies $\tilde{E} \to E$ of degree diving $N$. That is, the decomposition in irreducible components is
	\[
		\Hc_{N,g}(E) = \bigsqcup_{\tilde{E}\to E} \Hc_{N',g}^0(\tilde{E})
	\]
	where $N'=N/\deg(\tilde{E} \to E)$, for $N' \geq 2$.
\end{corollary}
\begin{proof}
The right hand side is included in the left by the map
\[
	(C \to \tilde{E}) \mapsto (C \to \tilde{E} \to E)
\]

\Cref{corollary:factorization}
says that the opposite inclusion holds, 
and \Cref{theorem:hurwitz_irreducible} says every term in the 
right is irreducible.
\end{proof}

\subsection{Our strategy}
\label{sub:our-strategy}

As a motivating example,
let us go over a simple argument by Fulton for 
the irreducibility of $\Mc_g$ (in the appendix of \cite{harris_kodaira_1982}).
 We follow Harris and Morisson account in \cite{harris_moduli_1998}.

The idea, which really goes back to Deligne and Mumford, 
is to look instead at the compactified $\overline{\Mc_g}$
and leverage the structure of the boundary to prove irreducibility.
More precisely, assume by induction that $\Mc_{g'}$ is irreducible
for $g'<g$. Then, by induction, one can establish that each boundary divisor
$\Delta_i$ in $\overline{\Mc_g}$ is irreducible. Moreover, one can prove that 
any pair of boundary divisors meet, simply by exhibiting an element in the
intersection $\Delta_i \cap \Delta_j$. Hence, there is a unique
connected component of $\overline{\Mc_g}$ meeting the boundary $\Delta$.

However,  $\overline{\Mc_g}$ is locally irreducible, and therefore connectedness implies irreducibility. 
Hence, we established the following reduction.

\begin{reduction*}
It is enough to show that every component $\Gamma$ of 
$\overline{\Mc_g}$ meets the boundary $\Delta$.
\end{reduction*}

To exhibit such a curve, we will pass to a more structured space:
the Hurwitz space $\Hc_{d,g}(\Pbb^1)$ of degree $d$ covers  $C \to \Pbb^1$.
It follows from Riemann--Roch that for large $d$ the Hurwitz space dominates
$\Mc_g$. 

If we were willing to use that the Hurwitz space is irreducible, we
would conclude irreducibility of $\Mc_g$ as well. However, using the reduction above, 
we can still prove irreducibility of 
$\overline{\Mc_g}$ without invoking the irreducibility of $\Hc_{d,g}(\Pbb^1)$. 
Instead, we will use the properness of the admissible covers compactification 
$\overline{\Hc_{d,g}}(\Pbb^1)$.

The branch morphism extends to the boundary of the admissible covers compactification, 
and every component of $\overline{\Hc_{d,g}}(\Pbb^1)$ dominates the Fulton--MacPherson configuration space 
$\Pbb^1[b]$. By choosing an extremal 
configuration of branch points, we can force the cover to be a reducible curve whose all irreducible components are rational. In particular, it is an
element of $\Delta$. For details, see \cite[Chapter 6A]{harris_moduli_1998}.

We will apply a similar argument to show the irreducibility of $\Hc_{N,g}^0(E)$.
Again, the thrust of the argument comes from exploiting boundary divisors
in a compactification of $\Hc_{N,g}^0(E)$. So let us start by choosing one.

A natural candidate is to use the Kontsevich space $\overline{\Mc}_{g}(E,N)$.
 However, this can have bad singularities at the locus of stable maps with contracted components, 
 and it is not clear that it is locally irreducible, 
 which is key to an argument like the one above.

We will instead use a partial compactification which is smooth, but still 
allows enough singular covers $X \to E$ for our reduction 
argument to go through.

\begin{definition}
Let $\widetilde{\Hc}_{N,g}(E) \subset \overline{\Mc}_{g}(E,N)$ to be the open substack 
parametrizing finite maps (that is, the complement of the locus where there are contracted components).
Points in $\widetilde{\Hc}_{N,g}(E)$ correspond to degree $N$, dominant, 
finite flat maps from a connected, arithmetic genus $g$ nodal curve $X$ to $E$.
There is an open and closed substack $\widetilde{\Hc}^0_{N,g}(E) \subset \widetilde{\Hc}_{N,g}(E)$ 
parametrizing primitive covers.
\end{definition}

\begin{remark}
Not being primitive is naturally a closed property, since if a family of covers 
$X_t \to E$ factors through $\tilde{E} \to E$ for generic $t$, then it factors
through
for all $t$. This implies that $\widetilde{\Hc}^0_{N,g}(E) \subset \widetilde{\Hc}_{N,g}(E)$ is an open substack.

However, it is also closed. 
Indeed, say $\tilde{E} \to E$ is a non-trivial isogeny, and
the composition $X_0 \to \tilde{E} \to E$ is a stable map which is a limit of primitive maps 
 $X_t \to E$. Then $(X_0 \to E)$ is a singular point of 
$\overline{\Mc}_g(E,N)$. Indeed, it lies in at least two components: the closure of 
$\Hc^0_{N,g}(E)$, and the closure of $\Hc^0_{N',g}(\tilde{E})$, where $N'$ is the 
degree of $X_0 \to \tilde{E}$. Both of these components are already $2g-2$ 
dimensional, so  $(X_0 \to E)$ is a singular point. But, as we will see below in 
\Cref{lemma:smooth}, ${\widetilde{\Hc}_{N,g}(E)}$
lives in the smooth locus of $\overline{\Mc}_{g}(E,N)$,
so we would not see such a $X_0 \to E$.
\end{remark}

\begin{lemma}
\label{lemma:smooth}
The space $\widetilde{\Hc}_{N,g}(E)$ is non-singular.
\end{lemma}
\begin{proof}
The tangent space at a point $f:C \to E$ is equal to $H^0(C, \Nc_f)$, where the \emph{normal sheaf} $\Nc_f$ is defined as
\[
	0 \to T_C \to f^* T_E \to \Nc_f \to 0
\]
When there are no contracted components, $\Nc_f$ is torsion. Hence, $H^1(C,\Nc_f)=0$, and Riemann--Roch says
$h^0(C, \Nc_f)=2g-2$, the length of the branch subscheme. As $\dim \widetilde{\Hc}_{N,g}(E)$ is $2g-2$ as well, 
we conclude it is non-singular.
\end{proof}

We will show that $\widetilde{\Hc}_{N,g}^0(E)$ is connected.
Of course, \Cref{theorem:hurwitz_irreducible} follows from this.
Our proof will be inductive.
\begin{inductionhypothesis}
\label{hypothesis:inductionHurwitz}
Suppose, by induction, that the Hurwitz space of primitive covers 
$\widetilde{\Hc}_{N,g'}^0(E)$ is connected for all genus one curves $E$, and 
integers $N\geq 2$ and $2\leq g'<g$.
\end{inductionhypothesis}

The base case $g=2$ is due to Kani in \cite{kani_hurwitz_2003}.
\begin{theorem*}[Kani]
 The Hurwitz space $\Hc^0_{N,2}(E)$ is irreducible.
\end{theorem*} 

We will assume $g\geq 3$ from now on.

As in Fulton's argument, let us take advantage of the boundary 
in our partial compactification.
Call $\Delta\subset \widetilde{\Hc}_{N,g}^0(E)$ the locus
of covers with singular source.

As opposed to the case of $\overline{\Mc_g}$, where there are relatively few
boundary divisors and they are easily understood by induction, our boundary $\Delta$ will
have quite a few boundary components, and
they are not so easily dealt with inductively. Nevertheless, we will prove the following.

\begin{finalreduction}
\label{reduction:final}
It is enough to show that every connected component of 
$\widetilde{\Hc}_{N,g}^0(E)$ contains a point in the boundary $\Delta$.
That is, it is enough to find, in each component, a cover with singular source. 
\end{finalreduction}

In \Cref{sub:reduction} we will prove \Cref{reduction:final} under the \Cref{hypothesis:inductionHurwitz}.
Then in \Cref{sub:the_degeneration_argument} we will establish that every component of Hurwitz space
does contain a point of $\Delta$. As in Fulton's argument, we have to pass to a more structured space
to find such a point. In our case, we will use the Severi variety of curves in $E \times \Pbb^1$. Specifically,
we will invoke \Cref{theorem:simple_hyperplane_to_hurwitz}.

\subsection{The reduction argument}
\label{sub:reduction}

In this section we will prove \Cref{reduction:final} under the \Cref{hypothesis:inductionHurwitz}. 
To do so, we will first focus on a few special boundary divisors, 
and then bootstrap to the whole boundary $\Delta$.

We define $D_{\text{sing}}\subset \widetilde{\Hc}_{N,g}^0(E)$ to be
 the closure of the locus of covers $X \to E$ with $X$ irreducible and with a single node, 
 for which the composition of the normalization map $\tilde{X} \to X$ and $X \to E$ is primitive.
 That is, $\tilde{X} \to E$ does not factor through non-trivial isogenies,
 as in the diagram in
  \cref{figure:general point of Dsing - fig02}.

\begin{figure}
\centering
\includegraphics{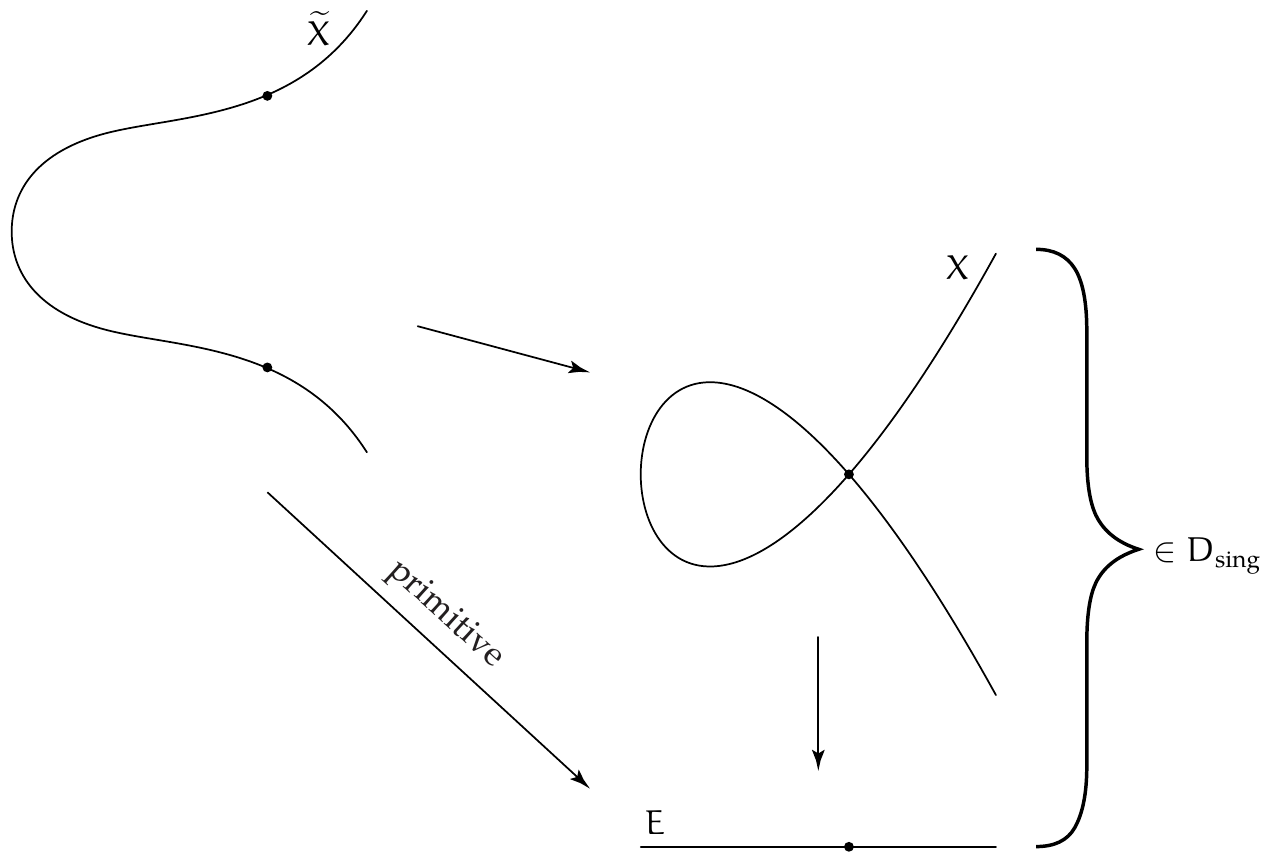}
\caption{General point of $D_{\text{sing}}$}
\label{figure:general point of Dsing - fig02}
\end{figure}

\begin{lemma}
\label{lemma:D_sing_irreducible}
Under \Cref{hypothesis:inductionHurwitz}, the set $D_{\text{sing}}$ is an irreducible divisor.
\end{lemma}
\begin{proof}
We may parametrize the points in $D_{\text{sing}}$ by the data of a primitive cover 
$\pi\from C \to E$
in $\Hc^0_{d,g-1}(E)$, and a pair of distinct points $p,q \in C$ in a fiber of $\pi$.
This data is recorded in the fiber product
$\Sigma=\Cc \times_E \Cc - \Delta$,
where $\Cc$ is the universal curve over $\Hc_{N,g-1}^0(E)$, and $\Delta$ is the diagonal.
 There is a birational isomorphism from 
$\Sigma$ to $D_{\text{sing}}$ that identifies the points $p$ and $q$. Note that
\[
	\dim D_{\text{sing}} = \dim \Sigma = 2(g-1)-2 +1 = 2g-3 = \dim \Hc^0_{d,g}(E) -1
\]
Hence $D_{\text{sing}}$ is a divisor.
Moreover, from the induction hypothesis, we know that $\Hc_{N,g-1}^0(E)$ is irreducible. 
From  \Cref{proposition:hurwitz-full-monodromy},
 the monodromy of $C\to E$ is
full, and therefore 2-transitive. Hence, $\Sigma$ is irreducible, and
so is $D_{\text{sing}}$.
\end{proof}

From the  \Cref{lemma:D_sing_irreducible}, we get our first reduction.

\begin{reduction}
\label{reduction:D-sing}
It is enough to show that every connected component of $\widetilde{\Hc}^0_{N,g}(E)$
contains a point in $D_{\text{sing}}$.
\end{reduction}
\begin{proof}
By  \Cref{lemma:smooth}, $\widetilde{\Hc}_{N,g}^0(E)$ is smooth,
and hence it is enough to show it is connected. However, by 
 \Cref{lemma:D_sing_irreducible}, 
only one connected component meets $D_{\text{sing}}$.
\end{proof}

Let us define a few more special boundary divisors.
For each isogeny ${\pi\from \tilde{E} \to E}$, let
$D_\pi \subset \widetilde{\Hc}_{N,g}^0(E)$
be the closure of the subscheme parametrizing covers with reducible source $X=X'\cup_p \tilde{E}$, with
 a single node $p$,
 such that $X' \to E$ is a primitive cover of degree $N-\deg \pi$, and genus $g-1$.
See  \cref{figure:general point of Dpi - fig03} for a picture of the general point of
$D_\pi$.

\begin{figure}
\centering
\includegraphics{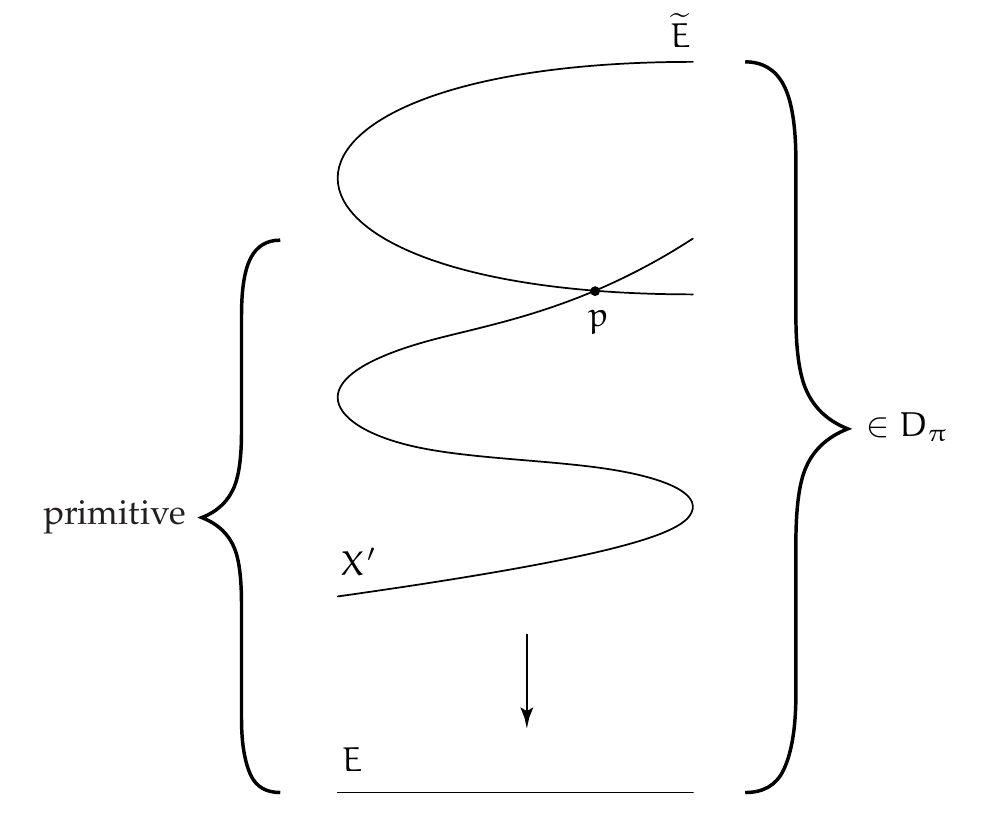}
\caption{General point of $D_{\pi}$}
\label{figure:general point of Dpi - fig03}
\end{figure}

\begin{remark}
Given two pointed covers $f_i\from (C_i,p_i) \to E$, for $i=1,2$, we can form a
cover $f\from  C \to E$ by \emph{attaching $C_2$ to $C_1$} in the following way. 
Take 
\[
C=C_1 \sqcup C_2/(p_1 \sim p_2)	
\]
and glue $f_1$ to $\tau \circ f_2$ where $\tau\from E \to E$ is the
translation which takes $f_2(p_2)$ to $f_1(p_1)$.

We may omit the point $p_i \in C$ in case $C_i \to E$ is an isogeny, since
then the choice of point is irrelevant.
\end{remark}

For example, a cover in $D_\pi$ is obtained by attaching the isogeny $\pi$
to a cover $X' \to E$ in $\Hc^0_{N-\deg \pi,g-1}(E)$.

\begin{lemma}
Under  \Cref{hypothesis:inductionHurwitz}, the set $D_{\pi}$ is an irreducible divisor.
\end{lemma}
\begin{proof}
There is a map from the universal curve 
$\Cc \to \widetilde{\Hc}^0_{N-\deg \pi,g-1}(E)$
to $D_\pi$, which takes a pointed cover $(X',p)\to E$ and attaches the isogeny
$\pi\from \tilde{E} \to E$ at the marked point $p$. We conclude that $D_\pi$
is a divisor, since 
\begin{align*}
	\dim D_{\pi} = \dim \Cc &= \dim \Hc^0_{N-\deg \pi,g-1} +1 \\ &= 2(g-1)-2+1 = 2g-3 = \dim \Hc^0_{N,g}(E) -1
\end{align*}
 Moreover, by induction, the universal curve $\Cc$ is irreducible, and so is $D_{\pi}$.
\end{proof}

We are on our way to prove 
\begin{reduction}
\label{reduction:D-sing-and-D-pi}
 It is enough to show that every connected component of
$\widetilde{\Hc}_{N,g}^0(E)$ contains a point in $D_{\text{sing}}$ or in $D_{\pi}$ for some isogeny $\pi$.
\end{reduction}

Indeed, this follows will follow from the  \Cref{reduction:D-sing}
 and the following proposition.
\begin{proposition}
\label{proposition:connected-boundary}
The union
\[
	D_{\text{sing}} \cup \bigcup_{\pi} D_{\pi} \subset \widetilde{\Hc}_{N,g}^0(E)
\]
is connected, where $\pi$ ranges through all isogenies $\pi:\tilde{E} \to E$ of degree at most $N-1$. 
\end{proposition}

In view of irreducibility of $D_{\pi}$ and $D_{\text{sing}}$, it is enough 
to prove that these divisors meet for every $\pi$. That is, we 
only have to exhibit a curve that is in both.

Here is a candidate. Let $d$ be the degree of the isogeny
$\pi\from \tilde{E} \to E$. Take a cover $\widetilde{X}_0 \to E$ in
$\Hc_{N-d,g-2}(E)$. Now pick two points $p_1,p_2 \in \widetilde{X}_0$
in a fiber of the map to $E$, and glue them to get a nodal curve $X_0 \to E$.
Then, pick a generic third point $q \in X_0$, and attach a copy
of $\tilde{E}$ at $q$. Let this be $X \to E$, as in \cref{figure:cover in both Dpi and Dsing - fig04}.

\begin{figure}
\centering
\includegraphics{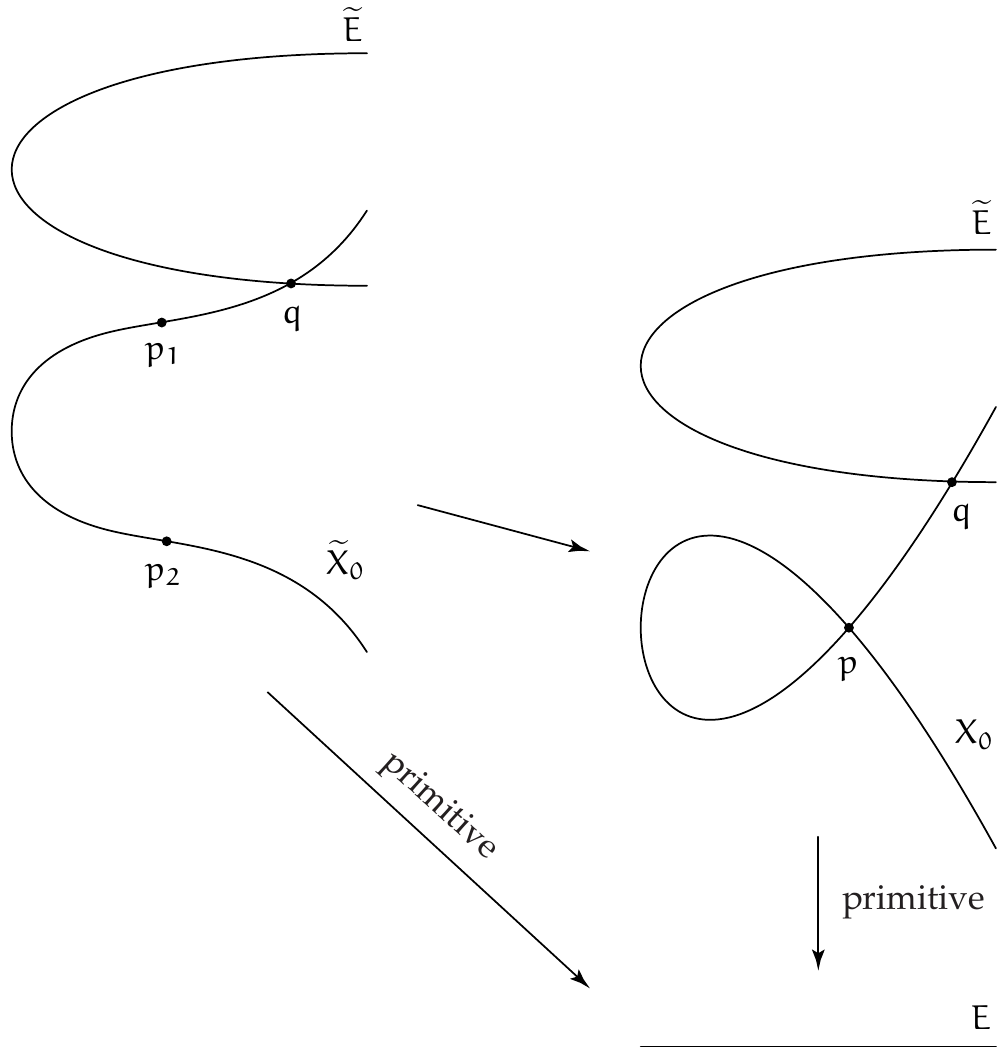}
\caption{A cover in both $D_{\pi}$ and $D_{\text{sing}}$}
\label{figure:cover in both Dpi and Dsing - fig04}
\end{figure}

From the definitions of $D_{\pi}$ and $D_{\text{sing}}$, we make the following
observation.

\begin{observation}
\label{lemma:intersection}
If $\widetilde{X}_0 \cup \tilde{E} \to E$ and $X\to E$ are both primitive covers,
 then $X\to E$ is a cover in the intersection $ D_{\text{sing}} \cap D_{\pi}$.
\end{observation}

\begin{corollary}
If $g\geq 4$, then $D_{\text{sing}}$ meets $D_{\pi}$ for every isogeny $\pi$.
\end{corollary}
\begin{proof}
When $g-2 \geq 2$, we may choose a primitive cover $\widetilde{X}_0 \to E$.
This guarantees that both conditions in \Cref{lemma:intersection} are satisfied.
\end{proof}

This settles  \Cref{proposition:connected-boundary} for $g\geq 4$,
and in consequence the \Cref{reduction:D-sing-and-D-pi}.

We can apply the same construction for $g=3$, but then $\widetilde{X}_0 \to E$
would be an isogeny. We have to choose it carefully so that the 
hypotheses of \Cref{lemma:intersection} are satisfied.
It turns out that we cannot choose such a $\widetilde{X}_0 \to E$ for every
isogeny $\pi$, but we can for a large subclass: the reduced ones.

\begin{definition}
An isogeny $\pi\from  \tilde{E} \to E$ is \emph{reduced}
 if it does not factor through
the multiplication by $n$ map $M_n\from E \to E$ for any integer $n>1$.
\end{definition}

\begin{lemma}
\label{lemma:technical-isogeny}
For any reduced isogeny $\pi\from \tilde{E} \to E$ and any $D \geq 2$, 
there is a reduced degree $D$ isogeny $\widehat{\pi}\from \widehat{E} \to E$ 
and points $p_1,p_2 \in \widehat{E}$ such that
\begin{enumerate}
	\item There is no common non-trivial subcover to $\pi$ and
	$\widehat{\pi}$.
	\item For any subcover $\pi'\from \widehat{E} \to E'$, the
	points $\pi'(p_1),\pi'(p_2) \in E'$ are the same if
	and only if $\pi'=\widehat{\pi}$ (that is, the subcover is trivial).
\end{enumerate}
\end{lemma}

We will prove \Cref{lemma:technical-isogeny} in
 \Cref{ssub:proof_of_lemma_technical_isogeny}.

\begin{corollary}
For $g=3$ and a reduced isogeny $\pi$, the divisor $D_{\pi}$ meets
$D_{\text{sing}}$.
\end{corollary}
\begin{proof}
Take $\widetilde{X}_0$ as the $\widehat{E}$ produced by   
\Cref{lemma:technical-isogeny}. This satisfies the conditions of
 \Cref{lemma:intersection}.
\end{proof}

To finish the proof of  \Cref{proposition:connected-boundary},
for each isogeny $\pi$, we need to find a reduced isogeny $\pi'$ such
that  $D_{\pi}$ meets $D_{\pi'}$.

\begin{lemma}
For $g=3$, the divisors $D_{\pi_1}$ and $D_{\pi_2}$ meet if $d_1+d_2=N-1$,
where $d_i$ is the degree of the isogeny $\pi_i\from \tilde{E}_i \to E$.
\end{lemma}
\begin{proof}
Form the nodal curve $X$ with components $\tilde{E}_1, \tilde{E}_2$ and $E$,
by gluing $\tilde{E}_1$ with $E$ at a point $p$, and $\tilde{E}_2$ with $E$
at a point $q \neq p$, as in \cref{figure:intersection of Dpi1 and Dpi2 - fig05}.
The cover $X \to E$ lies in both $D_{\pi_1}$ and $D_{\pi_2}$. 
\begin{figure}
\centering
\includegraphics{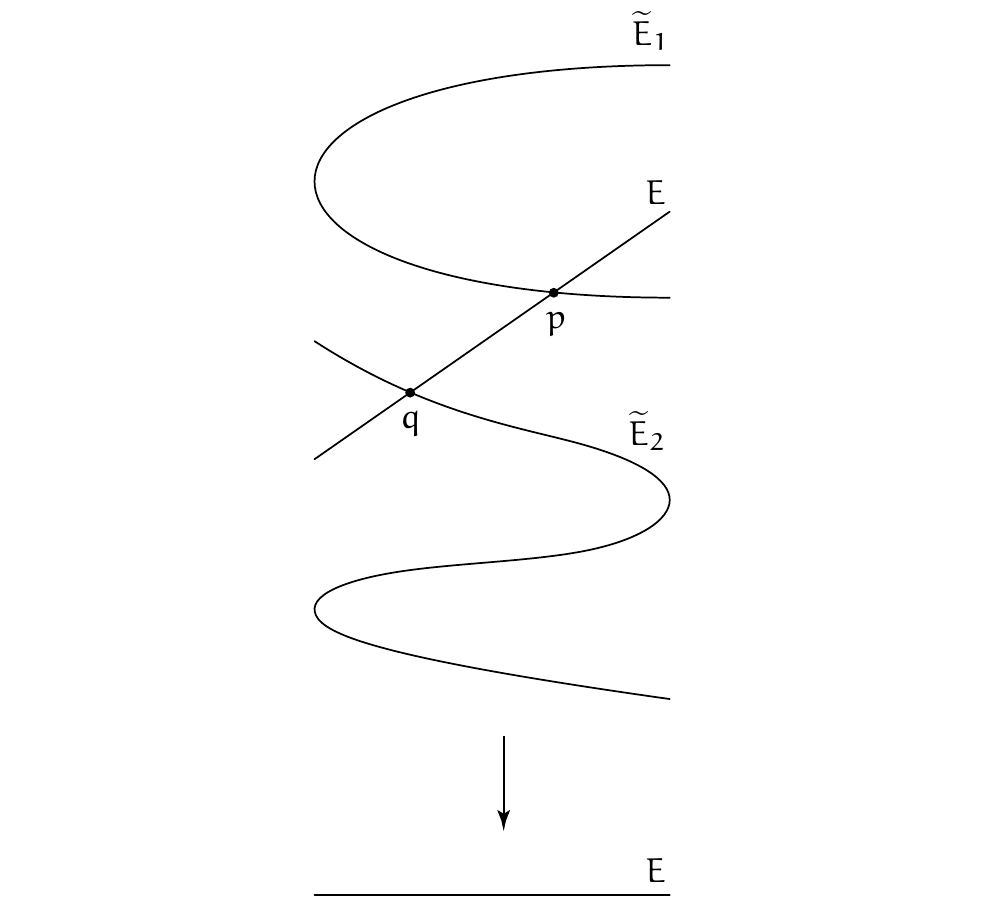}
\caption{A cover in both $D_{\pi_1}$ and $D_{\pi_2}$.}
\label{figure:intersection of Dpi1 and Dpi2 - fig05}
\end{figure}
\end{proof}

This completes the proof of \Cref{proposition:connected-boundary},
since there are reduced isogenies of any degree. This is easy to see directly, 
but is already a consequence of \Cref{lemma:technical-isogeny},
so we are not going to restate it as a separate claim.

We have so far settled  \Cref{reduction:D-sing-and-D-pi}. This gives a 
good mileage towards our goal, the \Cref{reduction:final}.
 To help organize what is left to do, let us introduce a piece of language.

\begin{definition}
Let $C \to E$ be a cover in $\widetilde{\Hc}^0_{N,g}(E)$, and  $\eta\in C$ 
a node in the source. Let $\widetilde{C}$ be the partial normalization
of $C$ at $\eta$. We call $\eta$
\begin{itemize}
	\item a \emph{disconnecting node} if $\widetilde{C}$ is disconnected.
	\item a \emph{$E$-superfluous node} if $\widetilde{C}$ is connected and
	$\widetilde{C} \to E$ is primitive.
	\item a \emph{$E$-relevant node} if $\widetilde{C} \to E$ is not primitive,
	i.e., it factors through
	a non-trivial isogeny.
\end{itemize}
\end{definition}  

In this language, a cover is in $D_{\text{sing}}$ if and only if it has an 
$E$-superfluous node. We can reexpress  \Cref{reduction:D-sing-and-D-pi} as saying that it is enough to find, 
in each connected component, a cover
with an $E$-superfluous node, or a disconnecting node
 such that if the two components of the partial normalization are $C_1 \sqcup C_2$, then
\begin{enumerate}
 	\item \label{item:condition1A} $C_1 \to E$ is an isogeny, and
 	\item $C_2 \to E$ is primitive.
 \end{enumerate} 

We can relax this and drop condition (\ref{item:condition1A}).

\begin{reduction}
\label{reduction:super-quasi-disconnecting}
It is enough to find, in each connected component of Hurwitz space $\Hc^0_{N,g}(E)$,
a cover $C \to E$ with an $E$-superfluous node or a disconnecting node such that the restriction 
of $C_1 \sqcup C_2 \to C \to E$ to $C_2 \to E$ is primitive.
\end{reduction}
\begin{proof}
Indeed, let us say we have a cover $C\to E$ and a disconnecting node 
$\eta \in C$ such that such that
$C_2 \to E$ is primitive. If 
$C_1 \to E$ is an isogeny, this is a point in $D_{\pi}$, and we are done. Assume otherwise
that $C_1 \to E$ is not an isogeny. We want to find another point in the same connected
component which is in one of the $D_\pi$ divisors.

Let $\pi\from  \tilde{E} \to E$ the maximal isogeny for which $C_1 \to E$ factors through, 
and let $\eta_1 \in C_1$ be the point over the node $\eta$.
Then $(C_1,\eta_1) \to \tilde{E}$ is an element of 
$\Cc \to \widetilde{\Hc}^0_{N_1,g_1}(\tilde{E})$, the universal curve
over the component of Hurwitz space $C_1 \to E$ lies on. 
By the \Cref{hypothesis:inductionHurwitz}, $\Cc$
is irreducible. Hence, we can connect $(C_1,\eta_1)$ to any
other element of $\Cc$. For example, we can take
$(C' \cup_p \tilde{E}, \eta') \to \tilde{E}$, where $\eta' \in C'$, for
some cover $C' \to \tilde{E}$.

Attaching $(C_2,\eta)\to E$ back to 
$(C' \cup \tilde{E} ,\eta') \to \tilde{E} \to E$,
 we get
\[
	X=\tilde{E} \cup_p C' \cup_{\eta} C_2 \to E
\]
as in \cref{figure:reduction:super-quasi-disconnecting - fig06}.
As $C_2 \to E$ is primitive, so is 
$C' \cup_\eta C_2 \to E$. Hence, 
the cover $X \to E$ is an element of $D_{\pi}$.
\begin{figure}
\centering
\includegraphics{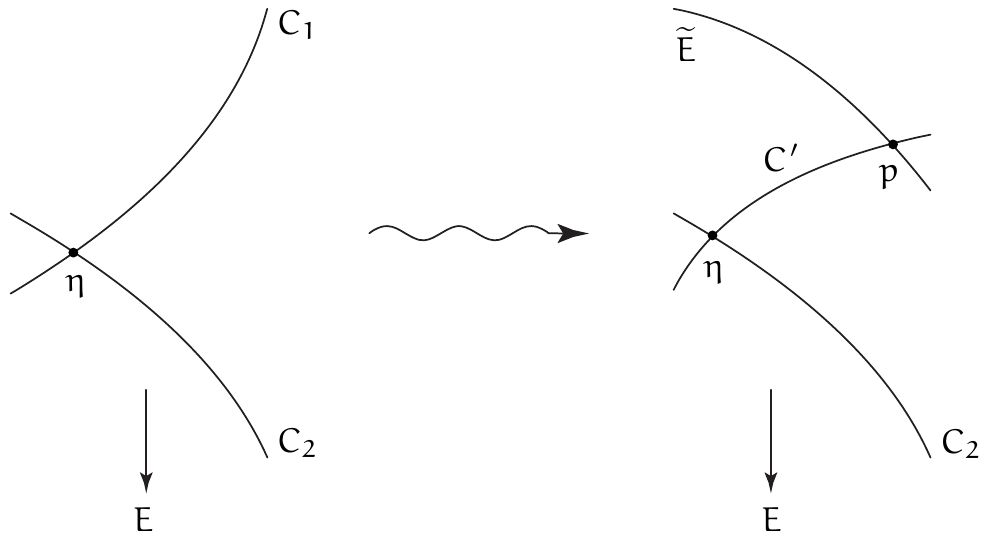}
\caption{From a disconnecting node with a
primitive component to a cover in $D_{\pi}$, as in \Cref{reduction:super-quasi-disconnecting}.}
\label{figure:reduction:super-quasi-disconnecting - fig06}
\end{figure}
\end{proof}

\begin{remark}
Note that for this argument it was
 important that our induction hypothesis accommodates for
any genus one target, as long as the genus of the source is smaller than $g$. 
\end{remark}

We can do even better.
\begin{reduction}
\label{reduction:super-disconnecting}
It is enough to find, in each connected component of Hurwitz space, a cover
with an $E$-superfluous node or a disconnecting node.
\end{reduction}

\begin{proof}
We use the same trick. Let $C_1 \sqcup C_2 \to C \to E$ be the partial
normalization of a curve with a disconnecting node $\eta$. Let $\eta_i \in C_i$
be the point over $\eta$. 
As $g\geq 3$,
at least one of the $C_i$ has genus at least $2$. Say that is $C_1$. Let
$\tilde{E} \to E$ be the maximal isogeny $C_1 \to E$ factors through.
Degenerate $(C_1,\eta_1) \to E$ to $(C' \cup \tilde{E}, \tilde{\eta}) \to E$
with $\tilde{\eta} \in \tilde{E}$. 
Attaching $(C_2,\eta_2)$ back, we get
$C' \cup_p \tilde{E} \cup_{\eta} C_2 \to E$,
as in \cref{figure:reduction:super-disconnecting fig07}.

\begin{figure}
\centering
\includegraphics{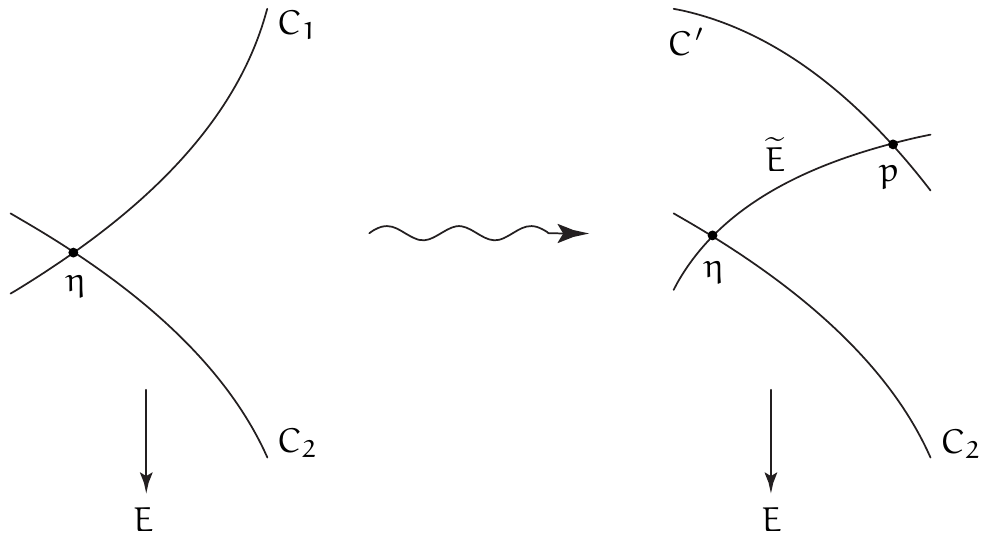}
\caption{From a cover with a disconnecting node
to a cover with a disconnecting node and a primitive component, as in \Cref{reduction:super-disconnecting}.}
\label{figure:reduction:super-disconnecting fig07}
\end{figure}

We claim that cover $\tilde{E} \cup_{\eta} C_2 \to E$ is primitive. Suppose not. Then, by definition, the map
${\tilde{E} \cup_{\eta} C_2 \to E}$ factors through a non-trivial isogeny, and
so would the original cover ${C_1\cup_{\eta}C_2 \to E}$
(because $C_1 \to E$ itself factors through $\tilde{E} \to E$). But the original cover is primitive
by assumption, which is a contradiction.

Hence, the node $p$ joining $C'$ and $\tilde{E}$ is a disconnecting node 
for which one of the components is primitive.
 So we are done by \Cref{reduction:super-quasi-disconnecting}.
\end{proof}

Let us finally prove the \Cref{reduction:final}.
\begin{proof}
Let $C \to E$ be a cover in $\Delta \subset \widetilde{\Hc}^0_{N,g}(E)$. Hence,
$C$ has a node, call it $\eta$. 
By  \Cref{reduction:super-disconnecting}, it is enough to show that in the connected component of 
$C \to E$ there is a cover with an $E$-superfluous or a disconnecting node.

If $\eta$ itself is $E$-superfluous
or disconnecting, we are done. We may assume that
$\eta$ is $E$-relevant, then. This means that the partial normalization
$\tilde{C} \to E$ is connected, and that it factors through a non-trivial
isogeny $\tilde{E} \to E$. Let $\eta_1,\eta_2 \in \tilde{C}$ be the preimages of the node $\eta$, and $N'= \deg(\tilde{C} \to \tilde{E})$, as in \cref{figure:E relevant node - fig08}.

\begin{figure}
\centering
\includegraphics{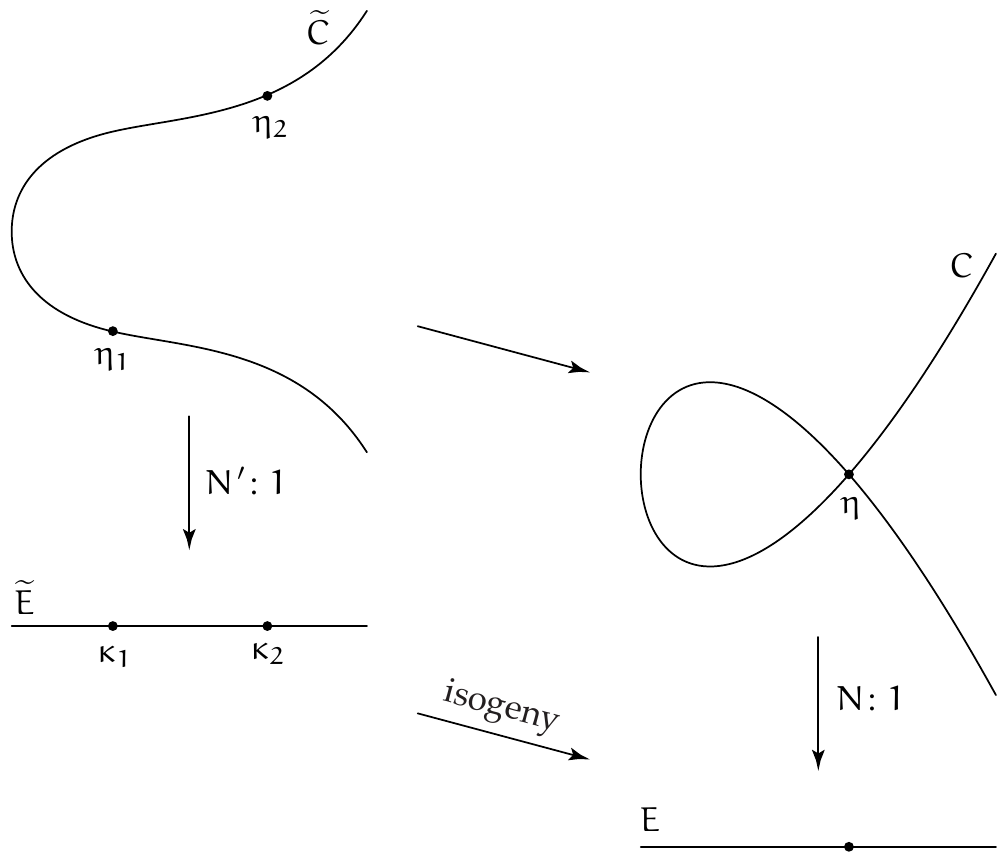}
\caption{An $E$-relevant node $\eta$.}
\label{figure:E relevant node - fig08}
\end{figure}

The data of $C \to E$ is encoded in the cover $\tilde{C} \to \tilde{E}$, with the two
marked points $\eta_1,\eta_2$. Let us construct a parameter space $\Pc$ for the triples
$(\tilde{C} \to \tilde{E}, \eta_1, \eta_2)$.

Let $\Cc \to \widetilde{\Hc}^0_{N',g-1}(\tilde{E})$
be the universal curve over the component of Hurwitz parametrizing 
primitive covers. The individual covers $C_b \to \tilde{E}$ assemble together in a map $\Cc \to \tilde{E}$. 
Then the triple $(\tilde{C} \to \tilde{E}, \eta_1, \eta_2)$ is a point in 
\[
	\Pc=\Cc \times_{E \times \widetilde{\Hc}} \Cc - \Delta
\]
where the map from $\Cc$ to $E$ is the composition of $\Cc \to \tilde{E}$ with
the isogeny $\tilde{E} \to E$.

Conversely, given any point 
$(\tilde{C}' \to \tilde{E}, \tilde{\eta}_1,\tilde{\eta}_2) \in \Pc$, we may identify
the points $\tilde{\eta}_1$ and  $\tilde{\eta}_2$ to get a cover $C' \to E$ in $\widetilde{\Hc}^0_{N,g}(E)$.

This prompt us to characterize the connected component of $\Pc$
containing the point ${(\tilde{C} \to \tilde{E},\eta_1,\eta_2)}$. If this connected component 
contains a point which gives rise to a cover of $E$ with
a disconnecting or $E$-superfluous node, we will be done.

Let $\kappa_i$ be the image of $\eta_i$ in $\tilde{E}$.
If $\kappa_1=\kappa_2$, then $C \to E$ would also factor
 through $\tilde{E} \to E$, which is contrary to our assumption that 
 $C\to E$ is primitive. Hence ${\kappa_1 \neq \kappa_2}$. 
We can even say more: after identifying $\kappa_1$ with $\kappa_2$, the
map from the nodal curve $\tilde{E}/(\kappa_1 \sim \kappa_2)$ 
to $E$ is a primitive cover.

Note that the evaluation map $\Cc \to \tilde{E}$ produces a map
\[
	f\from \Pc=\Cc \times_{E \times \widetilde{\Hc}} \Cc - \Delta \to \tilde{E} \times_E \tilde{E}
\]
Our point $(\tilde{C} \to \tilde{E}, \eta_1,\eta_2)$ in $\Pc$ maps to $(\kappa_1,\kappa_2)$.

We describe the components of $\Pc$ in terms of the map $f$.
\begin{proposition}
\label{proposition:connected-fibers}
 The map 
$f\from  \Pc \to \tilde{E} \times_E \tilde{E}$
 induces a bijection in connected components.
 \end{proposition} 

Before we prove this proposition, let us see how to establish the \Cref{reduction:final}
assuming it.

If $g\geq 4$, we will use the construction in \Cref{figure:teh case g geq 4 fig21}.
Take a primitive cover $C_0 \to \tilde{E}$ in $\Hc^0_{N',g-2}(\tilde{E})$. 
Pick a pair of distinct points $p_1,p_2 \in C_0$ in a generic fiber of the map to
$\tilde{E}$. 
Identify $p_1,p_2$ to get a cover $C_1 \to \tilde{E}$ with a node $p$.
Note that the node $p$ is $\tilde{E}$-superfluous, by construction.

Pick any preimages $\eta'_i \in C_1$ of $\kappa_i \in \tilde{E}$. 
By \Cref{proposition:connected-fibers},
 ${(C_1 \to \tilde{E}, \eta_1',\eta_2')}$ is in the same component
of $\Pc$ as the original $(\tilde{C} \to E, \eta_1,\eta_2)$. Hence, if we
identify $\eta'_1$ with $\eta'_2$, we get a primitive cover $C_2 \to E$ in the
same component as the original cover $C \to E$. However, $C_2 \to E$  has two nodes: $\eta'$ and $p$. But
as $p$ was $\tilde{E}$-superfluous, it is $E$-superfluous as well. By
\Cref{reduction:super-disconnecting}, we are done.

\begin{figure}
\centering
\includegraphics{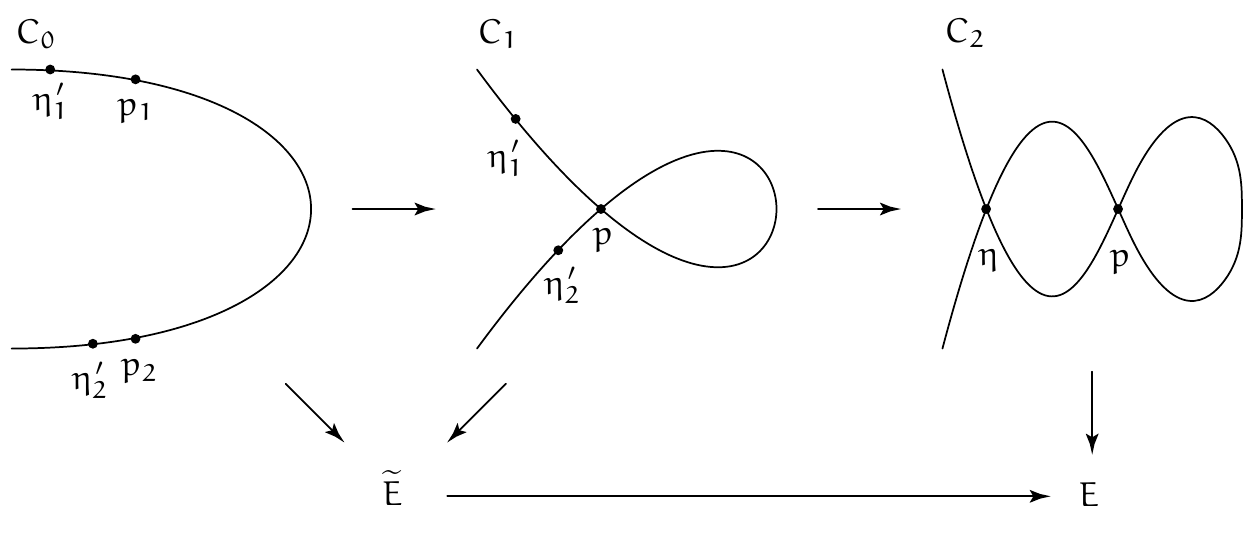}
\caption{The construction for the case $g\geq 4$.}
\label{figure:teh case g geq 4 fig21}
\end{figure}

If $g=3$, we have to use a different argument, because there are no primitive genus 2 covers of $\tilde{E}$
with an $\tilde{E}$-superfluous node. Nevertheless, we can take  
 an isogeny $\tilde{E}_0 \to \tilde{E}$ of degree $N'-1$, and attach
 an $\tilde{E}$ tail. Let the resulting cover be
 $C_0 = \tilde{E}_0 \cup_p \tilde{E} \to \tilde{E}$,
 as in \Cref{figure:the-case-g-3-fig22}.

  Now take $\kappa_1,\kappa_2 \in \tilde{E} \subset C_0$. By  
  \Cref{proposition:connected-fibers}, $(C_0, \kappa_1,\kappa_2)$
  is in the same connected component as $(\tilde{C},\eta_1,\eta_2)$. Let
  $C_1$ be the nodal curve obtained by identifying $\kappa_1$ and $\kappa_2$
  in $C_0$. We can degenerate our original cover $C \to E$ into $C_1 \to E$.
  But $p \in C_1$ is a disconnecting node! Hence, by \Cref{reduction:super-disconnecting}, we are done.
\begin{figure}
\centering
\includegraphics{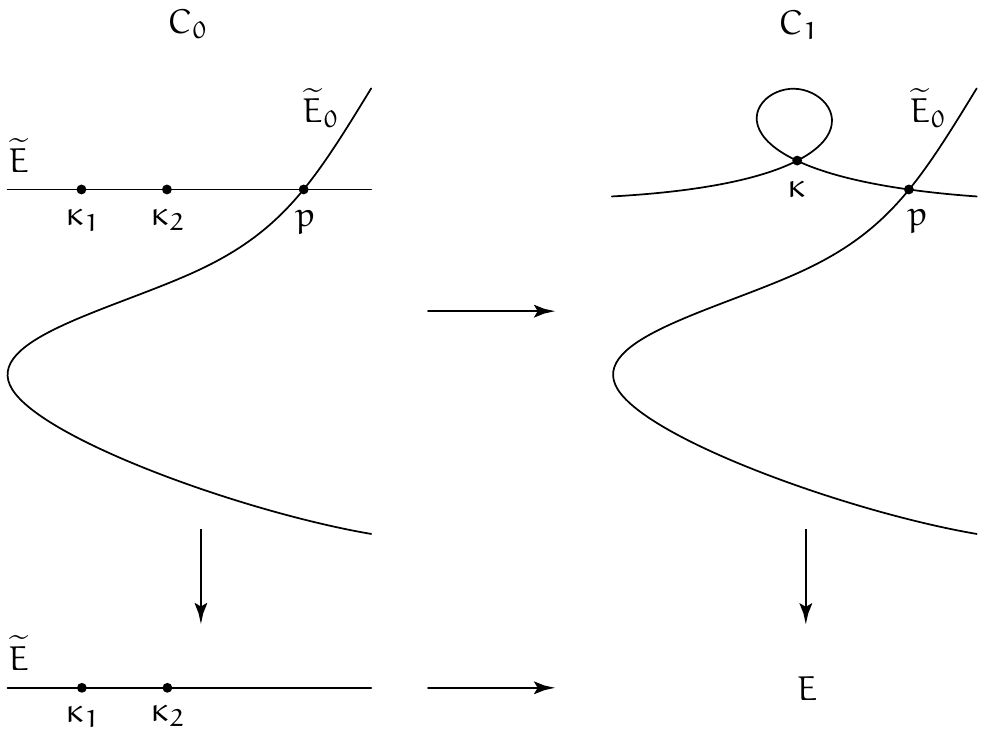}
\caption{The construction for the case $g=3$.}
\label{figure:the-case-g-3-fig22}
\end{figure}
\end{proof}

\begin{proof}[Proof of Proposition \ref{proposition:connected-fibers}]
Recall that $\tilde{E} \to E$ is an isogeny, and
$\Cc \to \widetilde{\Hc}^0_{N',g-1}(\tilde{E})$ is the universal
curve. The space $\Pc = \Cc \times_E \Cc - \Delta$ parametrizes covers 
$C \to \tilde{E}$ and pairs of distinct points $p_1,p_2 \in C$ that map to the
same point $q \in E$ under the composite map $C \to \tilde{E} \to E$.
Hence, there is a map 
	$\Pc \to E \times \widetilde{\Hc}^0_{N',g-1}(\tilde{E})$
which sends $(C \to \tilde{E}, p_1,p_2)$ to $(C \to \tilde{E}, q)$. This
turns $\Pc$ into a finite cover of an irreducible target. Hence $\Pc$
is equidimensional, i.e., all
 components of $\Pc$ have the same dimension.

Now fix a component $\Gamma \subset \tilde{E} \times_E \tilde{E}$, and
let 
\[
	\Pc_\Gamma =f^{-1}(\Gamma) \subset \Pc 
\]
where $f$ is the map
$f\from  \Pc  \to \tilde{E} \times_E \tilde{E}$.
We want to show that $\Pc_{\Gamma}$ is irreducible. There is a forgetful
map
\[
 \pi\from  \Pc_{\Gamma} \to \widetilde{\Hc}^0_{N',g-1}(\tilde{E})	
\]
sending $(C \to \tilde{E},p_1,p_2)$ to $(C \to \tilde{E})$.
The fiber of $\pi$ over $C \to \tilde{E}$ is
\[
	\pi^{-1}((\phi\from C \to \tilde{E})) = \set{p_1,p_2 \in C, (\phi(p_1),\phi(p_2))
	 \in \Gamma \text{ and } p_1\neq p_2}
\]
which is the same as saying that the fiber of $\pi$ is the preimage of $\Gamma$ in the map
${C \times_E C - \Delta \to \tilde{E} \times \tilde{E}}$. Hence, the following
diagram is Cartesian.
\[
	\xymatrix{
		\pi^{-1}((\phi\from C \to \tilde{E})) \ar[r] \ar[d] & \Gamma \ar[d] \\
		{\displaystyle C \times_E C}- \Delta \ar[r] & \tilde{E}\times_E \tilde{E}
	}
\]

Now observe that 
\begin{itemize}
	\item  the target of $\pi$ is irreducible, by \Cref{hypothesis:inductionHurwitz},
	\item all fibers of $\pi$ have dimension one,
	\item the general fiber of $\pi$ is irreducible: more precisely, if $C \to \tilde{E}$ is simply branched, then 
	\Cref{lemma:fiber product} says that the fiber of $\pi$ is irreducible.
\end{itemize}

Moreover, we already know that $\Pc_{\Gamma}$ is equidimensional.
Hence, it follows that $\Pc_{\Gamma}$ is irreducible.
\end{proof}

\subsection{The degeneration argument} 
\label{sub:the_degeneration_argument}
Given \Cref{reduction:final}, to complete the proof of \Cref{theorem:hurwitz_irreducible}
it is enough to prove the following.
\begin{proposition}
\label{propositon:meet-delta}
Every component of $\widetilde{\Hc}_{N,g}^0(E)$ meets the boundary $\Delta$.
\end{proposition}
We will use our study of hyperplane sections of $\widetilde{V}_{df+Ne,g}(E)$ to produce
such a cover in $\Delta$.
\begin{proof}
Choose $d$ large enough so that the Severi variety $\widetilde{V}_{df+Ne,g}(E)$ dominates
the Hurwitz space $\widetilde{\Hc}_{N,g}(E)$. 
For example, by Riemann--Roch $d=g+3$ is enough. For any component of $\widetilde{\Hc}_{N,g}^0(E)$, 
there is a component of $\widetilde{V}_{df+Ne,g}(E)$, say $V_0$,  dominating it. Now intersect $V_0$
with multiple hyperplanes of the form $H_p$, as in \Cref{sec:hyperplane_sections}. Eventually the
curve $E_0 \subset E \times \Pbb^1$ is split off generically. We want to see what the image of this generic point 
is in the Hurwitz space $\widetilde{\Hc}_{N,g}^0(E) \subset \overline{\Mc}_g(E,N)$. The map is initially not defined there, but
\Cref{theorem:simple_hyperplane_to_hurwitz} tells us how to extend it. The resulting cover
has singular source and there are no contracted components. Hence it is a point of 
$\Delta \subset \widetilde{\Hc}_{N,g}^0(E)$.
\end{proof}
 This finishes the
proof of \Cref{theorem:hurwitz_irreducible}.

Let us mention two alternatives ways to establish \Cref{propositon:meet-delta} in a more topological vein.
To produce a node in the source, we have to bring a pair branch points together, in a way that, just before they meet, their
monodromy transpositions exchange the same pair of sheets. To ensure this, we have to choose the path of approach carefully.

If one keeps track of how the permutation data changes as the branch points move around, 
one can re-encode the question in a combinatorial setting---we want to show that from any initial data, there
is a sequence of moves that will result in a pair of adjacent branch points with the same monodromy transposition.
However, it is still a tough problem.
For example, in
\cite{graber_note_2002,kanev_irreducibility_2005}, they can exhibit such a sequence of moves only when there are at least $2d-2$ branch points.

In \cite{berstein_classification_1984}, Berstein and Edmonds ask the same combinatorial problem, 
but allowing not only the moves on the branch points, but also Dehn twists on the target.
In Proposition~5.2 of \cite{berstein_classification_1984}, they show that under this setup, whenever the target has genus one, we can apply a 
sequence of moves to produce a pair of adjacent branch points with the same monodromy. 
 Tracing back through their argument, we see that no Dehn twists 
are used, so their result holds in our setup as well, effectively proving \Cref{propositon:meet-delta}.


As the genus of the target increases, the combinatorial problem gets even harder. Gabai and Kazez get around this challenge with the following beautiful construction.
\begin{lemma}[Gabai--Kazez]
\label{lemma:gabai-kazez-style}
For an arbitrary smooth curve $D$, every component of $\widetilde{\Hc}_{d,g}(D)$ meets the boundary $\Delta$.
\end{lemma}
This is similar to Lemma~3.5 in \cite{gabai_classification_1987}.
\begin{proof}
Given a cover $C \to D$, we want to choose a path $\gamma: [0,1] \to D$ connecting a pair of branch points, in a way
that the pair of sheets exchanged by the ramification point over $\gamma(0)$ get transported to the pair of sheets exchanged 
over $\gamma(1)$.

Choose a generic map $h:C \to [0,1]$. If we perturb $h$, then $C \to D \times [0,1]$ is a submersion, and the
image has at most double curves and some triple points. The double curves can be closed curves, or have endpoints 
at the ramification points of $C \to D$. Conversely, at every ramification point there will be exactly one double curve ending there.

Take a double curve connecting a pair of ramification points, and let $\gamma:[0,1] \to D$ be its image in $D$.
The sheets exchanged over the branch points $\gamma(0)$ and $\gamma(1)$  are exactly the two sheets that meet to form the double curve. So the transpositions agree, and we are done!
\end{proof}

\subsection{Proof of Lemma \ref{lemma:technical-isogeny}} 
\label{ssub:proof_of_lemma_technical_isogeny}
Let us recall the statement we want to prove. 
We are given a reduced isogeny $\pi\from \tilde{E} \to E$ and an integer $D \geq 2$.
We want to construct a reduced degree $D$ isogeny $\widehat{\pi}\from \widehat{E} \to E$ 
and a pair of points $p_1,p_2 \in \widehat{E}$ such that
\begin{enumerate}
	\item There is no common non-trivial subcover to $\pi$ and
	$\widehat{\pi}$.
	\item \label{item:points} For any subcover $\pi'\from \widehat{E} \to E'$, the
	points $\pi'(p_1),\pi'(p_2) \in E'$ are the same if
	and only if $\pi'=\widehat{\pi}$ (that is, the subcover is trivial).
\end{enumerate}

We are going to translate this into the language of lattices. Let 
$\Lambda \subset \Cbb$ be the rank 2 lattice such that
$E = \Cbb/\Lambda$.
Our isogeny $\tilde{E} \to E$ corresponds to a sublattice
$\tilde{\Lambda} \subset \Lambda$. To construct the second 
cover $\widehat{E} \to E$, we need to find another sublattice
 $\widehat{\Lambda} \subset \Lambda$ such that
\begin{enumerate}
	\item[(0')] \label{item:condition2} $\# \left(\Lambda / \widehat{\Lambda} \right) =D$ (this
	corresponds to the degree of $\widehat{E} \to E$ being $D$).
	\item[(1')] \label{item:condition1}
	 $\tilde{\Lambda}+\widehat{\Lambda} = \Lambda$ (and this to having no non-trivial common subcover)
\end{enumerate}

Given two points $p_1,p_2$ in a fiber of $\widehat{E} \to E$, we can pick a path between them, which produces an element $v \in \Lambda$. Conversely, given
$v \in \Lambda$ and any point $p_1\in \widehat{E}$, we can lift $v$ 
and get a second point $p_2 \in \widehat{E}$.  Condition (2) above translates to 
\begin{enumerate}
\item[(2')]
\label{equation:condition3}
	$\widehat{\Lambda}+\Zbb v = \Lambda$
\end{enumerate}

Let $m$ be the largest integer such that $\tilde{\Lambda} \subset m \Lambda$.
Note that saying that $\tilde{E} \to E$ is reduced means that $m=1$.
\begin{proposition}
\label{proposition:technical-lattice}
There exist a choice of $\widehat{\Lambda}$  and $v$ satisfying conditions \textnormal{(0')}, \textnormal{(1')} 
and \textnormal{(2')}
 if and only if $\gcd(D,m)=1$.
\end{proposition}

The case $m=1$ implies \Cref{lemma:technical-isogeny}.
\begin{proof}
We identify $\Lambda$ with $\Zbb\oplus \Zbb$ by choosing a basis.

\begin{lemma}
There is $(am,bm) \in \tilde{\Lambda}$ such that $\gcd(a,b)=1$.
\end{lemma}
\begin{proof}
Applying the lemma to $\tilde{\Lambda}/m$, we reduce to the case $m=1$. 
Pick the smallest positive $d$ such that $(0,d) \in \tilde{\Lambda}$. 
Choose another vector $(a,b) \in \Lambda$ such that $(a,b)$ and $(0,d)$ 
generate $\Lambda$. In particular $\gcd(a,b,d)=1$.

It may be that $\gcd(a,b) \neq 1$. But we may temper with it by adding arbitrary multiples
of $(0,d)$. So it is enough to find an $n$ for which ${\gcd(a,b+nd)=1}$.
For each prime $p$ dividing $a$, we have that $p$ does not divide $d$ or $b$, and therefore 
there is a residue class $n_p$ modulo $p$ such that
\[
	n \equiv n_p \mod p \implies p \notdivides b+nd
\]
Using the Chinese remainder theorem, we can find an integer $n$ such that no prime 
diving $a$ divides $b+nd$. Hence, $\gcd(a,b+nd)=1$, as we wanted to show. 
\end{proof}

\begin{lemma}
\label{lemma:construction-isogeny}
Take $(a',b')$ such that $ab'-a'b=1$. Then the lattice
\[
	\widehat{\Lambda}  = \Zbb (a',b') \oplus \Zbb (Da,Db)
\]
satisfies $\# \left(\Zbb \oplus \Zbb / \Lambda' \right) =D$,
and
\[
	\widehat{\Lambda}+\tilde{\Lambda} =\Zbb \oplus \Zbb (=\Lambda)
\]
as long as $\gcd(m,D)=1$. 

Moreover, if we take $v=(a,b)$, then 
\[
	\widehat{\Lambda} + \Zbb v = \Lambda
\]
\end{lemma}
\begin{proof}
That $\# \left(\Zbb \oplus \Zbb / \widehat{\Lambda} \right) =D$ follows
directly from $ab'-a'b=1$.

To compute the index of $\Lambda+\Lambda' \subset \Zbb \oplus \Zbb$,
first notice that
$(a,b) \in \Lambda+\Lambda'$, since both
$(Da,Db)$ and $(ma,mb)$ are in the lattice. But $(a',b')$ is also
in $\Lambda+\Lambda'$, and
$(a,b)$ and $(a',b')$ span $\Zbb \oplus \Zbb$.
\end{proof}

\Cref{lemma:construction-isogeny} shows that if $\gcd(m,D)=1$, then we can choose
$\widehat{\Lambda}, v$ satisfying the conditions of \Cref{proposition:technical-lattice}.

Conversely, suppose that $\gcd(D,m) \neq 1$. Let $p$ be a prime
dividing $\gcd(D,m)$. Look at the reduction of 
$\left(\widehat{\Lambda} + \tilde{\Lambda} \right)$
modulo $p$. As $\tilde{\Lambda} \subset m(\Zbb \oplus \Zbb)$, we have 
$\tilde{\Lambda} \equiv 0 \mod p$. Moreover, if
$(a,b)$ and $(a',b')$ generate $\widehat{\Lambda}$, then 
\[
	\det \left( \begin{array}{cc}
		a & b \\
		a'& b'
		\end{array} \right) = ab'-ba' = D\equiv 0 \mod{p}
\]
Hence $(a,b)$ and $(a',b')$ are dependent modulo $p$. Therefore,
$\widehat{\Lambda}+\tilde{\Lambda}$ has rank at most one in the reduction modulo
$p$. Therefore, it can't span the full lattice! This finishes the proof of
\Cref{proposition:technical-lattice}, which establishes \Cref{lemma:technical-isogeny}.
\end{proof}

\subsection{Varying the moduli of $E$} 
\label{sub:varying_the_moduli_of_E}

We will now study the version of Hurwitz space where the target curve is allowed to vary in moduli as well
(see \Cref{definition:global-hurwitz}).
Recall that $\Hc_{d,g,1}$ is fibered over $\Mc_{1,1}$, and the fiber over $(E,p)$ is exactly $\Hc_{d,g}(E)$.
We will use our knowledge of the components of the fiber (\Cref{corollary:components-of-HcNgE}) to describe the
components of the total space $\Hc_{d,g,1}$. More precisely, the construction in \Cref{corollary:factorization} works in pointed families, and
we get a map
\[
	\Hc_{d,g,1} \to \bigsqcup_{d \neq \tilde{d} \divides d} \Hc_{\tilde{d},1,1}
\]
sending a cover $(C,p) \to (E,q)$ to the maximal isogeny $(\tilde{E},\tilde{q}) \to (E,q)$ it factors through.
\Cref{corollary:components-of-HcNgE} says that the fibers of this map are irreducible. Hence it is enough
to describe the components on $\Hc_{\tilde{d},1,1}$. That is, we want to know what are the orbits of isogenies 
under the monodromy action induced by varying the j-invariant of the target.

Most of the isogenies will not be fixed by this monodromy action, but there is a special family which will be: the multiplication by $m$ map, for integers $m$. Hence, besides the degree $\tilde{d}$, there is another invariant: the maximal $m$ such that $\tilde{E} \to E$ factors through the multiplication by $m$ map $\times m \from E \to E$.
 Note that $m^2 \divides \tilde{d} \divides d$.

To prove \Cref{corollary:components-of-Hdg1}, we want to show that these are the only invariants. Here is a concrete way to prove that. The isogeny $\tilde{E} \to E$ corresponds to an inclusion of lattices, which can be represented by
an inclusion
\[
	M : \Zbb \oplus \Zbb \to \Zbb \oplus \Zbb
\]
We are allowed to act by $\GL_2(\Zbb)$ on the source (corresponding to change basis of the lattice for $\tilde{E}$).
We are also allowed to act by $\SL_2(\Zbb)$ on the target---this is the monodromy action. Under these two actions, 
the $2\times 2$ matrix $M$ can be put in a unique standard diagonal form
\[
	M = \left[
	\begin{array}{cc}
	d_1 & 0 \\
	0 & d_2
	\end{array}
	\right]
\]
such that $d_1$ divides $d_2$. Hence $M$ is classified by this pair of integers $(d_1,d_2)$.
Back to our original notation, $m=d_1$ and $\tilde{d}=\det M= d_1  d_2$. Hence these are the only two
discrete invariants, as we wanted to show. This completes the proof of \Cref{corollary:components-of-Hdg1}.

\begin{remark}
The problem of classifying $n \times m$ 
integer matrices modulo actions of $GL(\Zbb)$ on both sides 
is equivalent to the classification of abelian groups. It is, as a matter of fact,
one of the standard ways to derive it. See for example Theorem~6.4 in \cite{artin_algebra_1991}.
\end{remark}

\begin{remark}
We could ask in general what are the components of $\Hc_{d,g,h}$. Using 
\Cref{conjecture:hurwitz-connected}, we can reduce again to the case of unramified covers, that is, $g=d(h-1)+1$.
Unfortunately, for $h \geq 2$, this does not translate anymore into a simple question about
lattices, as above.
\end{remark}
\subsection{Further remarks on the higher genus case} 
\label{sub:further_remarks}

We may ask what this approach can say about \Cref{conjecture:hurwitz-connected}, the
connectedness of the Hurwitz scheme of primitive covers, in
the case the target has genus greater than one. Quite a few challenges arise. To start off, we do not have the induction base case. Kani's theorem 
\cite{kani_hurwitz_2003} assured us that 
the Hurwitz space of primitive genus two covers of elliptic curves is connected, which got our induction kicked off. 
There is no known direct algebraic proof of the corresponding result when the genus of the target is at least two. Aside from this, most of the arguments in this section would go through, but not all---notably, \Cref{lemma:technical-isogeny} really
used the fact that the target had genus one.

But perhaps the biggest challenge is in the analogue of 
\Cref{theorem:simple_hyperplane_to_hurwitz}. We do not know much about the Severi variety of curves in $C \times \Pbb^1$, when $C$ has genus at least two. The biggest obstruction is to get good dimension estimates on the Severi variety, and their generalizations imposing tangency conditions. \Cref{theorem:deftheory} is not enough anymore. In general, the less ample the anticanonical bundle of the surface is, the harder it is to work out the deformation theory of Severi varieties.



\section{Components of the Severi Variety \texorpdfstring{$V_{Ne+df,g}$}{}} 
\label{sec:components_of_the_severi_variety}

Let us come back to the study of Severi varieties on $E \times  \Pbb^1$.
We want to determine what its irreducible components are. A candidate is
the following.
\begin{definition}
Let $V^\text{prim}_{Ne+df,g} \subset |\Lc|$ be the subvariety parametrizing integral curves $C$
of geometric genus $g$, such that the composition of the normalization map
$\tilde{C} \to C$ with the projection map $C \to E$ is a primitive cover (that is, 
$\tilde{C}\to C \to E$ does not factor through non-trivial isogenies).
\end{definition}

We conjecture the following.
\begin{conjecture}
The variety $V^\text{prim}_{Ne+df,g}$ is irreducible, for $d\geq 3$.
\end{conjecture}
\begin{remark}
For $d=2$ it is not clear what to expect. In view of the construction in
\Cref{ssub:the_hyperelliptic_locus}, we see that all the nodes have to lie on four
special fibers (the preimages of the four ramification points of the double cover $E \to \Pbb^1$).
Hence, to each component of $V^\text{prim}_{Ne+2f,g}$, we can associate a partition of $\delta$ in four parts (where
$\delta=2N-1-g$ is the number of nodes). If more than one partition occurs, we would get more than one component.

On the other hand, for $g=2$ only one partition occurs for primitive covers (see \cite{kuhn_curves_1988}). It is not
clear what the behavior will be for $g>2$.
\end{remark}

Using our knowledge of Hurwitz spaces, we can establish irreducibility 
in at least a restricted range.

\begin{theorem}
\label{theorem:severi_irreducible}
When $d\geq 2g-1$, the Severi variety $V^\text{prim}_{Ne+df,g}$ is irreducible.
\end{theorem}
\begin{proof}
Consider the map
$V^\text{prim}_{Ne+df,g} \to \Hc^0_{N,g}(E)$.
The target is irreducible by \Cref{theorem:hurwitz_irreducible}.
To establish the irreducibility of the source, it is enough to show that the map 
is dominant with irreducible fibers of the same dimension.

Let $C \to E$ be a point in $\Hc^0_{N,g}(E)$. The fiber over it corresponds 
to choosing a line bundle $L\in \Pic^d(C)$, and a framing
for the map $C \to \Pbb^1$. The line bundle $L$ is in the non-special range,
 as $d \geq 2g-1$, and hence it has $d-g+1$ global sections. In total, we get
 irreducible fibers of dimension
 \[
 	1+ 2\times (d-g+1) -1
 \]
which is independent of the map $C \to E$.
\end{proof}


\appendix
\section{Deformation Theory}
\label{sec:appendix-def-theory}



The goal of this section is to prove \Cref{theorem:deftheory}. We
will only adapt the arguments already existing in the literature for our setup. 
In specific, we will follow \cite{caporaso_counting_1998} closely. 
Let us remind ourselves of the setup.
We have
\begin{itemize}
 	\item a smooth projective surface $\Sigma$,
 	\item a fixed curve $D \subset \Sigma$,
 	\item a fixed finite set $\Omega \subset D$,
 	\item a fixed homology class $\tau \in H_2(\Sigma, \Cbb)$,
 	\item and integers $g$ and $b$.
\end{itemize}

We are given a substack $W \subset \Mc_g(\Sigma, \tau)$, such that the
general point corresponds to maps
$f\from  C \to \Sigma$ such that
\begin{itemize}
\item the source $C$ is a smooth connected curve of genus $g$,
\item the map $f$ is birational onto its image,
\item the preimage $Q=f^{-1}(D - \Omega)$ is supported on
at most $b$ points.
\end{itemize}

We want to bound the dimension of $W$ in terms of 
\[
	\gamma= -(K_{\Sigma} + D) \cdot \tau +b
\]
More precisely, we want to show that if $\gamma \geq 1$ then
\[
	\dim W \leq g-1 + \gamma
\]
and to control the singularities of $f(C) \subset \Sigma$, according 
to how large $\gamma$ is.

To do so, we need some notation to keep track
of the points in $f^{-1}(D)$. Let $q_1,\ldots, q_{\bar{b}}$ be the support
of $Q = f^{-1}(D -\Omega)$. By our assumption, we have $\bar{b}\leq b$. 
Let $p_1,\ldots, p_a \in C$ be the remaining points
in the support of $f^{-1}(D)$. Note that $f(p_i)$ is in the finite set 
$\Omega$, and hence is
not allowed to vary in $W$. Let $\alpha_i, \beta_j$ be such that
\[
	f^{-1}(D) = \sum \alpha_i p_i + \sum \beta_j q_j 
\]

After a finite base change, we 
get $\widetilde{W} \subset \Mc_{g,a+\bar{b}}(\Sigma, \tau)$ in which the points
$p_i, q_j$ become sections instead of just multisections.

We want to determine the tangent space to $\widetilde{W}$ at $f$. 
We start by determining the tangent space to the 
whole Kontsevich space. Deformation theory tell us that a first order
deformation of the map $f\from C \to \Sigma$,
 without taking into account the tangency conditions, produces
 a section
$\sigma \in H^0(C,\Nc)$
where $\Nc$ is the normal sheaf to the map $f$, defined by the sequence:
\[
0 \to T_{C} \to f^*T_\Sigma \to \Nc \to 0
\]

Although we expect $\Nc$ to be torsion free, we cannot
assume that from the outset. Anyways, it will be useful
to consider the image  $\bar{\sigma} \in H^0(\Nc/\Nc_{\text{tor}})$ of
$\sigma$ under $H^0(\Nc) \to H^0(\Nc/\Nc_{\text{tor}})$.

To take the tangency conditions into account, we invoke
Lemma 2.6 from \cite{caporaso_counting_1998}.

\begin{lemma}[Caporaso--Harris]
\label{lemma:tangency-conditions}
The section $\bar{\sigma}$ vanishes to order $\alpha_i$ at $p_i$, and at least 
$\beta_j-\ell_j$ at $q_j$, where $\ell_j-1$ is the order of vanishing 
of $df$ at
$q_j$.
Moreover, 
the exact order of vanishing of $\bar{\sigma}$ is either $\beta_j-\ell_j$ or at least $\beta_j$.
\end{lemma}

We expect all $\ell_j$ to be $1$, and the tangent space to $\widetilde{W}$ to be
\[
	H^0\left(C, \Nc\left(-\sum \alpha_i p_i - \sum (\beta_j-1) q_j \right) \right)
\]
This sheaf has degree $2g-2+ \gamma$, and
Riemann--Roch gives $g-1+ \gamma$ for the expected dimension of $W$.

This estimate can be wrong for two reasons: 
\begin{enumerate}
\item The higher cohomology of 
$\Nc(-\sum \alpha_i p_i - \sum (\beta_j-1) q_j)$ might not vanish. 
\item The sheaf $\Nc(-\sum \alpha_i p_i - \sum (\beta_j-1) q_j)$
 could have torsion.
\end{enumerate}

The first issue will not arise if $\Nc$ is torsion free, and its degree 
is in the non-special range, 
which corresponds to the condition $\gamma \geq 1$. 

The second issue arises when the map $f$ has ramification
 (for example, if the image curve $C$ had a cusp). 
Nevertheless, this problem  was dealt with by Arbarello--Cornalba, 
and then Caporaso--Harris. They showed that for $f\from C \to \Sigma$ general,
the dimension of $H^0(\Nc/\Nc_{\text{tor}})$ is an upper bound on 
the dimension of the Kontsevich space. This assumes that
the evaluation map is birational (which we are assuming already, anyways), and that 
the characteristic is zero.

Their setup is the following. Let $\mathcal{C} \to B$ be a flat, smooth and proper 
family of curves over a smooth connected base $B$, and 
$f\from \mathcal{C} \to \Sigma \times B$ be a map over $B$.
To each fiber $f_b\from C_b \to \Sigma$, we associate
  the normal sheaf $\Nc_b$ for $f_b$. Then the 
\emph{Horiwaka map} sends a tangent vector at $b \in B$
 to the corresponding first order deformation 
of the map $f_b$:
\[
  k_b\from  T_bB \to H^0(C_b, \mathcal{N}_b)
\]
If the family $\mathcal{C} \to \Sigma\times B$ is nowhere isotrivial, this map is
injective, which gives an upper bound for the dimension of $B$.

\begin{lemma}[Arbarello--Cornalba, Caporaso--Harris]
\label{lemma:torsion}
If $\mathcal{C} \to S \times B$ is birational onto its image, then for $b \in B$
general
\[
  \im k_b \cap H^0(C_b, (\mathcal{N}_b)_{\text{tor}} )= 0
\]
\end{lemma}

We use this lemma to deal with the issues raised above, and prove 
\Cref{theorem:deftheory}.

\begin{proof}[Proof of \Cref{theorem:deftheory}]
Consider the following divisors on $C$:
\begin{align*}
D &= \sum \alpha_i p_i + \sum \beta_j q_j \\
D_0 &= \sum \ell_j q_j \\
D_1 &=D -D_0
\end{align*}
Note that, by an abuse of notation, $D = f^*(\Oc_\Sigma(D))$.

By \Cref{lemma:tangency-conditions}, the image of $T_{f}W \subset H^0(C,\Nc)$ on $H^0(C, \Nc/\Nc_{\text{tor}})$
is contained in $H^0(C, \Nc/\Nc_{\text{tor}}(-D_1))$. And by \Cref{lemma:torsion}, the dimension of the
latter is an upper bound on the dimension of $W$. That is,
\[
\dim W \leq h^0(C, \Nc/\Nc_{\text{tor}}(-D_1))
\]

We can describe very explicitly what the line bundle $\Nc/\Nc_{\text{tor}}(-D_1)$ is. 
The sheaf $\Nc_{\text{tor}}$ is supported on the ramification points of $f\from C \to \Sigma$. Let $\Gamma$ 
be the skyscraper sheaf of ramification away
from the $q_j$. Then
\[
\Nc_{\text{tor}} = \Gamma \oplus k^{\ell_i-1}_{q_j}
\]
And, as we know that
\[
c_1(\Nc)=K_C-f^*K_\Sigma
\]
we get
\begin{align*}
\Nc/\Nc_{\text{tor}} &= K_C\otimes f^*K_\Sigma^{-1} (- c_1(\Gamma) - \sum (\ell_i -1) q_j) \\
&= K_C(-f^*K_\Sigma - D_0 + \sum q_j - c_1(\Gamma))
\end{align*}
Hence,
\[
\Nc/\Nc_{\text{tor}}(-D_1) = K_C(-f^* K_\Sigma -D + \sum q_j - c_1(\Gamma))
\]
which is a sub sheaf of 
\[
\Lc =K_C(-f^* K_\Sigma -D + \sum q_j)
\]
whose degree is 
\[
\deg \Lc = 2g-2 -(K_\Sigma+D) \tau + \bar{b}
\]
Adding an effective divisor $B$ of degree $b-\bar{b}$, we get that 
$\deg \Lc(B) = 2g-2+\gamma$, and $\Nc/\Nc_{\text{tor}}(-D_1)$
is a subsheaf of $\Lc(B)$.

Hence, if $\gamma \geq 1$, we are in the non-special range, and
Riemann--Roch tells us
\begin{align*}
\dim W &\leq h^0(\Nc/\Nc_{\text{tor}}(-D_1) \\
&\leq h^0(\Lc(B)) = g-1 +\gamma
\end{align*}

Let us assume from now on that equality holds.

If $\gamma \geq 2$, then $\Gamma =0$ and $b=\bar{b}$. 
Indeed, if $r \in \Supp \Gamma$, or in $\Supp B$, then 
\[
\dim W \leq h^0(\Lc(B-r))
\]
which is still in the non-special range, and hence gives too small a dimension.

To be unramified along the $q_j$ we need more freedom---let us assume that 
$\gamma \geq 3$.
Then we can find a section of $\Lc$ vanishing to order exactly $1$ on $q_j$. Hence, the order
of vanishing of $\bar{\sigma} \in H^0(\Nc/\Nc_{\text{tor}})$ at $q_j$ is
$\beta_j-\ell_j+1$. By the \Cref{lemma:tangency-conditions}, this must be at least $\beta_j$, and hence 
$\ell_j=1$.
Hence $f$ is unramified and  $\Nc$ is torsion-free. Therefore,
$\Nc = K_C(-f^*K_\Sigma)$.

To check smoothness of  $f(C)$ along $D$, we only need show that pairs of points 
$q_j$ don't always
map to the same point. We ensure that by finding a section vanishing at $q_{j_1}$ but not at $q_{j_2}$. Again, $\gamma \geq 3$ is enough.

We can rule out tacnodes too. If $r_1,r_2$ map down to a tacnode,
we need to find a section vanishing at one but not the other, which we are guaranteed to find, again by 
Riemann--Roch.

Finally, let us rule out triple points. Suppose there is one,
and that $r_1,r_2,r_3 \in C$ are its preimages.
We would like to find a section vanishing 
at $r_1$ and $r_2$, but not $r_3$. To guarantee that we need 
$\Nc(-D+\sum q_j -r_1-r_2-r_3)$ 
to be non-special, which holds when $\gamma \geq 4$
\end{proof}

\bibliographystyle{alpha}
\bibliography{covers_of_genus_one_curves}

\end{document}